%% file: gln.tex
\documentclass[final]{amsart}
\input{header.tex}

\usepackage[utf8]{inputenc}
\usepackage{cleveref}
\usepackage{verbatim}
\newcommand{\Part}{\operatorname{Part}}
\newcommand{\red}{\operatorname{red}}
\newcommand{\cyc}{\operatorname{cyc}}
\newcommand{\loc}{\operatorname{loc}}
\newcommand{\ft}{\mathrm{ft}}
\newcommand{\Rbox}{R^\square}
\newcommand{\Rbarbox}{\bar{R}\,^\square}
\newcommand{\sgn}{\operatorname{sgn}}
\newcommand{\std}{\operatorname{std}}
\newcommand{\length}{\operatorname{length}}

\newcommand{\Bip}{\mathrm{Bip}}
\newcommand{\scs}{\operatorname{scs}}
\newcommand{\aut}{\operatorname{aut}}

\title{The Breuil--M\'{e}zard conjecture when $l\neq p$.}  
\author{Jack Shotton}

\begin{document}
\begin{abstract}
  Let $l$ and $p$ be primes, let $F/\QQ_p$ be a finite extension with absolute Galois group $G_F$,
  let $\FF$ be a finite field of characteristic $l$, and let \[\bar{\rho} : G_F\rarrow GL_n(\FF)\]
  be a continuous representation.  Let $R^\square(\rhobar)$ be the universal framed deformation ring
  for $\rhobar$.  If $l = p$, then the Breuil--M\'{e}zard conjecture (as formulated in
  \cite{JMJ:9091284}) relates the mod $l$ reduction of certain cycles in $R^\square(\rhobar)$ to the
  mod $l$ reduction of certain representations of $GL_n(\Oc_F)$.  We state an analogue of the
  Breuil--M\'{e}zard conjecture when $l \neq p$, and prove it whenever $l > 2$ using automorphy
  lifting theorems.  We give a local proof when $l$ is ``quasi-banal'' for $F$ and $\rhobar$ is
  tamely ramified.  We also analyse the reduction modulo $l$ of the types $\sigma(\tau)$ defined by
  Schneider and Zink \cite{SchneiderZink1999-KTypesTemperedComponentsGeneralLinear}.
\end{abstract}
\maketitle
\tableofcontents
\section{Introduction}

When $F$ is a $p$-adic field and $\rhobar$ is an $n$-dimensional mod $p$ representation of its
absolute Galois group $G_F$, the Breuil--M\'{e}zard conjecture relates singularities in the
deformation ring of $\rhobar$ to the mod $p$ representation theory of $GL_n(\Oc_F)$.  It was first
formulated, for $F = \QQ_p$ and $n = 2$, in \cite{breuil2002}, and (mostly) proved in this case in
\cite{Kisin2009-FontaineMazur}.  In full generality, the conjecture is formulated in
\cite{JMJ:9091284} but is not known in any cases with $n > 2$.  In this article we prove an analogue
of the Breuil--M\'{e}zard conjecture for mod $l$ representations of $G_F$ and $GL_n(\Oc_F)$, with
$F$ a $p$-adic field and $l$ an odd prime \emph{distinct} from $p$.

We give a precise statement, after setting up a little notation.  Let $F$ be a finite extension
of $\QQ_p$ with ring of integers $\Oc_F$, residue field $k_F$ of order $q$, and absolute Galois
group $G_F$, and let $l$ be a prime distinct from $p$.  Let $E$ be a finite extension of $\QQ_l$,
with ring of integers $\Oc$, uniformiser $\lambda$, and residue field $\FF$.  Let
\[ \rhobar : G_F \rarrow GL_n(\FF)\] be a continuous representation.  Then there is a universal
framed deformation ring $R^\square(\rhobar)$ parameterizing lifts of $\rhobar$.  Our
main result, stated below, relates congruences between irreducible components of $\Spec
R^\square(\rhobar)$ to congruences between representations of $GL_n(\Oc_F)$.

It is known that $\Spec R^\square(\rhobar)$ is flat and equidimensional of relative dimension $n^2$
over $\Spec\Oc$ -- see Theorem~\ref{thm:complete-intersection}.  Let $\Zc(R^\square(\rhobar))$ be
the free abelian group on the irreducible components\footnote{We suppose that $E$ is ``sufficiently
  large'' and in particular that all of these are geometrically irreducible.} of $\Spec
R^\square(\rhobar)$; similarly we have the group $\Zc(\Rbarbox(\rhobar))$ where $\Rbarbox(\rhobar) =
R^\square(\rhobar) \otimes_{\Oc}\FF$.  There is a natural homomorphism
\[\red : \Zc(R^\square(\rhobar))\longrightarrow\Zc(\bar{R}\,^\square(\rhobar))\]
taking an irreducible component of $\Spec(R^\square(\rhobar))$ to its intersection with the special
fibre (counted with multiplicities).

An \textbf{inertial type} is an isomorphism class of continuous representation
$\tau:I_F\rarrow GL_n(\bar{E})$ that may be extended to $G_F$.  If $\tau$ is an inertial type, then there
is a quotient $R^\square(\rhobar,\tau)$ of $R^\square(\rhobar)$ that (roughly speaking)
parameterizes representations of type $\tau$; that is, whose restriction to $I_F$ is isomorphic to
$\tau$.  Then $\Spec R^\square(\rhobar,\tau)$ is a union of
irreducible components of $\Spec R^\square(\rhobar)$.

Let $R_E(GL_n(\Oc_F)$ (resp. $R_{\FF}(GL_n(\Oc_F))$) be the Grothendieck group of finite dimensional
smooth representations of $GL_n(\Oc_F)$ over $E$ (resp. $\FF$), and let
\[\red : R_E(GL_n(\Oc_F)) \rarrow R_{\FF}(GL_n(\Oc_F))\] be the surjective map given by reducing a representation
modulo $l$.  In section~\ref{sec:breuil-mezard-lneq} we define a
homomorphism
\[\cyc : R_E(GL_n(\Oc_F)) \rarrow \Zc(R^\square(\rhobar))\] 
by the formula
\[\cyc(\theta) = \sum_{\tau}m(\theta^{\vee},\tau)Z(R^\square(\rhobar,\tau))\]
where the sum is over all inertial types, $Z(R^\square(\rhobar,\tau))$ is the formal sum of the
irreducible components of $\Spec R^\square(\rhobar,\tau)$, and $m(\theta^\vee,\tau)$ is the
multiplicity of $\theta^\vee$ in any generic irreducible admissible representation $\pi$ such that
$r_l(\pi)|_{I_F} \cong \tau$.

\begin{theorem*} Suppose that $l > 2$.  There is a unique map $\bar{\cyc}$ making the following
  diagram commute:
 \begin{equation}\label{eq:BM-CD}
\begin{CD}
R_E(GL_n(\Oc_F)) @>{\cyc}>> \Zc(R^\square(\bar{\rho})) \\
@V{\red}VV  @V{\red}VV \\
R_{\FF}(GL_n(\Oc_F)) @>{\bar{\cyc}}>> \Zc(\Rbarbox(\rhobar)).
\end{CD}
\end{equation} 
\end{theorem*}

This is Theorem~\ref{thm:BM-Gln} below. We conjecture (Conjecture~\ref{conj:BM}) that it is also
true for $l = 2$.  The content of the theorem is that congruences between representations of
$GL_n(\Oc_F)$ force congruences between irreducible components of $R^\square(\rhobar)$.

The image of the map $\cyc$ is precisely the $\ZZ$-span of the cycles $Z(R^\square(\rhobar,\tau))$.
We write down explicit elements $r(\tau)$ of the Grothendieck group $R_E(GL_n(\Oc_F))$ in terms of
inverse Kostka numbers, such that $\cyc(r(\tau)) = Z(R^\square(\rhobar,\tau))$.  Then we obtain (Corollary~\ref{cor:cyc-inverse}):
\begin{corollary*}
  For each inertial type $\tau$, 
  \[\red(Z(\Rbox(\rhobar,\tau))) = \bar{\cyc}(\red(r(\tau))).\]
\end{corollary*}
Thus knowledge of $\bar{\cyc}$ determines the generic multiplicities of the irreducible components
of $\Rbox(\rhobar,\tau) \otimes_\Oc \FF$.  Note that $r(\tau)$ is in general a virtual element of
the Grothendieck group; for instance, if $\tau$ is the non-split two-dimensional unipotent type, then
\[r(\tau) = \St - \mathbbm{1}\] where $\St$ is the Steinberg representation of $GL_2(k_F)$ and
$\mathbbm{1}$ is the trivial representation.  We do not attempt to describe $\bar{\cyc}$ here, but
hope to return to this question in future work.  The map $\cyc$ also appears in the $l = p$
situation when working with potentially semistable (rather than potentially crystalline) deformation
rings.

Our proof of the main theorem is `global', making use of the methods of
\cite{GeeKisin2013-BreuilMezardBarsottiTate} and \cite{JMJ:9091284}. We use the
Taylor--Wiles--Kisin patching method to produce an exact functor \[\theta \longmapsto
H_\infty(\theta)\] from the category of finitely generated $\Oc$-modules with a smooth
$GL_n(\Oc_F)$-action to the category of finitely generated $R^\square(\rhobar)$-modules, such that
the support of $H_\infty(\theta)$ -- counted with multiplicity -- is $\cyc(\theta)$.  As this
functor is compatible with reduction modulo $l$, we can deduce the theorem.

We can also give local proofs of (variants of) the theorem in some special cases.  In
\cite{MR3554238}, we studied the case $n=2$ and $l > 2$; explicitly calculated the rings
$R^\square(\rhobar,\tau)$ in this case and gave a local proof of the theorem.\footnote{Strictly
  speaking, this proof works with $R_E(GL_n(\Oc_F))$ replaced by the subgroup generated by
  `$K$-types'.} In section~\ref{sec:towards-local-proof} we prove the theorem (with
$R_E(GL_n(\Oc_F))$ replaced by a certain subgroup of $R_E(GL_n(k_F))$) in the case that $\rhobar$ is
tamely ramified and $l$ is \emph{quasi-banal}; that is, $l > n$ and $l|q - 1$.  The method is to
first observe that there is a scheme $\Xf$ of finite type over $\Spec \Oc$ --- it is the moduli
space of pairs of invertible matrices $\Sigma$ and $\Phi$ satisfying $\Phi \Sigma\Phi^{-1} =
\Sigma^q$ --- such that the $\Spf R^\square(\rhobar)$, for varying $\rhobar$, may all be obtained as
the completions of $\Xf$ at closed points.  This allows us to reduce the theorem to the case in
which $\rhobar$ is ``distinguished''; this is a certain genericity condition.  When $\rhobar$ is
distinguished we can compute all of the $R^\square(\rhobar,\tau)$ by elementary arguments.  As we
also have a good understanding of the representation theory of $GL_n(k_F)$ in the quasi-banal case,
we can deduce the theorem.  It seems likely that these methods could be pushed further; we have just
dealt with the simplest interesting case for general $n$.

Kisin \cite{Kisin2009-FontaineMazur} proved most cases of the original Breuil--M\'{e}zard 
conjecture, simultaneously with proving most cases of the Fontaine--Mazur conjecture for $GL_2/\QQ$.  The point is
that the information about the special fibres of local deformation rings provided by the
Breuil--M\'{e}zard conjecture is what is needed to prove automorphy lifting theorems in general
weight, using the Taylor--Wiles method as modified by Kisin in
\cite{Kisin2009-ModuliFFGSandModularity}.  The methods of
\cite{GeeKisin2013-BreuilMezardBarsottiTate}, \cite{JMJ:9091284} and this article can be viewed as
implementing this idea ``in reverse'', using known automorphy lifting theorems (or, in the case of
\cite{JMJ:9091284}, assuming automorphy lifting theorems) to deduce the Breuil--M\'{e}zard
conjecture.  We note, however, that no cases of the Breuil--M\'{e}zard conjecture are known when $l
= p$ and $n > 2$, the question being bound up with the weight part of Serre's conjecture and the
Fontaine--Mazur conjecture.

The other motivation behind our theorem is the ``Ihara avoidance'' method of
\cite{Taylor2008-AutomorphyII}, which arose in the $l \neq p$ setting.  Taylor's idea is to compare
the special fibres of very specific $R^\square(\rhobar,\tau)$, and combine this with the
Taylor--Wiles--Kisin method to prove non-minimal automorphy lifting theorems (i.e. automorphy
lifting theorems incorporating a change of level).  The similarity to Kisin's use of the
Breuil--M\'{e}zard conjecture to prove automorphy lifting theorems with a change of weight is clear;
thus it is natural to try to study local deformation rings when $l \neq p$ from the point of view of
the Breuil--M\'{e}zard conjecture.  Our proof actually \emph{depends} on Taylor's results, as it
makes crucial use of non-minimal automorphy lifting theorems.  We explain this example in detail in
section~\ref{sec:ihara-avoidance}.

In \cite{1209.5205}, Pa\v{s}k\={u}nas gives a purely local proof of most cases of the
Breuil--M\'{e}zard conjecture, relying on the $p$-adic Langlands correspondence (on which Kisin's
proof also depends).  He shows\footnote{When $\rhobar$ is `generic' --- the proof in the non-generic
  case is a little different.} that the universal deformation ring $R(\rhobar)$ of the residual
representation $\rhobar$ can be realised as the endomorphism ring of the projective envelope
$\tilde{P}$, in a suitable category, of the representation $\pi$ of $GL_2(\QQ_p)$ associated to
$\rhobar$ by the mod-$p$ Langlands correspondence.  Then the functor
\[\theta \mapsto \Hom_{\Oc[[GL_2(\ZZ_p)]]}(\tilde{P},\theta^\vee)^\vee\]
plays the same role in \cite{1209.5205} that the functor $\theta \mapsto H_\infty(\theta)$ does in
global proofs via patching.  Since the writing of this paper, Helm and Moss \cite{1610.03277} have
constructed the local Langlands correspondence in families conjectured by Emerton and Helm
\cite{1104.0321}.  It may be possible to derive the results of this paper from their result, by
methods analogous to those of \cite{1209.5205}, and we hope to return to this in the future.

Section~\ref{sec:k-types} has a rather different focus.  Certain of the representations of
$GL_n(\Oc_F)$ are more interesting than the others; these are the $K$-types.  For every inertial
type $\tau$ there is a corresponding $K$-type $\sigma(\tau)$, essentially constructed by Schneider
and Zink \cite{SchneiderZink1999-KTypesTemperedComponentsGeneralLinear}.  These representations have
an interesting `Galois theoretic' interpretation --- see Theorem~\ref{thm:K-types} below.  We
determine the multiplicities $m(\sigma(\tau),\tau')$ when $\tau$ and $\tau'$ are inertial types; the
answer is given in terms of certain Kostka numbers.  We also explain how to determine the mod $l$
reduction of the representations $\sigma(\tau)$ in terms of the mod $l$ reduction of representations
of certain general linear groups; in order to do this, we must work with a variant of the
construction of \cite{SchneiderZink1999-KTypesTemperedComponentsGeneralLinear}.
Sections~\ref{sec:symmetric-groups} to~\ref{sec:finite-gener-line} are used in
section~\ref{sec:towards-local-proof}, but otherwise the only place that section~\ref{sec:k-types}
is used in the rest of the paper is to derive the multiplicity formula of
Proposition~\ref{prop:type-mults}; in particular, sections~\ref{sec:k-types}
and~\ref{sec:towards-local-proof} are not required for the proof of Theorem~\ref{thm:BM-Gln}.

We briefly sketch the contents of the different sections.  Section~\ref{sec:deformation-rings} is
preliminary, containing the basic definitions of the relevant local deformation rings.
Theorem~\ref{thm:complete-intersection} of this section, which is due to David Helm, gives some of
their basic geometric properties and is probably of independent interest.  In
section~\ref{sec:cycles} we cover some commutative algebra to do with multiplicities and cycles.
Section~\ref{sec:inertial-types} deals with the stratification of the Bernstein centre by inertial
types and the associated fixed type deformation rings.  We also, in Theorem~\ref{thm:K-types},
introduce the $K$-types of \cite{SchneiderZink1999-KTypesTemperedComponentsGeneralLinear} and state
their formal properties.  Section~\ref{sec:breuil-mezard-lneq} contains the statement of the main
theorem and its proof given the formal properties of the globally constructed patching functor.  We
also state a formula for the multiplicity of a $K$-type in a generic smooth admissible
representation of $GL_n(F)$.  Section~\ref{sec:global-proof} constructs the patching functor needed
to prove the main theorem.  Sections~\ref{sec:autom-forms-galo} and~\ref{sec:galo-repr} setup the
necessary spaces of automorphic forms and associated Galois representations,
section~\ref{sec:real-local-repr} follows the appendix of~\cite{JMJ:9091284} to extend a local
Galois representation $\rhobar$ to a global representation arising automorphically, and section~5.4
carries out the patching argument.  Section~\ref{sec:k-types} contains the proof, via
Bushnell--Kutzko theory, of the multiplicity formula for $K$-types and also a coarse description of
their reduction modulo $l$.  Section~\ref{sec:towards-local-proof} contains a local proof of the
main theorem in a special case.

\subsection{Acknowledgements}
\label{sec:acknowledgements}

This work is part of the author's Imperial College PhD thesis.  I would like to thank my supervisor,
Toby Gee, for suggesting this problem and the approach via patching.  I would also like to thank
Matthew Emerton, David Helm, Vincent S\'{e}cherre, Sug Woo Shin and Shaun Stevens for helpful
comments, conversations or correspondence.

This research was supported by the Engineering and Physical Sciences Research Council, and the
Philip Leverhulme Trust. Part of it was conducted during a visit to the University of Chicago
sponsored by the Cecil King Foundation and the London Mathematical Society.

\section{Deformation rings}
\label{sec:deformation-rings}
\subsection{Definitions}\label{sec:def-defs}
Let $F / \QQ_p$ be a finite extension with ring of integers $\Oc_F$, and residue field $k_F$ of
order $q$.  Let $\bar{F}$ be an algebraic closure of $F$, $\bar{k}_F$ the induced algebraic closure
of $k_F$, and $G_F = \Gal(\bar{F}/F)$.  Let $I_F \normal G_F$ and $P_F \normal G_F$ be,
respectively, the inertia and wild inertia subgroups of $G_F$.  We have canonical
isomorphisms \[G_F/I_F = \hat{\ZZ}\] and \[I_F/P_F = \varprojlim_{k/k_F} k^\times
  \cong \prod_{l \neq p} \ZZ_l(1),\] where the limit is over finite extensions of $k_F$ contained in
  $\bar{k}_F$ and the transition maps are the norm maps.  Let $\phi \in G_F/I_F$ be arithmetic
  Frobenius, and denote also by $\phi$ a choice of lift to $G_F$.  Let $\sigma$ be a topological
  generator for $I_F/P_F$; this choice is equivalent to choosing a norm-compatible system of
  generators for the units in each finite extension $k$ of $k_F$, or to choosing a basis for each
  $\ZZ_l(1)$.  Then, via these choices, $G_F/P_F$ is isomorphic to the profinite completion of
\begin{equation}\label{eqn:presentation}\langle \phi, \sigma | \phi \sigma \phi^{-1} = \sigma^q
  \rangle = \ZZ \ltimes \ZZ[\frac{1}{p}].\end{equation}

Let $E/\QQ_l$ be a finite extension with ring of integers $\Oc$, uniformiser $\lambda$ and residue
field $\FF$.  Let $\Cc_\Oc$ denote the category of
artinian local $\Oc$-algebras with residue field $\FF$, and $\Cc_\Oc^\wedge$ the category of
complete noetherian local $\Oc$-algebras with residue field $\FF$.  If $A$ is an object of $\Cc_\Oc$
or $\Cc_\Oc^\wedge$, let $\mf_A$ be its maximal ideal.

Suppose that $\bar{M}$ is an $n$-dimensional $\FF$-vector space and that $\bar{\rho} : G_F \rarrow
\Aut_{\FF}(\bar{M})$ is a continuous homomorphism.  Let $(\bar{e}_i)_{i=1}^n$ be a basis for $\bar{M}$, so that
$\bar{\rho}$ gives a map $\bar{\rho} : G_F \rarrow GL_n(\FF)$.
  
Define two functors
\[D(\rhobar), D^\square(\rhobar) : \Cc_\Oc \rarrow \mathbf{Set}\] as follows:
\begin{itemize}
\item $D(\rhobar)(A)$ is the set of equivalence classes of $(M,\rho,\iota)$ where: $M$ is a free rank~$n$
  $A$-module, $\rho : G_F \rarrow \Aut_A(M)$ is a continuous homomorphism,
  and \[\iota:M\otimes_A\FF\isomto\bar{M}\] is an isomorphism commuting with the actions of $G_F$;
  
\item $D^\square(\rhobar)(A)$ is the set of equivalence classes of $(M,\rho, (e_i)_{i=1}^n)$ where: $M$ is
  a free rank $n$ $A$-module, $\rho : G_F \rarrow \Aut_A(M)$ is a continuous homomorphism, and
  $(e_i)_{i=1}^n$ is a basis of $M$ such that the isomorphism $\iota:M \otimes_A \FF
  \isomto \bar{M}$ taking $e_i \otimes 1$ to $\bar{e}_i$ commutes with the actions of $G_F$.
\end{itemize}
In the first case, $(M,\rho,\iota)$ and $(M',\rho',\iota')$ are equivalent if there is an
isomorphism $\alpha:M \rarrow M'$, commuting with the actions of $G_F$, such that $\iota =
\iota'\circ \alpha$; in the second case, $(M,\rho,(e_i)_i)$ and $(M',\rho',(e'_i)_i)$ are equivalent
if the isomorphism of $A$-modules $M \rarrow M'$ defined by $e_i \mapsto e_i'$ commutes with the
actions of $G_F$.  There is a natural transformation of functors $D^\square(\rhobar) \rarrow
D(\rhobar)$ given by forgetting the basis.

Alternatively, when $\rhobar$ is regarded as a homomorphism to $GL_n(\FF)$, we have the equivalent definitions
\[D^\square(\rhobar)(A) = \{\text{continuous $\rho : G_F \rarrow GL_n(A)$ lifting $\bar{\rho}$}\}\]
and \[D(\rhobar)(A) = \{\text{continuous $\rho : G_F \rarrow GL_n(A)$
  lifting $\bar{\rho}$}\}/\text{$1 + M_n(\mf_A)$}\] 
where the action of the group $1 + M_n(\mf_A)$ is by conjugation.
  
The functor $D(\rhobar)$ is not usually pro-representable, but the functor $D^\square(\rhobar)$
always is (see, for example, \cite{Kisin2009-ModuliFFGSandModularity}~(2.3.4)):
\begin{definition} The \emph{universal lifting ring} (or universal framed deformation ring) of
  $\rhobar$ is the object $R^\square(\rhobar)$ of $\Cc^\wedge_{\Oc}$ that pro-represents the functor
  $D^\square(\rhobar)$.  The universal lift is denoted $\rho^\square : G_F \rarrow
  GL_n(R^\square(\rhobar))$.
\end{definition}

\subsection{Geometry of $R^\square(\rhobar)$}

Recall the following calculation from \cite{BLGGT2014-PotentialAutomorphy}~\S1.2:

\begin{lemma} \label{lem:dimension}
  The scheme $\Spec R^\square(\rhobar)[1/l]$ is generically formally smooth of dimension $n^2$. \qed
\end{lemma}

Let $I_F \rarrow I_F/\tilde{P}_F$ be the maximal pro-$l$ quotient of $I_F$.  The next lemma enables
us to reduce to the case where the residual representation is trivial on $\tilde{P}_F$.  Suppose
that $\theta$ is an irreducible $\FF$-representation of $\tilde{P}_F$; write $[\theta]$ for the
orbit of the isomorphism class of $\theta$ under conjugation by $G_F$.  By
\cite{ClozelHarrisTaylor2008-Automorphy} Lemma~2.4.11, $\theta$ may be extended to an
$\Oc$-representation $\tilde{\theta}$ of $G_\theta$ where $G_\theta$ is the open subgroup $\{g\in
G_F : g\theta g^{-1} \cong \theta\}$ of $G_F$.  For each irreducible representation $\theta$ of
$\tilde{P}_F$, we pick such a $\tilde{\theta}$.  If $M$ is a finite-dimensional $\FF$-vector space
with a continuous action of $G_F$, then define
\[M_\theta = \Hom_{\tilde{P}_F}(\tilde{\theta}, M).\] This has a natural continuous
action of $G_\theta$ given by $(gf)(v) = gf(g^{-1}v)$; the subgroup $\tilde{P}_F$ of
$G_\theta$ acts trivially.  If $\rhobar : G_F \rarrow GL_n(\FF)$ is continuous and corresponds to
some choice of basis for $M$, then choose a basis for each $M_\theta$ to obtain a continuous homomorphism 
$\rhobar_\theta: G_\theta \rarrow GL_n(\FF)$.  

\begin{lemma} \label{lem:tame-reduction} (Tame reduction) 
    If $R^\square(\rhobar_\theta)$ is the universal framed deformation ring for the
    representation $\rhobar_\theta$ of $G_\theta / \tilde{P}_F$, then
    \[ R^\square(\rhobar) \cong \left(\widehat{\bigotimes}_{[\theta]} R^\square(\rhobar_\theta)
    \right)[[X_1, \ldots, X_{n^2 - \sum n^2_\theta}]]\] where $n_\theta = \dim \rhobar_\theta$.
\end{lemma}

\begin{proof}
  This is a modification, due to Choi~\cite{Choi2009-LocalDeformationRings}, of \cite{ClozelHarrisTaylor2008-Automorphy}
  Corollary~2.4.13 to take into account the framings.  See \cite{MR3554238} Lemma~2.3.
\end{proof}

The next result is due to David Helm, and will appear in a forthcoming paper of his.
I thank him for allowing me to include the proof here.

\begin{definition}
  Suppose that $R$ is a ring.  Let $\Mc(n,q)_R$ be the moduli space (over $\Spec R$) of pairs of
  matrices $\Sigma, \Phi \in GL_{n,R} \times_{\Spec R} GL_{n,R}$ such that \[\Phi \Sigma \Phi^{-1} =
  \Sigma^q.\] It is the closed subscheme of $GL_{n,R} \times_{\Spec R} GL_{n,R}$ cut out by the
  $n^2$ matrix coefficients of the above equation.  Denote by $\pi_\Sigma$ the morphism
  \begin{align*}\pi_\Sigma : \Mc(n,q)_R &\longrightarrow GL_{n,R} \\
    (\Sigma,\Phi) & \mapsto \Sigma.
  \end{align*}
\end{definition}

\begin{theorem}\label{thm:complete-intersection} The scheme $\Spec R^\square(\rhobar)$ is a 
  reduced complete intersection, flat and equidimensional of relative dimension $n^2$ over
  $\Spec \Oc$.
\end{theorem}
\begin{proof} 
  Suppose that $k$ is an algebraically closed field of characteristic distinct from $p$, and
  consider $\Mc(n,q)_k$.  Let $\Sigma_0$ be a closed point in the image of $\pi_\Sigma$, let $Z_0$
  be the centraliser of $\Sigma_0$ in $GL_{n,k}$ (a closed subgroup scheme of $GL_{n,k}$) and let
  $C_0$ be the conjugacy class of $\Sigma_0$ in $GL_{n,k}$, a locally closed subscheme of $GL_{n,k}$
  isomorphic to $GL_{n,k}/Z_0$.  Then $\pi_\Sigma^{-1}(\Sigma_0)$ is (by right multiplication on
  $\Phi$) a $Z_0$-torsor.  Thus the preimage $\pi_\Sigma^{-1}(C_0)$ in $\Mc(n,q)_k$ has
  dimension \[\dim C_0 + \dim Z_0 = n^2 - \dim Z_0 + \dim Z_0 = n^2. \] Since the eigenvalues of any
  $\Sigma$ in the image of $\pi_{\Sigma}$ must be $(q^{n!} - 1)$th roots of unity, the number of
  conjugacy classes $C_0$ of matrices in the image of $\pi_\Sigma$ is finite.\footnote{Here we use
    that $q > 1$.  It is unknown whether the moduli space of pairs of commuting matrices over $\CC$
    is Cohen--Macaulay (or even reduced!), although this is conjectured to be the case (see
    \cite{MR1731818}).}  Therefore \[\dim \Mc(n,q)_k = n^2.\]
 
  Now let $R = \Oc$.  We see that $\Mc(n,q)_\Oc \rarrow \Spec \Oc$ is equidimensional of
  dimension~$n^2$.  But the smooth scheme $GL_{n,\Oc}\times_\Oc GL_{n,\Oc}$ has relative dimension
  $2n^2$ over $\Spec \Oc$ and $\Mc(n,q)_\Oc$ is a closed subscheme cut out by $n^2$ equations; it
  follows that $\Mc(n,q)_\Oc$ is a local complete intersection.  In particular, it is a
  Cohen--Macaulay scheme.  As its fibres over the regular local ring $\Spec \Oc$ are of the same
  dimension, $n^2$, it is flat over $\Spec \Oc$.
 
  Now, by Lemma~\ref{lem:tame-reduction}, the assertions of the theorem may be reduced to the case
  in which $\rhobar$ is tamely ramified (using Lemma~3.3 of \cite{MR2827723} to propagate flatness,
  reducedness, and dimension from objects of $\Cc^\wedge_{\Oc}$ to their completed tensor products).
  In this case, any lift of $\rhobar$ to an object of $\Cc_{\Oc}$ is also tamely ramified, as $P_F$
  is pro-$p$.  Our choice of topological generators $\phi$ and $\sigma$ for $G_F/P_F$ satisfying the
  equation $\phi \sigma \phi^{-1} = \sigma^{q}$ provides a closed point of $\Mc(n,q)_\Oc$
  corresponding to $\rhobar$ and identifies $R^\square(\rhobar)$ with the completion of the local
  ring of $\Mc(n,q)_\Oc$ at this point (to see this, compare the $A$-valued points for $A$ an object
  of $\Cc_\Oc$).  Therefore, by the corresponding facts for $\Mc(n,q)_\Oc$, we have shown that
  $R^\square(\rhobar)$ is a complete intersection and is flat over $\Oc$.  It is reduced since it is
  generically reduced (by Lemma~\ref{lem:dimension} and the fact that it is $\Oc$-flat) and
  Cohen--Macaulay.
\end{proof}

We extract the following consequence of the proof:

\begin{proposition} \label{prop:cpts} If $k$ is a field of characteristic distinct from $p$ that
  contains all of the $(q^{n!}-1)$th roots of unity, and $C$ is a conjugacy class in $GL_n(\bar{k})$
  that is stable under the $q$th power map, then the Zariski closure in $\Mc(n,q)_k$ of
  $\pi_\Sigma^{-1}(C)$ is an absolutely irreducible component of $\Mc(n,q)_k$, denoted
  $\Mc(n,q,\Sigma \sim C)_k$.  Every irreducible component of $\Mc(n,q)_k$ is of this form.
\end{proposition}
\begin{proof} As $C$ is stable under the $q$th power map and $k$ contains the $(q^{n!}-1)$th roots of
  unity, $C$ contains a $k$-point.  Then $C$ is absolutely irreducible and the fibres of
  $\pi_{\Sigma}$ above points of $C$ are all absolutely irreducible of dimension $\dim \Mc(n,q)_k -
  \dim C$.  Therefore the closure of $\pi_{\Sigma}^{-1}(C)$ is absolutely irreducible of the same
  dimension as $\dim \Mc(n,q)_k$, and is therefore an absolutely irreducible component.

  As every point of $\Mc(n,q)_k$ is in $\pi_\Sigma^{-1}(C)$ for some $C$, we obtain the final statement.
\end{proof}

\subsection{Cycles}
\label{sec:cycles}

Suppose that $X$ is a noetherian scheme and that $\Fc$ is a coherent sheaf on $X$.  Let $Y$ be the
scheme-theoretic support of $\Fc$, and let $d \geq \dim Y$.  Let $\Zc^d(X)$ be the free abelian
group on the $d$-dimensional points of $X$; elements of $\Zc^d(X)$ are called $d$-dimensional
cycles.  If $\af \in X$ is a point of dimension $d$ write $[\af]$ for the corresponding element of
$\Zc^d(X)$ and define the multiplicity $e(\Fc,\af)$ to be the length of $\Fc_{\af}$ as an
$\Oc_{Y,\af}$-module (this is zero if $\af \not \in Y$).

\begin{definition} The cycle $Z^d(\Fc)$ associated to $\Fc$ is the element
\[\sum_{\af}e(\Fc,\af) [\af] \in \Zc^d(X).\]
\end{definition}

If $X = \Spec A$ is affine and $\Fc = \widetilde{M}$ is the coherent sheaf associate to a finitely
generated $A$-module $M$, then we will write $Z^d(M)$ for $Z^d(\Fc)$.  If $X$ is equidimensional of
dimension $d$, then we will usually drop $d$ from the notation, so that $\Zc(X) = \Zc^d(X)$, $Z(\Fc)
= Z^d(\Fc)$ etc.

If $i:X\rarrow X'$ is a closed immersion of $X$ in a noetherian scheme $X'$, then there is a natural
inclusion $i_*:\Zc^d(X) \rarrow \Zc^d(X')$ for each $d$.  For a coherent sheaf $\Fc$ on $X$ whose
support has dimension at most $d$, we then have \[i_*(Z^d(\Fc)) = Z^d(i_*(\Fc)).\]  We will often use
this compatibility without comment.

If $X$ is a noetherian scheme of dimension $d$ and $X \rarrow \Spec \Oc$ is a flat morphism, then let
\[j : \bar{X} = X \times_{\Spec\Oc}\Spec\FF \rarrow X\] be the inclusion of the special fibre and
denote by $\red$ the reduction map
\[\red :  \Zc(X) \rarrow \Zc(\bar{X})\]
which takes a $d$-dimensional point $\af$ with closure $Y$ to the cycle $Z^{d-1}(j^*\Oc_{Y})$.
The following is a special case of \cite{JMJ:9091284} Proposition~2.2.13:
\begin{lemma}\label{lem:cycle-reduction}
  In the above situation, if $\Fc$ is a coherent sheaf on $X$ such that multiplication by $\lambda$
  is injective on $\Fc$, then 
  \[\pushQED{\qed}\red(Z^d(\Fc)) = Z^d(j^*(\Fc)).\qedhere \popQED\] 
\end{lemma}

If $f : X \rarrow Y$ is a flat morphism of noetherian schemes, with $X$ and $Y$ equidimensional of
dimensions $d$ and $e$ respectively, then we define a map \[f^* : \Zc^e(Y) \rarrow \Zc^d(X)\] by
taking a point $\af \in Y$ with closure $Z$ of dimension $e$ to the cycle \[Z^d(f^*\Oc_Z) \in
\Zc^d(X).\]

\begin{lemma}\label{lem:flat-pullback}
  In the above situation, if $\Fc$ is a coherent sheaf on $Y$ then 
\[f^*(Z^d(\Fc)) = Z^d(f^*(\Fc)).\]
\end{lemma}

\begin{proof}
  We may suppose that $X = \Spec S$ and $Y = \Spec R$ for noetherian rings $R$ and $S$, so that $f$
  induces a flat map $f^* : R \rarrow S$, and $\Fc = \widetilde{M}$ for a finitely generated
  $R$-module $M$.  If $\bfrak$ is a minimal prime of $S$ and $\af= \bfrak \cap R$, then $\af$ is a
  minimal prime of $R$ (by the going down property of flat morphisms) and we must show:
   \[\length_{R_\af}(M_\af)\length_{S_\bfrak}((R/\af
   \otimes_R S)_\bfrak)=\length_{S_\bfrak}((M\otimes_R S)_{\bfrak}).\] Replacing $R$ by $R_\af$, $S$ by $S_\bfrak$, and $M$ by $M_\af$, we may
   assume that $R, S$ are local and artinian and that $f$ is a local map of local rings, in which
   case we must show that
 \[\length_S(M \otimes_R S) = \length_R(M)\length_S(R/\af \otimes_R S),\]
 which is true as $S$ is flat over $R$ and $M$ has a finite composition series whose factors are all
 isomorphic to $R/\af$.
\end{proof}

Recall from \cite{MR2827723} lemma~3.3 that an object $R$ of $\Cc_{\Oc}^\wedge$ is
\emph{geometrically integral} (resp. \emph{geometrically irreducible}) if, for every finite
extension $E'/E$, $\Spec R\otimes \Oc_{E'}$ is integral (resp. irreducible).  Since $\FF$ and $E$
are perfect, geometrically integral is equivalent to reduced and geometrically irreducible.

\begin{lemma} \label{lem:product}
  Suppose that $R$ and $S$ objects of $\Cc_{\Oc}^\wedge$ that are either
  \begin{enumerate}
  \item flat over $\Oc$; or
  \item $\FF$-algebras,
  \end{enumerate}
  and that $R$ and $S$ are equidimensional of dimensions $d$ and $e$ respectively.  Suppose that
  every minimal prime $\pf$ of $R$ has the property that $R/\pf$ is geometrically integral, and that
  the same is true for $S$.  Then
  \[\Zc(R \hat{\otimes} S) = \Zc(R) \otimes \Zc(S).\]
\end{lemma}
\begin{proof}
  In case 1, by \cite{MR2827723} lemma~3.3 part~5, every minimal prime $\pf$ of $R \hat{\otimes} S$ is of the
  form $(\qf_1 \hat{\otimes} S + R\hat{\otimes} \qf_2)$ for uniquely determined minimal primes
  $\qf_1$ and $\qf_2$ of $R$ and $S$.  By \cite{MR2827723} lemma~3.3 part 2,
  $(R\hat{\otimes} S)/\pf$ has dimension $d + e - 1$, so that $R \hat{\otimes} S$ is equidimensional
  of dimension $d + e - 1$.  The map taking $[\pf]$ to $[\qf_1] \otimes [\qf_2]$ is the required
  isomorphism.

  The proof of case 2 is the same, appealing to \cite{MR2827723} lemma~3.3 part~6 rather than part~5.
\end{proof}

\section{Types.}
\label{sec:inertial-types}

In this section, unless otherwise stated all representations will be over a fixed algebraic
closure $\bar{E}$ of $E$.  We say that a representation of $W_F$ or $I_F$ on a finite-dimensional
$\bar{E}$-vector space $V$ is \emph{smooth} if it is continuous for the discrete topology on $V$,
and $\emph{continuous}$ if it is continuous for the $l$-adic topology on $V$.

\subsection{Inertial types}

A Weil--Deligne representation of the Weil group $W_F$ is a pair $(r,N)$ where 
\begin{itemize}
\item $r : W_F \rarrow GL(V)$ is a smooth representation on a finite-dimensional vector space $V$;
\item $N \in \End(V)$ satisfies \[r(g) N r(g)^{-1} = \|g\|N\] where $\|\cdot\| : W_F \onto W_F/I_F
  \onto q^{\ZZ}$ takes an arithmetic Frobenius element to~$q$.
\end{itemize}

If $\rho : W_F \rarrow GL(V)$ is a continuous representation of $W_F$ on a finite-dimensional vector space
$V$, then there is an associated Weil--Deligne representation (see for example \cite{MR546607}) that we denote $\WD(\rho)$.  

If $\rho : W_F \rarrow GL(V)$ is a smooth irreducible representation of $W_F$ on a finite dimensional
vector space $V$ and $k \geq 1$ is an integer, then define a Weil--Deligne representation
$\Sp(\rho,k)$ by
\[\Sp(\rho,k) = \left(V \oplus V(1) \oplus \ldots \oplus V(k-1),N\right)\]
where for $0 \leq i \leq k-2$, $N : V(i) \isomto V(i+1)$ is the isomorphism of vector spaces induced
by some choice of basis for $\bar{E}(1)$, and $N(V(k-1)) = 0$.  We define  $\Sp(\rho, 0) = 0$.

Every Frobenius-semisimple\footnote{Recall from, for example, \cite{MR546607} (4.1.3) that a
  Weil--Deligne representation $(r,N)$ is Frobenius--semisimple if $r$ is semisimple.  These
  representations form the Galois side of the local Langlands correspondence.} Weil--Deligne representation $(r,N)$
  is isomorphic to one of the form
\[\bigoplus_{i=1}^j \Sp(\rho_i,k_i)\] for smooth irreducible representations $\rho_i : W_F \rarrow
GL(V_i)$ and integers $j \geq 0$ and $k_i \geq 1$ for $i = 1, \ldots, j$.  Up to obvious
reorderings, the integers $j$ and $k_i$ are unique, and the representations $\rho_i$ are unique up to isomorphism.

\begin{definition}
  An \textbf{inertial type} is an isomorphism class of finite dimensional continuous representations
  $\tau$ of $I_F$ such that there exists a continuous representation $\rho$ of $W_F$ with
  $\rho|_{I_F} \cong \tau$.
\end{definition}

\subsection{}

The classification of (Frobenius-semisimple) Weil--Deligne representations yields a classification
of inertial types, which we now describe.

\begin{definition}
  The set $\Ic_0$ of \textbf{basic} inertial types is the set of inertial types $\tau_0$ that extend
  to a continuous \emph{irreducible} representation of $G_F$.
\end{definition}

Note that the $\tau_0$ do not need to be irreducible representations of $I_F$.

\begin{lemma}\label{lem:clifford}
Suppose that $t,t'$ are positive integers, $\rho_1,\ldots, \rho_t, \rho_1',\ldots, \rho_{t'}'$ are
  irreducible representations of $W_F$, and $k_1,\ldots,k_t, k_1',\ldots,k_{t'}'$ are positive integers.
  Then the representations of $W_F$ associated to \[\bigoplus_{i=1}^t \Sp(\rho_i,k_i)\] and
  \[\bigoplus_{i=1}^{t'} \Sp(\rho'_i,k'_i)\] have isomorphic restrictions to $I_F$ if and only if $t = t'$ 
  and there is an ordering $j_1, \ldots, j_t$ of $1,\ldots, t$ such that $k_i = k'_{j_i}$ and
  $\rho_i|_{I_F} \cong \rho'_{j_i}|_{I_F}$ for each $1 \leq i \leq t$.
  \end{lemma}

  \begin{proof} The ``if'' direction is clear.  We show the ``only if'' direction.  If $\rho$ is a
    continuous representation of $W_F$ with $\WD(\rho) = (r,N)$, then $r|_{I_F}$ and the
    $r|_{I_F}$-equivariant endomorphism $N$ are determined up to isomorphism by $\rho|_{I_F}$ (this
    follows from the construction of $\WD(\rho)$, see \cite{MR546607} Corollary~4.2.2).  So
    we may assume that $\rho = r$, so that all the $k_i$ are zero.  Now use the fact (proved by an
    exercise in Clifford theory) that, if $\rho$ is an irreducible representation of $W_F$,
    then \[\rho |_{I_F} \cong \mu_1 \oplus \ldots\oplus \mu_s\] for some integer $s$ and pairwise
    non-isomorphic irreducible representations $\mu_i$ of $I_F$ which are in a single orbit for the
    action of $G_F/I_F$ on irreducible representations of $I_F$; the representation $\mu_1$
    determines $\rho|_{I_F}$.  Therefore, if
  \[\bigoplus_{i=1}^t \rho_i|_{I_F} \cong \bigoplus_{i=1}^{t'} \rho'_i|_{I_F}\]
  then $\rho_1|_{I_F}$ has an irreducible component in common with some $\rho'_{j_1}|_{I_F}$, and so
  $\rho_1|_{I_F}\cong \rho'_{j_1}|_{I_F}$.  The lemma follows by induction.
\end{proof}

Let $\Part$ be the set of integer sequences $P = (P(1),P(2),\ldots)$ which are decreasing and
eventually zero.  We regard $P \in \Part$ as a partition of the integer
$\deg(P) = \sum_{i=1}^\infty P(i)$. For each $\tau_0 \in \Ic_0$, choose an irreducible extension
$\rho_{\tau_0}$ of $\tau_0$ to $W_F$.
\begin{definition}
Let $\Ic$ be the set of functions $\Pc : \Ic_0 \rarrow \Part$ with finite support. For $\Pc \in \Ic$ we can
form the Weil--Deligne representation 
\[\bigoplus_{\tau_0 \in \Ic_0} \bigoplus_{i=0}^\infty \Sp(\rho_{\tau_0},\Pc(\tau_0)(i)).\]
We define $\tau_\Pc$ to be the restriction to $I_F$ of the associated representation of $W_F$; it is
an inertial type.
\end{definition}
By Lemma~\ref{lem:clifford}, the isomorphism class of $\tau_{\Pc}$ is independent of the choices of
the $\rho_{\tau_0}$, and the map $\Pc \mapsto \tau_\Pc$ is a bijection between $\Ic$ and the set of
inertial types.  To $\Pc \in \Ic$ we associate the `supercuspidal support', the function $\scs(\Pc): \Ic_0
\rarrow \ZZ_{\geq 0}$ given by $\scs(\Pc)(\tau_0) = \deg \Pc(\tau_0)$. If $\tau = \tau_\Pc$ we
write $\scs(\tau) = \scs(\Pc)$.

Let $\succeq$ be the dominance order on $\Part$; that is, the partial order defined by $P_1 \succeq
P_2$ if and only if $\deg P_1 = \deg P_2$ and, for all $k \geq 1$, \[\sum_{i=1}^k P_1(i) \geq
\sum_{i=1}^k P_2(i).\] Then $\succeq$ induces a partial order on $\Ic$ for which $\Pc \succeq \Pc'$
if and only if $\Pc(\tau_0) \succeq \Pc'(\tau_0)$ for all $\tau_0 \in \Ic_0$; we also sometimes
regard $\succeq$ as a partial order on the set of inertial types.

\subsection{Fixed type deformation rings}
\label{sec:fixed-type-deform}

Let $\rhobar : G_F \rarrow GL_n(\FF)$ be a continuous representation and let $\tau$ be an inertial
type.  Suppose moreover that $\tau$ is defined over $E$, and that $E$ contains all
the $(q^{n!}-1)$th roots of unity.  We say that a morphism
$x:\Spec\bar{E} \rarrow \Spec R^\square(\rhobar)$ has type $\tau$ if the corresponding Galois
representation $\rho_x : G_F \rarrow GL_n(\bar{E})$ does.  Since $\tau$ is defined over $E$ this only
depends on the image of $x$.

\begin{definition}
  If $\tau$ and $\rhobar$ are as above, then $R^\square(\rhobar,\tau)$ is the reduced quotient of
  $R^\square(\rhobar)$ such that $\Spec R^\square(\rhobar,\tau)$ is the Zariski closure in $\Spec
  R^\square(\rhobar)$ of the $\bar{E}$-points of type $\tau$.
\end{definition}

If $x$ is an $\bar{E}$-point of $\Spec R^\square(\rhobar,\tau)$, say that $x$ is
\textbf{non-degenerate} if the associated Galois representation $\rho_x$ satisfies
$\WD(\rho_x) = r_l(\pi)$ for an irreducible admissible representation $\pi$ of $GL_n(F)$ that is
\emph{generic}\footnote{The significance to us of non-degenerate/generic representations is that
  they are contained in a unique component of the deformation rings
  (Proposition~\ref{prop:fixed-type} part 4) and the theory of $K$-types works well
  (Theorem~\ref{thm:K-types} part 3).} (see below for the defininitions of $r_l$ and generic).

\begin{proposition} \label{prop:fixed-type}
 For each inertial type $\tau$ defined over $E$:
 \begin{enumerate}
 \item $\Spec R^\square(\rhobar,\tau)$ is a union of irreducible components of $\Spec
   R^\square(\rhobar)$;
 \item if $x$ is a non-degenerate $\bar{E}$-point of $\Spec R^\square(\rhobar)$, then $x$ lies on a
   unique irreducible component of $\Spec R^\square(\rhobar)$ and $R^\square(\rhobar)[1/l]$ is
   formally smooth at $x$;
\item the non-degenerate $\bar{E}$-points are Zariski dense in $\Spec R^\square(\rhobar)$;
\item if $x$ is a non-degenerate $\bar{E}$-point of $\Spec R^\square(\rhobar,\tau)$, then $\rho_x$ has type $\tau$. 
 \end{enumerate}
\end{proposition}
\begin{proof}
  Parts 2--4 follow from \cite{BLGGT2014-PotentialAutomorphy} Lemmas~1.3.2 and~1.3.4.  To show the
  first part we use Proposition~\ref{prop:cpts}.  Firstly, note that under the isomorphism of
  Lemma~\ref{lem:tame-reduction} we have that
  \[ R^\square(\rhobar,\tau) \cong \left(\widehat{\bigotimes}_{[\theta]}
      R^\square(\rhobar_\theta,\tau_\theta) \right)[[X_1, \ldots, X_{n^2 - \sum n^2_\theta}]]\] for
  some tamely ramified inertial types $\tau_\theta$.  We have to check that every minimal prime
  ideal of $R^\square(\rhobar, \tau)$ pulls back to a minimal prime ideal of $R^\square(\rhobar)$.
  This property may be checked after enlarging $E$ to a finite extension $E'$ with ring of integers
  $\Oc'$, which we choose so that every irreducible component of
  $R^\square(\rhobar, \tau)\otimes_\Oc \Oc'$ and of every $R^\square(\rhobar, \tau')$ is
  geometrically integral.  Then, by \cite{MR2827723} lemma~3.3 part~5 (see also
  Lemma~\ref{lem:product}), it suffices to prove the claim in the case that $\rhobar$ is tamely
  ramified.

  So suppose that $\rhobar$ is tamely ramified.  From our choices of topological generators
  $\sigma, \phi$ of $G_F/I_F$ we have, as in the proof of Theorem~\ref{thm:complete-intersection},
  that $R^\square(\rhobar)$ is the completed local ring of $\Mc(n,q)_\Oc$ at the closed point of the
  special fibre corresponding to $\rhobar$; in particular we have a flat morphism
  $i:\Spec R^\square(\rhobar) \rarrow \Mc(n,q)_\Oc$.  Let $C$ be the conjugacy class in
  $GL_n(\bar{E})$ of $\tau(\sigma)$.  Then in Proposition~\ref{prop:cpts} we defined the irreducible
  component $\Mc(n,q, \Sigma \sim C)_E$ of $\Mc(n,q)_E$; let $\Mc(n,q,\Sigma\sim C)_\Oc$ be its
  closure in $\Mc(n,q)_\Oc$, which is an irreducible component.  Then
  \[\Spec R^\square(\rhobar, \tau) = i^{-1}(\Mc(m,q,\Sigma\sim C)_\Oc) \subset R^\square(\rhobar)\] is a union of irreducible
  components of $\Spec R^\square(\rhobar)$ by the going down theorem.
\end{proof}

\subsection{$K$-types}
\label{sec:k-types-1}

Recall the local Langlands correspondence $\rec_F$ of \cite{MR1876802} Theorem~A, which is defined
over the complex numbers.  If $\pi$ is an irreducible admissible $\bar{\QQ}_l$-representation of $GL_n(F)$ and
$\iota : \bar{\QQ}_l \isomto \CC$ is our choice of isomorphism, let 
\[r_l(\pi) = \iota^{-1} \circ \rec_F(\iota \circ (\pi \otimes |\det|^{\frac{1-n}{2}})).\] Then
$r_l(\pi)$ is an $n$-dimensional Frobenius-semisimple Weil--Deligne representation of $W_F$ over
$\bar{\QQ}_l$ and is independent of the choice of $\iota$ (see
\cite{MR1838079}~\S7.4).\footnote{This normalisation is convenient for local--global compatibility;
  the notation agrees with that of \cite{1310.0831} but differs from that of \cite{MR1876802} --- our
  $r_l(\pi)$ is their $r_l(\pi^\vee \otimes |\det|^{1-g})$.}

If $\Sc : \Ic_0 \rarrow \ZZ_{\geq 0}$ is a function with finite support such that \[\sum_{\tau_0 \in
  \Ic_0} \dim \tau_0 \Sc(\tau_0) = n,\] then we can consider the full subcategory $\Omega_{\Sc}$ of
$\Rep_{\bar{E}}(GL_n(F))$ all of whose irreducible subquotients $\pi$ satisfy
\[ \scs(r_l(\pi)|_{I_F}) = \Sc.\]
The category $\Rep_{\bar{E}}(GL_n(F))$ is then the direct product of the 
$\Omega_{\Sc}$; these are the \textbf{Bernstein components} of $\Rep_{\bar{E}}(GL_n(F))$.  See, for
example, \cite{MR1643417}~\S1.  If $\Sc$ is supported on a single $\tau_0$ and maps it to 1, then we
say that $\Omega_{\Sc}$ is supercuspidal.  This is equivalent to every irreducible object of
$\Omega_{\Sc}$ being supercuspidal.

It is one of the main results of the theory of Bushnell and Kutzko developed in
\cite{BushnellKutzko1993-AdmissibleDualGLN} and \cite{BushnellKutzko1999-SemisimpleTypesGLN} that,
for each Bernstein component $\Omega$ of \[\Rep_{\bar{E}}(GL_n(F)),\] there is a compact open subgroup
$J \subset GL_n(F)$ and a representation $\lambda$ of $J$ with the following property: if $\pi \in
\Rep_{\bar{E}}(GL_n(F))$ is generated by its $\lambda$-isotypic vectors, then $\pi$ is in $\Omega$.  We
call $(J,\lambda)$ a \textbf{type} for the Bernstein component $\Omega$.  If $K \supset J$ is a maximal
compact subgroup of $GL_n(F)$ and $\Omega$ is supercuspidal, then $\Ind_J^K \lambda$ is irreducible
and is a \textbf{$K$-type} for $\Omega$.

In \cite{SchneiderZink1999-KTypesTemperedComponentsGeneralLinear}, Schneider and Zink
refine this by providing $K$-types for a certain `stratification' of $\Rep_{\bar{E}}(GL_n(F))$.  We use
their results in the following Galois--theoretic form (c.f. \cite{MR2656025} Proposition~6.3.3):

\begin{theorem} \label{thm:K-types} Let $\tau$ be an inertial type of dimension $n$.  Then there is a smooth
  irreducible $\bar{E}$-representation $\sigma(\tau)$ of $GL_n(\Oc_F)$ such that, for each
  irreducible admissible $\bar{E}$-representation $\pi$ of $GL_n(F)$, we have:
  \begin{enumerate}
  \item if $\pi|_{GL_n(\Oc_F)}$ contains $\sigma(\tau)$, then $r_l(\pi)|_{I_F} \preceq \tau$;
  \item if $r_l(\pi)|_{I_F} \cong \tau$, then $\pi|_{GL_n(\Oc_F)}$ contains $\sigma(\tau)$ with
    multiplicity one;
  \item if $r_l(\pi)|_{I_F} \preceq \tau$ and $\pi$ is \emph{generic}, then $\pi|_{GL_n(\Oc_F)}$
    contains $\sigma(\tau)$.
  \end{enumerate}
\end{theorem}

\begin{proof}
  This is \cite{MR2656025} Proposition~6.3.3, except that we have replaced the hypothesis
  `tempered' with `generic'.  That we can do this follows from the proof of
  \cite{SchneiderZink1999-KTypesTemperedComponentsGeneralLinear} Proposition~5.10 --- the only
  property of tempered representations that is used is that they occur as the \emph{irreducible}
  parabolic induction of a discrete series representation, and this continues to hold for generic
  representations. See also Corollary~\ref{cor:construction-proof} below.\end{proof}

\begin{example} \label{eg:types-gl2}
  Let $\Pc_0,\Pc_1\in \Ic$ be the maps that take the trivial representation to
  (respectively) $(1,1,0,0,\ldots)$ and $(2,0,0,\ldots)$, and everything else to zero.  Let $\tau_0$
  and $\tau_1$ be the corresponding inertial types; they are respectively the
  trivial two-dimensional representation and the non-trivial unipotent two-dimensional
  representation of $I_F$.  We have $\Pc_0 \prec \Pc_1$ and they are not comparable to any other
  elements of $\Ic$.  

  The representation $\sigma(\tau_0)$ is the trivial representation of $GL_2(\Oc_F)$, while
  $\sigma(\tau_1)$ is inflated from the Steinberg representation of $GL_2(k_F)$.  

  Then $\pi$ contains $\sigma(\tau_0)$ if and only if $\pi$ is unramified, and so if and only if
  $r_l(\pi)|_{I_F} =\tau_0$.  On the other hand, $\pi$ containing $\sigma(\tau_1)$ implies that
  $r_l(\pi)|_{I_F}$ is unipotent --- that is to say, that $r_l(\pi)|_{I_F} \preceq \tau_1$ ---
  but the converse is false for $\pi$ an unramified character (these are non-generic).
\end{example}

\begin{remark}
  In general the representation $\sigma(\tau)$ is \emph{not} determined by the above properties ---
  this already happens when $n=2$ if $|k_F| = 2$, see \cite{2002}, A.1.5, (3).  It is known to be
  unique when $\tau$ corresponds to a \emph{supercuspidal} Bernstein component, see
  \cite{MR2180458}, and expected to be unique if $p > n$, see \cite{JMJ:9091284} Conjecture~4.1.3.
  
  We will give an explicit construction of $\sigma(\tau)$ in
  section~\ref{sec:k-types}, and see
  Corollary~\ref{cor:construction-proof} for a proof that the
  representations we construct have the desired properties (modulo the
  translation into Galois theoretic language, which is straightforward
  and exactly as in \cite{MR2656025}).  Our construction follows
  closely that of
  \cite{SchneiderZink1999-KTypesTemperedComponentsGeneralLinear}, and
  it seems likely that the two constructions yield the same
  representations $\sigma(\tau)$, but we do not need this and have not
  checked it.  When $n = 2$ it is not hard to check that both
  constructions do agree, and that they agree with the construction of
  \cite{2002} (even when $|k_F| = 2$).
\end{remark}

\subsection{The Bernstein--Zelevinsky classification}
\label{sec:bernst-zelev-class}

It will be useful to recall a little notation to do with the Bernstein--Zelevinsky classification
of irreducible admissible representations of $GL_n(F)$; we follow \cite{MR689531}.  For
definiteness, fix a choice of square root of $q$ in $\bar{E}$.  Then if $P \subset GL_n(F)$ is a
standard parabolic subgroup with Levi factor $M = \prod_{i=1}^k M_i$ and unipotent radical $U$, and
if $\rho_i$ are smooth representations of $M_i$, we can regard $\tensor_i \rho_i$ as a
representation of $P$ by allowing $U$ to act trivially and then form the normalised parabolic
induction of $\tensor_i \rho_i$ from $P$ to $GL_n(F)$; call this representation
\[\rho_1 \times \ldots \times \rho_k.\]
If $\pi$ is an irreducible supercuspidal representation of $GL_m(F)$ and $k \geq 1$ is an integer,
let \[\Delta(\pi,k) = \{\pi,\pi\otimes |\det|, \ldots , \pi \otimes|\det|^{k-1}\}.\] A set of this
form is called a \textbf{segment}.  Two segments $\Delta_1$ and $\Delta_2$ are called
\textbf{linked} if $\Delta_1 \not\subset \Delta_2$, $\Delta_2 \not\subset \Delta_1$ and $\Delta_1
\cup \Delta_2$ is a segment, and we say that $\Delta(\pi,k)$ \textbf{precedes} $\Delta(\pi',k')$ if
they are linked and $\pi' = \pi\otimes |\det|^s$ for some $s \geq 1$.  If $\Delta = \Delta(\pi,k)$
is a segment, let $L(\Delta)$ be the unique irreducible quotient of \[\pi \times (\pi \otimes
|\det|) \times \ldots \times (\pi \otimes |\det|^{k-1});\] it is an irreducible admissible
representation of $GL_{km}(F)$. If $\Delta_1, \ldots, \Delta_t$ are segments then we may reorder
them so that, for $i < j$, $\Delta_i$ does not precede $\Delta_j$.  Then
\[L(\Delta_1) \times \ldots \times L(\Delta_t)\] is a representation of $GL_n(F)$ for suitable $n$,
with a unique irreducible quotient $L(\Delta_1,\ldots,\Delta_t)$, which is independent of the
ordering chosen (so long as the `precedence' condition is satisfied).  Every irreducible admissible
representation of $GL_n(F)$ is of this form, uniquely up to reordering the $\Delta_i$. The
representation \[L(\Delta_1)\times \ldots \times L(\Delta_t)\] is irreducible if and only if no two of
the $\Delta_i$ are linked. In this case $L(\Delta_1,\ldots,\Delta_t)=L(\Delta_1)\times \ldots \times
L(\Delta_t)$ is \textbf{generic}, and moreover every irreducible generic representation is of this
form (see \cite{MR584084} Theorem~9.7).

The compatibility with the above classification of Frobenius-semisimple Weil--Deligne
representations is as follows.  If $d_1,\ldots,d_t$ are positive integers with $\sum d_i = n$,
$\pi_1,\ldots,\pi_t$ are supercuspidal representations of $GL_{d_i}(F)$, and
$k_1,\ldots,k_t$ are positive integers, then for \[\Delta_i = \Delta(\pi_i \otimes
|\det|^{\frac{1-d_i}{2}},k_i)\] we have:
\begin{equation}\label{eq:loclangBZ}\bigoplus_{i=1}^t \Sp(r_l(\pi_i),k_i) =
  r_l(|\det|^{\frac{n-1}{2}}\otimes L(\Delta_1,\ldots,\Delta_t)).\end{equation}

\emph{The next two paragraphs are only required in section~\ref{sec:k-types}}.  A \textbf{supercuspidal
  pair} is a pair $(M,\pi)$ where $M$ is a Levi subgroup of some $GL_n(F)$ and $\pi$ is a
supercuspidal representation of $M$.  We say that supercuspidal pairs $(M,\pi)$ and $(M',\pi')$ are
\textbf{inertially equivalent} if there is an element $g \in G$ and an unramified character $\alpha$
of $M'$ such that $M' = gMg^{-1}$ and $\pi' = \alpha \otimes \pi^{g}$.  We write $[M,\pi]$ for the
inertial equivalence class of $(M,\pi)$.  If $\Omega$ is a Bernstein component of
$\Rep_{\bar{E}}(GL_n(F))$, then there is a unique inertial equivalence class of supercuspidal pair
$[M,\pi]$ such that every irreducible object of $\Omega$ is a subquotient of a representation
parabolically induced from a supercuspidal pair $(M,\pi)$ in that inertial equivalence class (for
some choice of parabolic subgroup).

The \textbf{essentially discrete series} representations\footnote{That is, the unramified twists of
  discrete series representations} of $GL_n(F)$ are precisely those of the form $L(\Delta)$ for some
segment $\Delta$.  Define a \textbf{discrete pair} to be a pair $(M,\pi)$ where $M$ is a Levi
subgroup of some $GL_n(F)$ and $\pi$ is an essentially discrete series representation of $M$; say
that discrete pairs $(M,\pi)$ and $(M',\pi')$ are inertially equivalent if there is an element $g
\in G$ and an unramified character $\alpha$ of $M'$ such that $M' = gMg^{-1}$ and $\pi' = \alpha
\otimes \pi^{g}$, and write $[M,\pi]$ for the inertial equivalence class of $(M,\pi)$.  If $\pi$ is
supercuspidal this agrees with the notion of inertial equivalence for supercuspidal pairs.  If
$\Pc\in \Ic$ then we can associate an inertial equivalence class $[M,\pi]$ of discrete pairs to
$\Pc$ as follows: for every $\tau_0 \in \Ic_0$ pick a supercuspidal representation $\pi_{\tau_0}$ of
$GL_{\dim(\tau_0)}(F)$ with $r_l(\pi_{\tau_0})|_{I_F} \cong \tau_0$.  Then \[\left[
  \prod_{\tau_0\in \Ic_{0},i\in \NN} GL_{\Pc(\tau_0)(i)\dim(\tau_0)}(F), \bigotimes_{\tau_0\in
    \Ic_{0},i\in \NN}L(\Delta(\pi_{\tau_0},\Pc(\tau_0)(i)))\right]\] is the required class of
discrete pairs.  If $(M,\pi) = (\prod_{i=1}^r M_i,\bigotimes_{i=1}^r L(\Delta_i))$ is a discrete
pair, then we can define $L(M,\pi)$ to be $L(\Delta_1,\ldots,\Delta_r)$.  From
equation~\eqref{eq:loclangBZ} we see that, if $\Pc \in \Ic$ has degree $n$, and $[M,\pi]$ is the
associated inertial equivalence class of discrete pair, then the irreducible admissible
representations $\pi$ of $GL_n(F)$ such that $r_l(\pi)$ has type $\tau_\Pc$ are precisely the
$L(M,\pi)$ for $(M,\pi)$ in the inertial equivalence class $[M,\pi]$.

\section{The Breuil--M\'{e}zard conjecture.}
\label{sec:breuil-mezard-lneq}

\subsection{Reduction maps} Let $\rhobar : G_F \rarrow GL_n(\FF)$ be a continuous representation, and suppose that
$E$ is large enough that, for every inertial type $\tau$ that is the type of some lift of $\rhobar$, both
$\tau$ and $\sigma(\tau)$ are defined over $E$.  We have defined the framed deformation ring
$R^\square(\rhobar)$, which is flat and equidimensional of relative dimension $n^2$ over $\Oc$.  Thus we have
the free abelian groups $\Zc(R^\square(\rhobar))$ on the irreducible components of
$R^\square(\rhobar)$, $\Zc(\Rbarbox(\rhobar))$ on the irreducible components of
$\Rbarbox(\rhobar) = \Rbox(\rhobar) \otimes \FF$, and a reduction map
\[ \red : \Zc(\Rbox(\rhobar)) \rarrow \Zc(\Rbarbox(\rhobar)).\] Let $R_E(GL_n(\Oc_F))$ be the Grothendieck group of
finite-dimensional smooth representations of $GL_n(\Oc_F)$ over $E$, and let $R_{\FF}(GL_n(\Oc_F))$
be the Grothendieck group of finite-dimensional smooth representations of $GL_n(\Oc_F)$ over $\FF$.
Then the operation of choosing a $GL_n(\Oc_F)$-invariant lattice and reducing modulo $\lambda$
defines a group homomorphism:
\[\red : R_E(GL_n(\Oc_F)) \longrightarrow R_{\FF}(GL_n(\Oc_F))\]
that is independent of the choice of lattice.

\subsection{Cycle map.} \begin{lemma}
  If $\pi$ and $\pi'$ are generic irreducible admissible representations of $GL_n(F)$ such that
  $r_l(\pi)|_{I_F}\cong r_l(\pi')|_{I_F}$, then \[
  \pi|_{GL_n(\Oc_F)} \cong \pi'|_{GL_n(\Oc_F)}.\]
\end{lemma}

\begin{proof}
  Let $\Pc\in \Ic$ be such that $r_l(\pi)|_{I_F}\cong r_l(\pi')|_{I_F}\cong \tau_\Pc$ and
  let $\tau_1, \ldots, \tau_r$ be the elements of $\Ic_0$ with $\deg\Pc(\tau_i) = d_i \neq 0$.  Pick
  supercuspidal representations $\pi_{i}$ of $GL_{\dim \tau_i}(F)$ such that $r_l(\pi_i)|_{I_F}
  \cong \tau_i$ and let $\Delta_{i,j}$ be the segment $\Delta(\pi_i,\Pc(\tau_i)(j))$ for each $j$
  such that $\Pc(\tau_i)(j) \neq 0$.  Then every generic irreducible admissible representation $\pi$
  of $GL_n(F)$ such that $r_l(\pi)|_{I_F} \cong \tau$ is of the form
  \[ (\alpha_{1,1}\circ \det)L(\Delta_{1,1}) \times \ldots \times (\alpha_{i,j}\circ
  \det)L(\Delta_{i,j}) \times \ldots\] for unramified characters $\alpha_{i,j}$ of $F^\times$.  The
  lemma follows from the following consequence of the Iwasawa decomposition: for any parabolic
  subgroup $P \subset GL_n(F)$ and representation $\rho$ of $P$,
\[(\Ind_P^G \rho)|_{GL_n(\Oc_F)} = \Ind_{P\cap GL_n(\Oc_F)}^{GL_n(\Oc_F)}(\rho|_{P \cap GL_n(\Oc_F)}). \qedhere\]
\end{proof}

\begin{definition}\label{def:multiplicities}
  If $\theta$ is a finite-length representation of $GL_n(\Oc_F)$ and $\tau'$ is an inertial type, then
  $m(\theta,\tau')$ is defined to be the non-negative integer
\[\dim \Hom_{\bar{E}[GL_n(\Oc_F)]}(\theta,\pi|_{GL_n(\Oc_F)})\] for any generic
irreducible admissible representation $\pi$ of $GL_n(F)$ such that \[r_l(\pi)|_{I_F} \cong \tau'.\]
\end{definition}

\begin{proposition}\label{prop:type-mults}
  Suppose that $\Pc,\Pc' \in \Ic$.  Then 
\[m(\sigma(\tau_{\Pc}),\tau_{\Pc'}) = \prod_{\tau_0 \in \Ic_0} m(\Pc(\tau_0),\Pc'(\tau_0))\]
where $m(\Pc(\tau_0),\Pc'(\tau_0))$ is the Kostka number (Definition~\ref{def:kostka})  for the pair of partitions
$\Pc(\tau_0),\Pc'(\tau_0)$ (and is in particular zero if $\deg \Pc(\tau_0) \neq \deg \Pc'(\tau_0)$
for some $\tau_0$).
\end{proposition}
\begin{proof}
  This is proved as Corollary~\ref{cor:multiplicity-calculation} below.
\end{proof}

\begin{definition}\label{def:cyc}
  Define a homomorphism
  \[\cyc : R_E(GL_n(\Oc_F)) \longrightarrow \Zc(R^\square(\rhobar))\]
  given (on irreducible $E$-representations $\sigma$ of $GL_n(\Oc_F)$) by:
  \[\cyc(\sigma) = \sum_{\Pc \in \Ic} m(\sigma^\vee,\tau_\Pc) Z(R^\square(\rhobar,\tau_\Pc)).\]
  This sum makes sense since $m(\sigma^\vee, \tau_\Pc)$ is non-zero for only finitely many $\tau_\Pc$.
\end{definition}

\begin{conjecture} \label{conj:BM}
  There exists a unique homomorphism 
\[\bar{\cyc} : R_{\FF}(GL_n(\Oc_F)) \longrightarrow \Zc(\bar{R}\,^\square(\rhobar))\]
making the following diagram commute:
\begin{equation}
\begin{CD}
  \label{eq:BM}
R_E(GL_n(\Oc_F)) @>{\cyc}>> \Zc(R^\square(\bar{\rho})) \\
@V{\red}VV  @V{\red}VV \\
R_{\FF}(GL_n(\Oc_F)) @>{\bar{\cyc}}>> \Zc(\bar{R}\,^\square(\rhobar)).
\end{CD}
\end{equation}
\end{conjecture}

Certainly there is at most one map $\bar{\cyc}$ making diagram (\ref{eq:BM}) commute.  This is because the
map $\red : R_E(GL_n(\Oc_F)) \rarrow R_{\FF}(GL_n(\Oc_F))$ is surjective, which follows from the
corresponding fact for finite groups (see \cite{MR0450380} Theorem~33) because every smooth $E$- or
$\FF$-representation of $GL_n(\Oc_F)$ factors through a finite quotient.

The main result of this chapter is:
\begin{theorem} \label{thm:BM-Gln}
  If $l > 2$ then Conjecture~\ref{conj:BM} is true.
\end{theorem}

\begin{proof}
  To prove the existence of the map $\bar{\cyc}$, we must show that $\ker(\red) \subset \ker(\cyc)$;
  this may be checked after making a finite extension of $E$.  Therefore we can and do assume that
  every irreducible component of $R^\square(\rhobar)$ and of $R^\square(\rhobar)\otimes\FF$ is
  geometrically irreducible.
  
  Let $\Rep^{fg}_\Oc(GL_n(\Oc_F))$ be the category of finitely generated $\Oc$-modules with a smooth
  action of $GL_n(\Oc_F)$.  In the next section, we will show (using the
  Taylor--Wiles--Kisin patching method) that there are positive integers $c$ and $d$, a
  geometrically integral object $A$ of $\Cc_{\Oc}^\wedge$, and an exact functor $H_\infty$ from
  $\Rep^{fg}_\Oc(GL_n(\Oc_F)^{\times d})$ to the category of finitely generated modules over
  \[R^\square(\rhobar)^{\otimes d}\hat{\otimes} A\]
  with the following properties: 
  \begin{itemize}
  \item for all $\sigma \in \Rep^{fg}_\Oc(GL_n(\Oc_F)^{\times d})$, \[H_\infty(\sigma\otimes_\Oc \FF) =
    H_\infty(\sigma) \otimes_\Oc \FF;\] 
  \item if $\sigma \in \Rep^{fg}_\Oc$ is $\lambda$-torsion free, then so is $H_\infty(\sigma)$;
  \item if $\sigma = \bigotimes_{i=1}^d \sigma_i \in \Rep^{fg}_\Oc(GL_n(\Oc_F)^{\times d})$ is finite free as an $\Oc$-module, then 
    \[Z(H_\infty(\sigma)) = c \cdot \cyc^{\otimes d}(\sigma).\] Here the left hand side is an
    element of $\Zc(R^\square(\rhobar)^{\otimes d}\hat{\otimes}A))$, the right hand side is an
    element of $\bigotimes_{i=1}^d \Zc(R^\square(\rhobar))$, and we identify these groups using
    Lemma~\ref{lem:product} and the fact that $A$ is geometrically integral.
  \end{itemize}
  The last of these is Corollary~\ref{cor:patch=cycle}.

  Let $\bar{A} = A \otimes \FF$, and let $C = Z(\bar{A}) \in \Zc(\bar{A})$; clearly $C \neq 0$.  We
  identify $\Zc(\Rbarbox(\rhobar)^{\otimes d} \hat{\otimes} \bar{A})$ with
  $\bigotimes_{i=1}^d \Zc(\Rbarbox(\rhobar))\otimes \Zc(\bar{A})$ using Lemma~\ref{lem:product}.
  Now, $Z(\cdot)$ is additive on short exact sequences (see \cite{JMJ:9091284} Lemma~2.2.7) and,
  using Lemma~\ref{lem:cycle-reduction} and the first two properties above, we find that the
  following diagram commutes (the horizontal maps are well defined since $H_\infty(\cdot)$ is
  exact):
  \[\begin{CD}
   R_E(GL_n(\Oc_F))^{\otimes d} @>Z(H_\infty(\cdot))>> \Zc(R^\square(\rhobar))^{\otimes d}\\
   @V^{\red^{\otimes d}}VV @V^{\red^{\otimes d}\otimes C}VV\\
   R_{\FF}(GL_n(\Oc_F))^{\otimes d} @>Z(H_\infty(\cdot))>> \Zc(\Rbarbox(\rhobar))^{\otimes d}\otimes
   \Zc(\bar{A}).
   \end{CD}\]
   
   Moreover, the topmost map is just $c \cdot\cyc^{\otimes d}$ by the third listed property of
   $H_\infty$.    We deduce that $\ker(\red) \subset \ker(\red \circ \cyc)$; if not, then we may pick $\alpha \in
   \ker(\red)$ with $\beta = \red(\cyc(\alpha)) \neq 0$.  But then 
   \[c\cdot\red^{\otimes d}(\cyc^{\otimes d}(\alpha\otimes \ldots \otimes \alpha))\otimes C = c(\beta \otimes
   \ldots \otimes \beta \otimes C) \neq 0\] and also
   \[
   c\cdot\red^{\otimes d}(\cyc^{\otimes d}(\alpha \otimes \ldots \otimes \alpha))\otimes C =
   Z(H_\infty(\red(\alpha) \otimes \ldots \otimes \red(\alpha))) = 0,\] a contradiction.
\end{proof}
\begin{remark} Let $\Tc$ be the subgroup of $R_E(GL_n(\Oc_F))$ generated by the $\sigma(\tau)$ for
  inertial types $\tau$, and let $\bar{\Tc}$ be the subgroup of $R_{\FF}(GL_n(\Oc_F))$ generated by
  those irreducible representations appearing as a constituent of some $\red(\sigma(\tau))$.  Then
  $\red : \Tc \rarrow \bar{\Tc}$ is surjective, by Theorem~\ref{thm:type-reduction-gln} below.  It
  follows that the version of Theorem~\ref{thm:BM-Gln} in which $R_E(GL_n(\Oc_F))$
  (resp. $R_{\FF}(GL_n(\Oc_F))$) is replaced by $\Tc$ (resp. $\bar{\Tc}$) is also true --- the only
  possible issue being the uniqueness of $\bar{\cyc}$.  When $l > n$, $l \mid q - 1$, and
  $\rhobar|_{\tilde{P}_F}$ is trivial, we prove a version of the theorem with a still further
  restricted choice of $\Tc$ in section~\ref{sec:towards-local-proof} below, using local methods.
\end{remark}

It is natural to ask to what extent Theorem~\ref{thm:BM-Gln} gives a `formula' for the cycle
$\red(Z(\Rbox(\rhobar,\tau)))$.  This is answered by:
\begin{proposition} \label{prop:cyc-image} The image of $\cyc$ is the subgroup $\Hc$ of $\Zc(\Rbox(\rhobar))$ spanned by
  the cycles $Z(R^\square(\rhobar,\tau))$ for varying $\tau$.  Moreover, the restriction of $\cyc$ to
  $\Tc$ (see the previous remark) is a bijection onto $\Hc$.
\end{proposition}
\begin{proof}
  It is clear that the image of $\cyc$ is contained in $\Hc$.  By Proposition~\ref{prop:type-mults}
  and basic properties of Kostka numbers, the matrix of multiplicities $m(\sigma(\tau),\tau')$ is
  upper triangular with `1's on the diagonal (for an appropriate ordering of the various $\tau$).
  It follows that $\cyc$ restricted to $\Tc$ is a bijection onto $\Hc$, and thus that the image of
  $\cyc$ is all of $\Hc$.
\end{proof}

Let $\cyc^{-1}$ be the inverse of $\cyc : \Tc \rarrow \Hc$.  For $\tau$ an inertial type, let
$r(\tau) = \cyc^{-1}(Z(R^\square(\rhobar,\tau)))$; then
\[r(\tau) = \sum_{\tau'} m^{-1}(\sigma(\tau), \tau') \sigma(\tau')\] where
$m^{-1}(\sigma(\tau),\tau')$ is the $(\tau, \tau')$-entry of the \emph{inverse} of the matrix
$(m(\sigma(\tau),\tau'))_{\tau,\tau'}$.  Using Proposition~\ref{prop:type-mults},
$m^{-1}(\sigma(\tau), \tau')$ can be written as a product of entries of the inverse to the matrix of
Kostka numbers; moreover, it is zero unless $\tau$ and $\tau'$ have the same semisimplification.  If
$\tau$ is semisimple then $r(\tau) = \sigma(\tau)$, but in general it is only an element of the
Grothendieck group.  As an immediate consequence of Theorem~\ref{thm:BM-Gln} we have:
\begin{corollary} \label{cor:cyc-inverse}
  For each inertial type $\tau$, 
  \[\red(Z(\Rbox(\rhobar,\tau))) = \bar{\cyc}(\red(r(\tau))).\]
\end{corollary}
This expresses $\red(Z(\Rbox(\rhobar,\tau)))$ in terms of the $\bar{\cyc}(\theta)$ for $\theta$ running
over the irreducible $\bar{\FF}$-representations of $GL_n(\Oc_F)$.  We say nothing here about the
determination of the $\bar{\cyc}(\theta)$.

\begin{example} Let $n = 2$.  As in Example~\ref{eg:types-gl2}, let $\tau_0$ and $\tau_1$ be
  respectively the trivial and non-trivial two-dimensional unipotent representations of $I_F$.  Then
  $\sigma(\tau_0) = \mathbbm{1}$ is the trivial representation and $\sigma(\tau_1) = \St$ is
  inflated from the Steinberg representation of $GL_2(k_F)$.  In this case,
  $\cyc(\mathbbm{1}) = Z(R^\square(\rhobar,\tau_0))$ and
  $\cyc(\St) = Z(R^\square(\rhobar, \tau_1)) + Z(R^\square(\rhobar,\tau_0))$.  Inverting
  this map, we find that $r(\tau_0) = \mathbbm{1}$, while
  \[r(\tau_1) = \St - \mathbbm{1}.\]
\end{example}

\begin{example} Let $n = 3$.  Let $\tau_0$, $\tau_1$ and $\tau_2$ be the three-dimensional unipotent
  representations of $I_F$ for which the Weil--Deligne monodromy operator $N$ has rank $0$, $1$ and
  $2$, respectively.  Then $\sigma(\tau_0)$, $\sigma(\tau_2)$ are the inflations to $GL_3(\Oc_F)$
  of, respectively, the trivial representation and the Steinberg representation of $GL_3(k_F)$;
  $\sigma(\tau_1)$ is then inflated from the remaining irreducible unipotent representation of
  $GL_3(k_F)$.  Then the representations $r(\tau_i)$ are as follows:
  \begin{align*}
    r(\tau_0) &= \sigma(\tau_0) \\
    r(\tau_1) &= \sigma(\tau_1) - 2 \sigma(\tau_0) \\
    r(\tau_2) &= \sigma(\tau_2) - \sigma(\tau_1) + \sigma(\tau_0).
  \end{align*}
\end{example}

\begin{remark}
  If multiple components of $\Spec \Rbox(\rhobar)$ have the same type, then $\Hc$ will be a strict
  subgroup of $\Zc(R^\square(\rhobar,\tau))$; this happens, for instance, if $n = 2$,
  $\rhobar = \mathbbm{1} \oplus \chi$ where $\chi$ is the cyclotomic character, and
  $q \equiv -1 \bmod l$ (see \cite{MR3554238} Proposition~5.6).
\end{remark}

\begin{remark}
  It is also natural to ask what the image of $\bar{\cyc}$ is.  One can obtain a result similar to
  Proposition~\ref{prop:cyc-image}.  For every isomorphism class of irreducible representation
  $\bar{\tau}$ of $I_F$ over $\FF$ that extends to a representation of $G_F$, one can say what it
  means for an irreducible component of $\Rbarbox(\rhobar)$ to have inertial type $\bar{\tau}$ and
  consider the cycle $Z(\bar{\tau})$ equal to the formal sum of the irreducible components of type
  $\bar{\tau}$.  Then the image of $\bar{\cyc}$ is the subgroup of $\Zc(\Rbarbox(\rhobar))$ spanned
  by the cycles $Z(\bar{\tau})$.  This can be proved by constructing a `minimal lift' $\tau$ of
  $\bar{\tau}$ (this is related to the `minimal deformations' of
  \cite{ClozelHarrisTaylor2008-Automorphy} section~2.4.4) and showing that
  $\red(Z(\Rbox(\rhobar, \tau))) = Z(\bar{\tau})$ using the methods of
  section~\ref{sec:reduct-finite-type}.
\end{remark}

\begin{remark}
  An explicit description of $\bar{\cyc}(\sigma)$ for irreducible $\FF$-representations $\sigma$ of
  $GL_n(k_F)$ would make computing the decomposition numbers of $GL_n(k_F)$ in characteristic $l$
  equivalent to computing the reduction map $\Zc(R^\square(\rhobar)) \rarrow \Zc(\Rbarbox(\rhobar))$
  for the case of a tamely ramified $\rhobar$.  On the other hand, I do not know whether it is
  realistic to expect such an explicit description.  Note that, when $l = p$, determining
  $\bar{\cyc}$ (or at least, the irreducible representations for which it is non-zero) is
  essentially the weight part of Serre's conjecture --- compare for instance Remark~5.5.3 of
  \cite{JMJ:9091284}.
\end{remark}

\begin{remark}
  In the $l = p$ setting, \cite{JMJ:9091284} Conjectures~4.1.6 and~4.2.1 deal only with the
  potentially crystalline situation; this corresponds to working only with semisimple $\tau$, so
  that $r(\tau) = \sigma(\tau)$.  Comparing with their Conjecture~4.2.1, for suitable definitions
  of $\cyc$ and $\bar{\cyc}$ in the $l = p$ setting we would have (in their notation) that
  \[\cyc(\sigma(\tau) \otimes L_\lambda) = Z(R^\square_{\bar{r},\lambda, \tau})\] for a semisimple
  inertial type $\tau$ and dominant weight $\lambda$, and that \[\bar{\cyc}(F_a) = \Cc_a\] for a Serre
  weight $a$.  Conjecture~4.2.1 of \cite{JMJ:9091284} can then be reformulated as 
  \begin{equation}\label{eqn:bmlp}\bar{\cyc}(\red(\sigma(\tau) \otimes L_\lambda)) = \red(\cyc(\sigma(\tau) \otimes L_\lambda)).\end{equation}
  To generalise their conjecture to the potentially semistable case, the map $\cyc$ should be
  extended to representations of the form $\sigma(\tau) \otimes L_\lambda$ for $\tau$ not
  necessarily semisimple using a formula like that of Proposition~4.3.  Then we would conjecture
  that equation~\eqref{eqn:bmlp} continues to hold.
\end{remark}

\subsection{Ihara avoidance} \label{sec:ihara-avoidance}
  We explain the relation between our results and the Ihara avoidance deformations of
  \cite{Taylor2008-AutomorphyII}.  Suppose that $l > n$, that $q \equiv 1 \mod l$, and that
  $\rhobar$ is trivial.  In this case let $\tau_{ps}$ be the tame inertial type for which the
  eigenvalues of a generator of tame inertia are distinct $l$th-roots of unity, and for $P$ a
  partition of $n$ let $\tau_P$ be the unipotent inertial type corresponding to $P$.  Then one finds
  \begin{equation}\label{eq:ps-red} \red(\sigma(\tau_{ps})) = \sum_P m(P,(1^n))\red(\sigma(\tau_P)).\end{equation}
  Combining this with Proposition~\ref{prop:type-mults} and using properties of the Kostka numbers,
  Theorem~\ref{thm:BM-Gln} shows that
  \[Z(R^\square(\rhobar,\tau_{ps})\otimes \FF) = \sum_P \binom{n}{P}Z(R^\square(\rhobar,
    \tau_P)\otimes \FF)\]
  where $\binom{n}{P}$ is the multinomial coefficient; in particular $R^\square(\rhobar,
  \tau_{ps})\otimes \FF$ is highly non-reduced.  We prove this formula locally in
  section~\ref{sec:towards-local-proof}.

  The other ingredient of Ihara avoidance is that $R^\square(\rhobar,\tau_{ps})$ is irreducible,
  which must be verified by other means (by showing that the generic fibre is smooth and connected).
  Granted this, the level-changing method can be described as follows.  To simplify matters, imagine
  that we are in a global setting in which patched modules $H_\infty$ can be defined as modules over
  $R^\square(\rhobar)$.  If $\sigma$ is a representation of $GL_n(\Oc_F)$ over $E$ or $\FF$, let
  $Z_{\aut}(\sigma) = Z(H_\infty(\sigma))$.  Then $Z_{\aut}$ is additive and compatible with
  reduction modulo $l$, and we always have an inequality of
  cycles \begin{equation}\label{eq-ineq}Z_{\aut}(\sigma) \leq \cyc(\sigma)\end{equation} if $\sigma$
  is defined over $E$.

  Suppose that some $Z_{\aut}(\sigma(\tau_Q))$ is non-zero.  Then $Z_{\aut}(\bar{\sigma(\tau_Q)})$
  is non-zero and so $Z_{\aut}(\sigma(\tau_{ps}))$ is non-zero by equation~\eqref{eq:ps-red}.  Since
  $\cyc(\sigma(\tau_{ps})) = Z(R^\square(\rhobar, \tau_{ps}))$ is irreducible, by
  inequality~\eqref{eq-ineq} we must have
  $Z_{\aut}\left(\sigma(\tau_{ps})\right) = \cyc(\sigma(\tau_{ps}))$.  But now
  Theorem~\ref{thm:BM-Gln} and equation~\eqref{eq:ps-red} imply that
  \[\sum_P m(P,(1^n)) Z_{\aut}\left(\bar{\sigma(\tau_P)}\right) = \sum_P m(P, (1^n))\bar{\cyc}\left(\bar{\sigma(\tau_P)}\right).\]
  Since every $m(P, (1^n))$ is non-zero, this together with \eqref{eq-ineq} implies that, for
  \emph{every} $P$,
  \[Z_{\aut}(\sigma(\tau_P)) = \cyc(\sigma(\tau_P)).\] This is exactly `change of level': we started
  by assuming that the globalisation of $\rhobar$ was automorphic of type $\tau_Q$ and deduced that
  it is automorphic of type $\tau_P$ for every $P$.  Note that this argument is circular with our
  global proof of Theorem~\ref{thm:BM-Gln}, since that theorem relies on non-minimal modularity
  lifting theorems that in turn depend on \cite{Taylor2008-AutomorphyII}, but is valid with the
  local proof of section~\ref{sec:towards-local-proof}.

\section{Global proof.}
\label{sec:global-proof}

For the entirety of this section, we assume that $l > 2$.

\subsection{Automorphic forms}
\label{sec:autom-forms-galo}

We define the spaces of automorphic forms on definite unitary groups that we will patch using the
Taylor--Wiles--Kisin method.  See also \cite{ClozelHarrisTaylor2008-Automorphy},
\cite{JMJ:9091284}, \cite{MR2941425},
\cite{Thorne2012-AutomorphySmallResidualImage}.  Our reason for reproducing this now standard
material here is that we need to allow more general level at places $v \nmid l$ than is considered in
those references; hopefully it will be clear that there is no essential difference.

\subsubsection{}

Let $L$ be an imaginary CM field with maximal totally real subfield $L^+$ satisfying the following hypotheses:
\begin{enumerate}
\item $[L^+:\QQ]$ is divisible by 4;
\item $L/L^+$ is unramified at all finite places;
\item every place $v \mid l$ of $L^+$ splits in $L$.
\end{enumerate}

Let $\delta_{L/L^+}$ be the non-trivial character of $G_L$ that is trivial $G_{L^+}$, and let $c$ be the non-trivial element of $\Gal(L^+/L)$.  Then as in
\cite{Thorne2012-AutomorphySmallResidualImage} section 6 (see also
\cite{ClozelHarrisTaylor2008-Automorphy} section 3.3), we may choose a group scheme $G$ over
$\Oc_{L^+}$ and an $L^+$-linear involution $*$ on $M_n(L)$ such that:
\begin{itemize}
\item $(xy)^* = y^*x^*$ for all $x,y \in M_n(L)$ and $x^* = x^c$ for $x \in Z(M_n(L))\cong L$;
\item for any
  $L^+$-algebra $R$, \[G(R) = \{ g \in M_n(L)\tensor_{L^+}R : g^*g=1\};\]
\item for every finite place $v$ of $L^+$, $G\times_{L^+}{L^+_v}$ is quasi-split;
\item for every infinite place $v$ of $L^+$, $G(L_v^+)\cong U_n(\RR)$, the compact unitary group;
\item there is a maximal order $A \subset M_n(L)$ with $A^* = A$ and $G(\Oc_{L^+}) = G(L^+) \cap A$;
\item for $v$ a finite place of $L^+$ split as $ww^c$ in $L$ there is an isomorphism 
\[\iota_w : G(L^+_v) \rarrow GL_n(L_w)\]
such that $\iota_w(G(\Oc_{L^+_v})) = GL_n(\Oc_{L_w})$ and $\iota_{w^c}(x) = ({}^t\iota_w(x)^c)^{-1}$.
\end{itemize}

Let $S$ be a set of finite places of $L^+$ split in $L$, and let $S_l$ be the set of places of $L^+$
above $l$.  Suppose that $S \cap S_l = \emptyset$ and write $T = S \cup S_l$.  Suppose that $U$ is a
subgroup of $G(\AA^\infty_{L^+})$, and write $U_v$ for the image of the projection of $U$ to
$G(L^+_v)$.  For each $v \in T$ we choose a place $\tilde{v}$ of $L$ above $v$, and let $\tilde{S}$,
$\tilde{S}_l$, $\tilde{T}$ be the sets of $\tilde{v}$ for $v$ in $S$, $S_l$ and $T$ respectively.
Call $U$ \textbf{good} if it is compact and if:
\begin{itemize}
\item for $v \in T$, $U_v \subset G(\Oc_{L^+_v})$;
\item for some $v \in S$, $U_v$ contains no elements of finite order $l$.
\end{itemize}

\subsubsection{}

For $v \in S$, let $M_v$ be an $\Oc$-module with an $\Oc$-linear action of $G(\Oc_{L^+_v})$ which is
continuous for the discrete topology on $M_v$.

Suppose that $E$ contains the images of all embeddings $L \into \bar{E}$.  Let $I_l =
\Hom(L^+,E)$, so that $I_l$ surjects onto $S_l$ with $\theta \in I_l$ mapping to a place
$v(\theta)$ of $L^+$ above $l$.  Let $\tilde{I_l}$ be the set of embeddings $L \into E$
inducing a place of $\tilde{S_l}$.  Restriction to $L^+$ defines a bijection $\tilde{I_l}
\isomto I_l$.  Now let
\[\ZZ^n_{+} = \{(\lambda_1,\ldots \lambda_n)\in\ZZ^n : \lambda_1 \geq \ldots \geq \lambda_n\}.\]
As in \cite{MR2941425} definition~2.2.1, we associate to each $\lambda \in \ZZ^n_+$ a representation
$\xi_\lambda$ (defined over $\Oc$) of $GL_n/\Oc$, and let $M_\lambda = \xi_\lambda(\Oc)$ and
$V_\lambda = \xi_\lambda(E)$.

Now suppose that $\lambda = (\lambda_\theta)_\theta \in (\ZZ_+^n)^{\tilde{I_l}}$ is a tuple of
elements of $\ZZ_+^n$ indexed by $\theta \in \tilde{I}_l$.  Then we define
\[M_\lambda = \tensor_\theta M_{\lambda_\theta}\] and regard this as a representation of $\prod_{v
  \in S_l} G(\Oc_{L^+_v})$ via the product of the composites of the maps
\[G(\Oc_{L^+_{v(\theta)}}) \xrightarrow{\iota_{\tilde{v}(\theta)}} GL_n(\Oc_{L_{\tilde{v}(\theta)}})
\xrightarrow{\;\theta\;} GL_n(\Oc) \xrightarrow{\xi_{\lambda_{\theta}}}
GL(M_{\lambda_\theta}).\]

Finally let $M = \bigotimes_{v\in S} M_v$, a representation of $\prod_{v \in S} G(\Oc_{L^+_v})$ and
hence (by projection) of any good subgroup $U$; we also consider the representation $M \tensor M_\lambda$ of $\prod_{v
  \in T}G(\Oc_{L^+_v})$ and hence of $U$.  

\begin{definition}
  Suppose that $U$ is a good \emph{open} subgroup.  Then $S_\lambda(U,M)$ is the space of functions
\[f : G(L^+) \backslash G(\AA_{L^+}^\infty) \rarrow M\tensor M_\lambda\]
such that $f(gu) = u^{-1}f(g)$ for all $u \in U$.

If $V$ is a good subgroup of $G(\AA_{L^+}^\infty)$ then define 
\[S_\lambda(V,M) = \varinjlim S_\lambda(U,M)\]
where the limit runs over all good open subgroups $U$ containing $V$.
\end{definition}

If $M$ is a finitely generated $\Oc$-module then $S_\lambda(U,M)$ is a finitely generated $\Oc$-module,
because $G(L^+)\backslash G(\AA_{L^+}^\infty)/U$ is a finite set.
\begin{lemma} \label{lem:level-changing} Let $U$ be a good open subgroup.
 \begin{enumerate}
 \item  The functor \[(M_v)_{v\in T} \longmapsto S_\lambda(U,\bigotimes_v M_v)\] is exact.
 \item If $V \subset U$ is a normal, good, open subgroup, then there are isomorphisms of
   $\Oc$-modules
     \[ S_\lambda(V,M) \longrightarrow S_\lambda(U,M) \tensor_\Oc \Oc[U/V]\] and
     \[ S_\lambda(V,M)_{U/V} \xrightarrow{\tr_{U/V}} S_\lambda(U,M)\]      
     where $S_\lambda(V,M)_{U/V}$ denotes the $U/V$-coinvariants in $S_\lambda(V,M)$.
\end{enumerate}
\end{lemma}
\begin{proof}
  This may be proved by the argument of \cite{Thorne2012-AutomorphySmallResidualImage} Lemmas~6.3 and 6.4, using the assumption that
  $U$ has no elements of order $l$.
\end{proof}

\subsubsection{Hecke operators} \label{sec:hecke-operators} Now suppose that $U = U_SU^S$ where
$U_S \subset \prod_{v \in S}G(\Oc_{L^+_v})$ and $U^S = \prod_{v\not\in S} U_v$ where $U_v \subset
G(L^+_v)$ for each $v \not \in S$.  Suppose also that $S \cap S_l = \emptyset$ and that for finite
places $v \not \in T$ of $L^+$ split in $L$ we have $U_v = G(\Oc_{L^+_v})$.
We define Hecke operators, following \cite{MR2941425}, section 2.3.

\begin{definition} 
  \begin{enumerate}
  \item Let $v\not \in T$ be a place of $L^+$ splitting as $ww^c$ in $L$.  Then
    for $1 \leq j \leq n$ define the operator $T_w^{(j)}$ on $S_\lambda(U,M)$ as the double coset operator:
    \[T^{(j)}_w = \left[U \iota_w^{-1}\twomat{\varpi_w
      1_j}{0}{0}{1_{n-j}}U\right]\]
    for some (any) choice of uniformiser $\varpi_w$ of $\Oc_{L_w}$, where $1_j$ is the $j \times j$
    identity matrix.

  \item For $v \in S_l$, $w$ a place of $L$ above $v$, and $\varpi_w \in \Oc_{L_w}$ a uniformiser,
    define:
    \[T^{(j)}_{\lambda,\varpi_w} = \left((w_0\lambda) \twomat{\varpi_w1_j}{0}{0}{1_{n-j}}\right)^{-1}\left[U\iota_w^{-1}\twomat{\varpi_w1_j}{0}{0}{1_{n-j}}U\right]\]
    where $w_0\lambda$ is the conjugate of $\lambda$ by the longest element $w_0$ of the Weyl group.
  \end{enumerate}
\end{definition}

Let $\TT^T$ be the polynomial ring over $\Oc$ generated by all the $T^{(j)}_w$ and
$(T_w^{(n)})^{-1}$, and $\tilde{\TT}^T$ the polynomial ring over $\TT$ generated by all
$T^{(j)}_{\lambda,\varpi_w}$ for $w \in S_l$, $0 \leq j \leq n$.  Let $\TT^{T}_\lambda(U,M)$ and
$\tilde{\TT}^T_\lambda(U,M)$ be, respectively, the images of $\TT^T$ and $\tilde{\TT}^T$ in
$\End(S_\lambda(U,M))$.

\subsubsection{Ordinary forms} (see \cite{MR2941425} section~2.4) Say that a maximal ideal $\mf$ of
$\tilde{\TT}_\lambda^{T}(U,M)$ is \textbf{ordinary} if each $T^{(j)}_{\lambda,\varpi_w}$ has
non-zero image in $\tilde{\TT}_\lambda^{T}(U,M)/\mf$.
\begin{definition}
  Let \[S^{\ord}(U,M) = \prod_{\mf\, \ord}S(U,M)_{\mf}\] where the product is over ordinary maximal
  ideals $\mf$ of $\tilde{\TT}_{\lambda}^T(U,M)$.

  Let \[\TT^{T,\ord}_\lambda(U,M) = \mathrm{im}( \TT^T \rarrow S^{\ord}(U,M)).\]
\end{definition}
There is an idempotent $e_0 \in \tilde{\TT}_{\lambda}^T(U,M)$ such that $S^{\ord}(U,M) = e_0
S(U,M)$, and the formation of $e_0$ is compatible with changing $U$ away from places above $l$ or
changing $M$.

\subsubsection{Base change}
\label{sec:base-change}

Keep the assumptions of the previous section, and suppose also that $U_v$ is a hyperspecial maximal
compact subgroup of $G(L^+_v)$ for each place $v$ of $L^+$ inert in $L$ and that $M$ is a finite
free $\Oc$-module.  Suppose that $\Ac$ is the space of (complex-valued) automorphic forms on
$G(\AA_{L^+})$ and that $\pi = \bigotimes'_v \pi_v$ is an irreducible constituent of $\Ac$ with
weight $\lambda_\infty \in (\ZZ^n_+)^{\Hom(L^+,\CC)}$ such that (recalling the fixed isomorphism
$\iota :\bar{E} \isomto \CC$):
\begin{itemize}
\item for $\theta \in \Hom(L^+,\CC)$, $(\lambda_\infty)_\theta = \lambda_{\iota \circ \theta}$;
\item for $v \not \in S$ a place of $L^+$, $\pi_v^{U_v} \neq 0$;
\item for $v \not \in S$ a place of $L^+$ split as $ww^c$ in $L$, $T_w^{(j)}$ acts as a scalar
  $\iota(a^{(j)}_w)$ for some $a^{(j)}_w \in \bar{E}$.
\end{itemize}

Let $f_\pi : \TT^T \rarrow \bar{E}$ be the homomorphism taking $T_w^{(j)}$ to $a^{(j)}_w$.  

\begin{lemma} \label{lemsec:base-change}
  Suppose that $\pi$, $U$, $\lambda$ and $M$ satisfy the above hypotheses.  Then we have the formula:
    \[ \dim \left(S_\lambda(U,M) \otimes_{\TT^T,f_\pi} \bar{E}\right) = \dim \Hom_{U_S}\left((M \otimes_{\Oc}
  \bar{E})^\vee, \bigotimes_{v \in S}\pi_v\otimes_{\CC,\iota^{-1}}\bar{E}\right).\]
\end{lemma}

\begin{proof} (sketch) As in \cite{ClozelHarrisTaylor2008-Automorphy} Proposition~3.3.2, we have:
  \[S_{\lambda}(U,M)\otimes_{\Oc,\iota} \CC \cong \Hom_{U_S\times
    G(L^+_\infty)}\left((M\otimes_{\Oc,\iota}\CC)^{\vee}\otimes V_{\infty}^\vee, \Ac^{U^S}\right)\]
  where $V_\infty$ is an algebraic representation of $G(L^+_\infty)$ constructed from
  $\lambda_\infty$.  It suffices to show that any other irreducible $\pi' \subset \Ac$ satisfying
  the above three conditions (for the same values of $a^{(j)}_w$) is actually equal to $\pi$.  By
  \cite{MR2856380} corollaire 5.3, such $\pi$ and $\pi'$ have base changes $\Pi$ and $\Pi'$ to
  $GL_n(\AA_{L})$ such that for each place $w$ of $L$ above a place $v$ of $L^+$, $\Pi_w$ is the
  local base change of $\pi_w$.  By strong multiplicity one for $GL_n(\AA_L)$, $\Pi_w = \Pi'_w$ for
  each place $w$ of $L$.  Since each place of $S$ is split in $L$ and $\pi$ and $\pi'$ are assumed
  $U_v$-spherical at places $v \not \in S$, we deduce that $\pi \cong \pi'$ as representations of
  $G(\AA_{L^+})$.  But by \cite{MR2856380} Th\'{e}or\`{e}me 5.4, $\pi$ appears with multiplicity one
  in $\Ac$, so that $\pi = \pi'$.
\end{proof}

In a similar vein, suppose that $\Pi$ is a regular algebraic conjugate self-dual cuspidal automorphic
representation of $GL_n(\AA_L)$ which is unramified outside of places dividing $S$, and let $U^S$
be as above.  The following is also a consequence of \cite{MR2856380} Th\'{e}or\`{e}me 5.4, Corolllaire
5.3, and strong multiplicity one:

\begin{lemma}\label{lem:descent} There is an automorphic representation $\pi$ of $G(\AA_{L^+})$
  such that, at each finite place $v\not\in S$, $\pi_v^{U_v} \neq 0$ and $\Pi_v$ is the spherical
  base change of $\pi_v$ (relative to our chosen hyperspecial maximal compact $U_v$, if $v$ is
  inert). \qed
\end{lemma}

\subsection{Galois representations}
\label{sec:galo-repr}

\subsubsection{} Recall some notation from \cite{BLGGT2014-PotentialAutomorphy}
section 1.1.  Let $\Gc_n$ be the algebraic group $(GL_n \times GL_1) \rtimes \{1,j\}$ where
$j(g,a)j^{-1} = (a {}^tg^{-1},a)$, $\Gc^0_n$ be the connected component $GL_n\times GL_1$ of
$\Gc_n$, and $\nu :\Gc_n \rarrow GL_1$ be defined by $\nu(g,a) = a$, $\nu(j) = -1$.  If $\Gamma$ is
a group, $\Delta$ is an index 2 subgroup, and $\rho : \Gamma \rarrow \Gc_n(A)$ is a representation
(for some ring $A$) such that $\rho^{-1}(\Gc_n^0(A)) = \Delta$, then let $\breve{\rho}$ be the
composition of $\rho|_{\Delta}$ with the projection $\Gc_n^0(A) \rarrow GL_n(A)$.

\subsubsection{Ordinary deformations}
\label{sec:ordin-deform}

Suppose that $k/\QQ_l$ is a finite extension and that $E$ contains the images
of all embeddings $k \into \bar{E}$.  If $\lambda \in (\ZZ^n_+)^{\Hom(k,E)}$, and $\bar{r}_l : G_k
\rarrow GL_n(\FF)$ is a continuous representation, denote by $R^\square_{\lambda,\crord}(\bar{r}_l)$ the
ring called $R^{\triangle_\lambda,\rom{cr}}$ in \cite{MR2941425}.

\begin{proposition}\label{prop:ord-rings}
  The scheme $\Spec R^\square_{\lambda,\crord}(\bar{r}_l)$ is reduced, $\Oc$-flat and
  equidimensional of relative dimension $[k:\QQ_l]\tfrac{n(n-1)}{2} + n^2$ over $\Oc$ (if it is
  non-zero).  The $\bar{E}$-points of $\Spec R^\square_{\lambda,\crord}(\bar{r}_l)[1/l]$ are those
  $\bar{E}$-points $x$ of $\Spec R^\square(\bar{r}_l)[1/l]$ such that the associated Galois
  representation $r_{l,x}$ is ordinary of weight $\lambda$ (in the sense of
  \cite{MR2941425} Definition~3.3.1) and crystalline.
\end{proposition}

\begin{proof}
  This can all be found in \cite{MR2941425}~section 3.3.
\end{proof}

\begin{lemma} \label{lem:ordinary-rep}
  If $k = \QQ_l$, $\bar{r}_l$ is trivial and $\lambda = ((l-2)(n-1),(l-2)(n-2),\ldots,(l-2),0)$, then
  $R^\square_{\lambda,\crord}(\bar{r}_l)$ is geometrically integral and non-zero.
\end{lemma}
\begin{proof}
  The representation $V = \bigoplus_{i=1}^n \Oc(-(i-1)(l-1))$ is a lift of $\bar{r}_l$ such that
  $V \tensor E$ is crystalline and ordinary of weight $\lambda$.  By \cite{MR2941425}
  Lemma~3.4.3, \[\Spec\left(R^\square_{\lambda,\crord}(\bar{r}_l)[1/l]\right)\] is irreducible, and
  in fact the proof of that Lemma shows that it is geometrically irreducible.  Since
  $R^\square_{\lambda,\crord}(\bar{r}_l)$ is $\Oc$-flat and reduced,
  $R^\square_{\lambda,\crord}(\bar{r}_l)$ is geometrically integral, as required.
\end{proof}

\subsubsection{Global deformations}
\label{sec:global-deformations}

Suppose that $l'$ is a prime, $L_v /\QQ_{l'}$ is a finite extension, $\bar{r}_v : G_{L_v}\rarrow GL_n(\FF)$ is a continuous
representation and $\Cc_v$ is a finite set of irreducible components of $\Spec R^\square(\bar{r}_v)$ (if
$l' \neq l$) or of $R^\square_{\lambda,\crord}(\bar{r}_v)$ for some $\lambda$ (if $l'=l$).  Then by
\cite{BLGGT2014-PotentialAutomorphy} Lemma~1.2.2, 
$\Cc_v$ determines a local deformation problem for $\bar{r}_v$.

We recall some notation for global deformation problems from
\cite{ClozelHarrisTaylor2008-Automorphy}, section 2.3.  Suppose that $L$, $L^+$, $T$ and $\tilde{T}$ are
as above and that:
\begin{itemize}
\item $\rhobar : G_{L^+} \rarrow \Gc_n(\FF)$ is a continuous representation, unramified outside
  $T$, with~$\bar{\rho}^{-1}(\Gc_n^0(\FF)) = G_L$; 
\item $\mu : G_{L^+} \rarrow \Oc^\times$ is a continuous lift of $\nu \circ \bar{\rho}$; 
\item for each $v \in T$, $\Cc_v$ is a non-empty set of components of
  $R^\square(\breve{\bar{\rho}}|_{G_{L_{\tilde{v}}}})$ (if $v\nmid l$) or of some
  $R^\square_{\lambda_v,\crord}(\breve{\bar{\rho}}|_{G_{L_{\tilde{v}}}})$ (if $v\mid l$).
\end{itemize}
Then the data \[\Sc = (L/L^+,T,\tilde{T},\Oc,\bar{\rho},\mu,\{\Cc_v\}_{v\in T})\] determines a
deformation problem for $\bar{\rho}$; if $\breve{\bar{\rho}}$ is absolutely irreducible, then there
is a universal deformation ring $R^{\univ}_{\Sc}$ and universal deformation \[r_{\Sc}^{\univ}:G_{L^+} \rarrow
\Gc_n(R^{\univ}_\Sc)\] of type $\Sc$, defined in \cite{ClozelHarrisTaylor2008-Automorphy} section 2.3.

\begin{proposition} \label{prop:global-dimension}
  If $\mu(c_v) = -1$ for all $v \mid \infty$ (where $c_v$ is complex conjugation associated to $v$)
  then
\[\dim R^{\univ}_\Sc \geq 1.\]
\end{proposition}
\begin{proof}
  This follows from \cite{ClozelHarrisTaylor2008-Automorphy} Corollary 2.3.5 and the dimension
  formulae for the $\Cc_v$; see \cite{BLGGT2014-PotentialAutomorphy} Proposition~1.5.1.
\end{proof}

Define also the $T$-framed deformation ring $R^{\square_T}_\Sc$ as in
\cite{ClozelHarrisTaylor2008-Automorphy} Proposition~2.2.9; it is an algebra over
$\widehat{\bigotimes}_{v\in T}R^\square_{\Cc_v}$ where $R^\square_{\Cc_v}$ is the quotient of
  $R^\square(\breve{\bar{\rho}}|_{G_{L_{\tilde{v}}}})$ corresponding to $\Cc_v$.
 
  \subsubsection{} Now let $L$, $\lambda$, $T$, $U$ and $M$ be as in
  section~\ref{sec:hecke-operators}, and suppose that $M$ is finitely generated as an $\Oc$-module.
  Suppose that $\mf$ is a non-Eisenstein maximal ideal of $\TT^{T, \ord}_\lambda(U,M)$.
\begin{proposition}\label{prop:galois-reps} There is a unique continuous
  homomorphism
\[r_\mf : G_{L^+,T} \rarrow \Gc_n(\TT^{T,\ord}_\lambda(U,M)_\mf)\]
such that 
\begin{enumerate}
\item $r_\mf^{-1}(\Gc^0_n(\TT^{T,\ord}_\lambda(U,M)_\mf)) = G_{L,T}$;
\item $\nu \circ r_\mf = \epsilon^{1-n}\delta^n_{L/L^+}$;
\item if $v \not \in T$ splits as $ww^c$ in $L$, then $r_\mf(\Frob_w)$ has characteristic
  polynomial \[\sum_{j=0}^n (-1)^j \Nm(w)^{j(j-1)/2}T_w^{(j)}X^{n-j};\] \item for each $v \in S_l$,
  $r_{\mf}|_{G_{L_v}}$ factors through $R^\square_{\lambda,\crord}
  (\bar{r}_\mf|_{G_{L_{\tilde{v}}}})$. \end{enumerate} 
\end{proposition} 
\begin{proof} Suppose first that $M$ is finite free as an $\Oc$-module.  Then the construction of
  $r_\mf$ is standard (see \cite{ClozelHarrisTaylor2008-Automorphy} Proposition~3.4.4).  The first
  three properties are deduced as in that reference, and the final property is proved as in
  \cite{MR2941425} Lemma~3.3.4.
  
In general, note that $M$ admits a surjection from an $\Oc[U]$-module $P$ that is finite free as an
$\Oc$-module.  Indeed, the action of $U$ on $M$ factors through a finite quotient $\bar{U}$ of $U$,
and we may take $P$ to be the projective envelope of $M$ as an $\Oc[\bar{U}]$-representation.  Then
there will be a $\tilde{\TT}^T$--equivariant surjection $S_\lambda(U,P) \onto S_\lambda(U,M)$ inducing a
surjection $\TT^{T,\ord}_\lambda(U,P)_{\mf} \onto \TT^{T,\ord}_{\lambda}(U,M)_{\mf}$.  The Galois
representation for $M$ is then the representation for $P$ composed with this surjection.  \end{proof}

\subsection{Realising local representations globally}
\label{sec:real-local-repr}

Recall that we have a representation $\bar{\rho} : G_F \rarrow GL_n(\FF)$.  The aim of this section
is to globalise $\bar{\rho}$, as in Proposition~\ref{suitable-globalisation} below.  We follow
\cite{JMJ:9091284} Appendix A closely, and the reader wishing to follow the arguments will need to
have that paper to hand.  Note that in \cite{JMJ:9091284} the residue characteristic of the
coefficient field is called $p$, whereas here it is called $l$.

\subsubsection{Adequacy} \label{sec:adequacy}

Thorne, in \cite{Thorne2015-2adiclifting} Definition 2.20) has modified the definition of adequacy
from that in \cite{Thorne2012-AutomorphySmallResidualImage} to allow some cases where $l \mid n$ ---
the definitions coincide if $l \nmid n$.  Let us repeat the new definition here:

\begin{definition} \label{def:adequacy} Let $V$ be a finite-dimensional vector space over $\bar{\FF}$.  A
  subgroup $H \subset GL(V)$ is \textbf{adequate} if it acts irreducibly on $V$ and if:
  \begin{enumerate}
  \item $H^1(H,\bar{\FF}) = 0$;
  \item $H^1(H,\End(V)/\bar{\FF}) = 0$ where $H$ acts on $\End(V)$ by conjugation and $\bar{\FF}$ is the
    subspace of scalar endomorphisms;
  \item For each simple $\bar{\FF}[H]$-submodule $W \subset \End(V)$, there is a semisimple element
    $\sigma \in H$ with an eigenvalue $\alpha \in \FF$ such that $\tr e_{\sigma,\alpha}W\neq 0$,
    where $e_{\sigma,\alpha}$ is the projection onto the $\alpha$-eigenspace of $\sigma$.
  \end{enumerate}
\end{definition}

With this definition, the main theorems of \cite{Thorne2012-AutomorphySmallResidualImage} (Theorems
7.1, 9.1, 10.1 and 10.2) continue to hold, by \cite{Thorne2015-2adiclifting} Corollary 7.3.  

\begin{lemma} \label{lem:adequacy} Let $GL_n.2$ be the smallest algebraic subgroup of $GL_{2n}$
  containing the block diagonal matrices of the form $(g,{}^tg^{-1})$ and a matrix $J$ such that
  $J(g,{}^tg^{-1})J^{-1} = ({}^tg^{-1},g)$.  Then for $m$ sufficiently large, both \[(GL_n.2)(\FF_{l^m})\subset
  GL_{2n}(\bar{\FF}_l)\] and \[GL_n(\FF_{l^m}) \subset GL_n(\bar{\FF_l})\] are adequate.  In
  other words, Lemma~A.1 of \cite{JMJ:9091284} continues to hold with the revised definition of adequate.
\end{lemma}
\begin{proof}
  This is a consequence of \cite{1405.0043} Theorem~11.5, remembering our running assumption that $l>2$.
\end{proof}

\subsubsection{} The main result of this section is:

\begin{proposition} \label{suitable-globalisation} There is an imaginary CM field $L$ with maximal
  totally real subfield $L^+$, and there are continuous representations
\[\bar{r} : G_{L^+} \rarrow \Gc_n(\bar{\FF})\]
and \[r : G_{L^+} \rarrow \Gc_n(\bar{E})\] satisfying the following hypotheses:
\begin{enumerate}
\item $r$ is a lift of $\bar{r}$;
\item $\bar{r}^{-1}(\Gc_n^0(\bar{\FF})) = G_L$;
\item $\breve{r}$ is of the form $r_{l,\iota}(\pi,\chi)$ for a regular algebraic, cuspidal,
  polarized automorphic representation $(\pi,\chi)$ (see \cite{BLGGT2014-PotentialAutomorphy},
  Theorem~2.1.1 for the notation $r_{l,\iota}$);
\item $\breve{\bar{r}}(G_{L(\zeta_l)}) = GL_n(\FF_{l^m})$ for $m$ large enough that the conclusion
  of Lemma~\ref{lem:adequacy} holds (in particular, $\breve{\bar{r}}(G_{L^+(\zeta_l)})$ is adequate);
\item $\nu \circ r = \epsilon^{1-n} \delta_{L/L^+}^n$ and similarly for $\bar{r}$ (note that this
  determines $\chi$);
\item Every place $v$ of $L^+$ dividing $lp$ splits completely in $L$;
\item For each place $v$ of $L^+$ dividing $p$, there is an isomorphism $L^+_v \cong F$ and
  a place $\tilde{v}$ of $L$ dividing $v$ such that $\breve{\bar{r}}|_{G_{L_{\tilde{v}}}} \cong
  \bar{\rho}$;
\item For each place $v$ of $L^+$ dividing $l$, we have that $L^+_v = \QQ_l$ and there is a
  place $\tilde{v}$ of $L$ dividing $v$ such that $\breve{\bar{r}}|_{G_{L_{\tilde{v}}}}$ is trivial and $\breve{r}|_{G_{L_{\tilde{v}}}}$ is
  ordinary of weight $\lambda$ for $\lambda$ as in Lemma~\ref{lem:ordinary-rep};
\item $\bar{L}^{\ker \bar{r}}$ does not contain $L(\zeta_l)$;
\item if $v$ is a place of $L^+$ not dividing $lp$, then $\bar{r}$ and $r$ are
  unramified at $v$;
\item $[L^+ : \QQ]$ is divisible by 4, and $L/L^+$ is unramified at all finite places.

\end{enumerate}
\end{proposition}

We will prove this over the course of the next three lemmas.  The first step is to realise $\bar{r}$
as the local component of some (not yet automorphic) representation $\bar{r}$, using
\cite{Calegari2012-EvenFontaineMazurII} Proposition~3.2.

\begin{lemma} \label{lem:globalisation}
 There exist a CM field $L_1$ with maximal totally
real subfield $L_1^+$ and a continuous representation $\bar{r}:G_{L_1^+} \rarrow
\Gc(\bar{\mathbb{F}})$ satisfying properties 2 and 4--11 of Proposition~\ref{suitable-globalisation}
(at least as they pertain to $\bar{r}$).
\end{lemma}
\begin{proof}
  This is a straightforward modification of the proof of \cite{JMJ:9091284} Proposition~A.2 to
  include conditions on $L_1$ and $\bar{r}$ at places dividing $p$.
\end{proof}

Now we show that $\bar{r}$ is potentially automorphic over some CM extension $L/L_1$.  This
basically follows the proof of Proposition~3.3.1 of \cite{BLGGT2014-PotentialAutomorphy}, making
modifications to control the splitting in $L$ of places of $L_1$ above $l$ and $p$ (as in
\cite{JMJ:9091284}).  The first step is to show that this $\bar{r}$ lifts to a characteristic zero
representation with good properties.

\begin{lemma}
  Let $\bar{r}$ be as in Lemma~\ref{lem:globalisation}.  Then there is a continuous representation $r:G_{L_1^+} \rarrow
  \Gc_n(\bar{\QQ}_l)$ lifting $\bar{r}$ satisfying all of the properties of Proposition
  \ref{suitable-globalisation} except possibly automorphy (property 3).
  \label{lem:lifting}
\end{lemma}
\begin{proof}
  This is proved in \cite{BLGGT2014-PotentialAutomorphy}, Proposition~3.2.1, under the hypothesis
  that $l \geq 2n+1$.  We examine the proof of that proposition and show that in our case we may
  remove the hypothesis on $l$.  The only way in which this hypothesis is used is to verify, using
  Proposition~2.1.2 of that paper, the adequacy of the image of the induction of $\breve{\bar{r}}$
  from $G_{L_1(\zeta_l)}$ to $G_{L^+_1(\zeta_l)}$. However, by property 4 of Proposition
  \ref{suitable-globalisation} we can use Lemma~\ref{lem:adequacy} instead of
  \cite{BLGGT2014-PotentialAutomorphy} Proposition~2.1.2.  (Note that Theorems 9.1 and 10.2 of
  \cite{Thorne2012-AutomorphySmallResidualImage} remain true with this definition, and so in
  \cite{BLGGT2014-PotentialAutomorphy}, Theorems 2.4.1 and 2.4.2, and hence also Proposition~3.2.1
  remain true.)
\end{proof}

\begin{lemma}
  \label{lem:automorphic-globalisation} 
  There is a CM extension $L/L_1$, linearly disjoint from $\bar{L_1}^{\ker \bar{r}}(\zeta_l)$ over
  $L_1$, such that every place of $L_1$ dividing $lp$ splits completely in $L$ and such that
  $L$ and $\breve{r}|_{G_{L}}$ satisfy all the properties required in Proposition
  \ref{suitable-globalisation}.  In particular, Proposition~\ref{suitable-globalisation} is true.
\end{lemma}

\begin{proof} The proof of \cite{JMJ:9091284} Proposition~A.6 goes through
  with the following modifications -- we temporarily adopt the notation of their proof to indicate
  what must be changed.  The field $(L')^+$ must be chosen so that, for each place $v \mid p$ of
  $(L')^+(\zeta_N)^+$, there is a point $P_v \in \tilde{T}((L')^+(\zeta_N)^+_v)$.  The field extension
  $F^+/L^+$ can then be chosen so that all the places of $L^+$ above $p$ split completely (as well
  as all those above $l$).  Finally, instead of using Theorem~4.2.1 of
  \cite{BLGGT2014-PotentialAutomorphy} we use Theorem~2.4.1 of that paper, which applies by our
  assumption that $r$ is ordinary.
\end{proof}

\subsection{Patching}
\label{sec:patching}
\subsubsection{} Let $\bar{r} : G_{L^+} \rarrow \Gc_n(\bar{\FF})$ be the representation provided by
Proposition~\ref{suitable-globalisation}, and (enlarging $E$ if necessary) assume that $\bar{r}$ is
valued in $\Gc_n(\FF)$.  Thus $\bar{r}$ is the reduction modulo $\lambda$ of the Galois
representation $r_{l,\iota}(\pi,\chi)$ associated to some regular algebraic polarized cuspidal
automorphic representation $(\pi,\chi)$ of $GL_n(\AA_L)$.  Use property 9 of
Proposition~\ref{suitable-globalisation} and the Chebotarev density theorem to choose a place $v_1$
of $L^+$ such that $v_1$ splits in $L$, the residue field of $L^+_{v_1}$ has order $\neq 1 \mod l$,
$\bar{r}$ is unramified at $v_1$, and $\ad(\bar{r}(\Frob_{v_1})) = 1$.  Then every lift of
$\bar{r}|_{G_{L_{\tilde{v}_1}}}$ is unramified, so that
$R^\square(\breve{\bar{r}}|_{G_{L_{\tilde{v}_1}}})$ is equal to the unramified deformation ring, and
is in particular formally smooth.  Take $S$ to be the set of places of $L^+$ dividing $p$
together with the place $v_1$, and recall that $T = S \cup S_l$ and $\tilde{T}$ is a choice of a
place $\tilde{v}$ of $L$ above each $v \in T$.  Let $\lambda \in (\ZZ_n^+)^{\tilde{I}_l}$ have all
components equal to the weight in Lemma~\ref{lem:ordinary-rep}.  Let $U = \prod_v U_v$ where:
\begin{itemize}
\item for $v$ a place of $L^+$ split in $L$, $U_v = G(\Oc_{L^+_v})$;
\item for $v$ a place of $L^+$ inert in $L$, $U_v\subset GL_n(L^+_v)$ is a hyperspecial
  maximal compact subgroup;
\item for $v = v_1$, $U_{v_1}$ is the preimage under $\iota_{\tilde{v}_1}$ of the Iwahori subgroup
  of $GL_n(\Oc_{L^+_v})$.
\end{itemize}
Then the assumptions on $v_1$ imply that $U_{v_1}$ has no $l$-torsion and so $U$ is good.

For $v \in T$ a place of $L^+$ dividing $\tilde{v} \in \tilde{T}$, let $R_{\tilde{v}}$ be:
\begin{itemize}
\item $R_{\tilde{v}} = R^\square(\breve{\bar{r}}|_{G_{L_{\tilde{v}}}})$ if $v \in S$;
\item $R_{\tilde{v}} = R^\square_{\lambda,\crord}(\breve{\bar{r}}|_{G_{L_{\tilde{v}}}})$ if $v \in S_l$.
\end{itemize}
Let $R^{\loc} = \widehat{\bigotimes}_{v \in T} R_{\tilde{v}}$.

There is a global deformation problem
\[\Sc = (L/L^+,T,\tilde{T},\Oc,\bar{r}, \epsilon^{1-n}\delta_{L/L^+}^n,\{R_{\tilde{v}}\}_{v\in
    T})\] with universal deformation $r^{\univ}_{\Sc}: G_{L^+,T} \rarrow \Gc_n(R^{\univ}_{\Sc})$.
Let $f_\pi:\TT^T\rarrow \Oc$ be the homomorphism such that $\iota \circ f_\pi(T^{(j)}_w)$
gives the eigenvalue of $T^{(j)}_w$ acting on $\pi_w^{GL_n(\Oc_{L_w})}$ via $\iota_w$ for
$w$ above a split place of $L^+$ not in $T$.

\subsubsection{} Let $U_p = \prod_{v \mid p} G(\Oc_{L^+_v})$ and $U_S = \prod_{v \in S} U_v =
U_pU_{v_1}$.  Let $\Rc$ be the category of smooth representations of $U_p$ on finitely generated
$\Oc$-modules and let $\Rc^f$ be the category of smooth representations of $U_p$ on finite-length
$\Oc$-modules.  If $\sigma \in \Rc$ then let $M_\sigma$ be the underlying module of $\sigma$
regarded as a representation of $U_S$ by letting $U_p$ act through $\sigma$ and $U_{v_1}$ act
trivially. We define an $R^{\univ}_\Sc$-algebra $\TT(\sigma)$ and a $\TT(\sigma)$-module $H(\sigma)$ by:
\begin{itemize}
\item $\TT(\sigma) = \TT^{T,\ord}_\lambda(U,M_\sigma)_\mf$ with the $R^{\univ}_\Sc$-algebra structure provided by
  Proposition~\ref{prop:galois-reps};
\item $H(\sigma) = S^{\ord}_\lambda(U,M_\sigma)_\mf$.
\end{itemize}
Note that $\TT^{T,\ord}_\lambda(U,M)_{\mf} \neq 0$ whenever $S(U,M_{\sigma})$ contains an eigenform on
which $\TT^T$ acts through $f_{\pi}$, by property (8) of Proposition~\ref{suitable-globalisation} and
Lemma~5.2.1 of \cite{MR2941425}.

\subsubsection{} By Lemma~\ref{lem:adequacy} and property (4) of Proposition
\ref{suitable-globalisation}, $\breve{\bar{r}}|_{G_{L(\zeta_l)}}$ is adequate.  We follow the proof
of \cite{Thorne2012-AutomorphySmallResidualImage} Theorem~6.8, but apply the ordinary projector
$e_0$ to everything --- this makes no difference.  Using Proposition~4.4 of
\cite{Thorne2012-AutomorphySmallResidualImage}, we obtain an integer
$r \geq [L^+:\QQ]\frac{n(n-1)}{2}$ and, for each $N \geq 1$, a set $Q_N$, disjoint from $T$, of $r$
finite places of $L^+$ split in $L$ and a set $\tilde{Q}_N$ of choices of places of $L$ above those
of $Q_N$.  As in \cite{Thorne2012-AutomorphySmallResidualImage}, for each $N$ and each $\sigma$ we
can find rings $R^{\univ}_{N}$ and $R^{\square_T}_N$, an $R^{\univ}_N$-algebra $\TT_N(\sigma)$ and a
finitely generated $\TT_N(\sigma)$-module $H_N(\sigma)$ enjoying the following properties:
\begin{itemize}
\item There is an isomorphism $R^{\univ}_N \hat{\otimes} \Oc[[y_1,\ldots,y_{n^2\# T}]]\cong R^{\square_T}_N$.
\item For each $v \in Q_N$, $\Nm v \equiv 1 \mod l^N$.  Let $\Delta_N$ be the maximal
  $l$-power-order quotient of $\kappa(\tilde{v})^\times$, where $\kappa(\tilde{v})$ is the residue
  field of $\tilde{v}$, and let $\af_N$ be the augmentation ideal
  in the group ring $\Oc[\Delta_N]$.
\item There are natural homomorphisms $\Oc[\Delta_N]\rarrow R^{\univ}_N$ and $\Oc[\Delta_N] \rarrow
  \End(H_N(\sigma))$ such that the composite $R^{\univ}_N \rarrow \TT_N(\sigma)\rarrow \End(H_N(\sigma))$ is
  an $\Oc[\Delta_N]$-algebra homomorphism.
\item With the above $\Oc[\Delta_N]$-algebra structures, there are natural isomorphisms $R^{\univ}_N
  / \af_N \isomto R^{\univ}_\Sc$, $\TT_N(\sigma)/\af_N \isomto \TT(\sigma)$, and $H_N(\sigma)/\af_N
  \isomto H(\sigma)$ (this relies on Lemma~\ref{lem:level-changing}).
\item The map $\Oc[\Delta_N] \rarrow R^{\univ}_N \rarrow \TT_N(\sigma)$ makes $H_N(\sigma)$ into a
  finite free $\Oc[\Delta_N]$-module.
\item We may and do choose a surjective $\Oc$-algebra homomorphism \[R^{\loc}[[z_1,\ldots,z_g]] \onto R^{\square_T}_N\]
  where $g = r - [L^+:\QQ]\frac{n(n-1)}{2}$.
\item The functor $\sigma \mapsto H_N(\sigma)$ is a covariant exact functor from $\Rc$ to the
  category of finitely generated $R^{\univ}_N$-modules.
\end{itemize}
\begin{remark}
  Strictly speaking, the proof in \cite{Thorne2012-AutomorphySmallResidualImage} that
  \[R^{\univ}_N\rarrow \End(H_N(\sigma))\] is an $\Oc[\Delta]$-algebra homomorphism, and the
  construction of the isomorphism \[H_N(\sigma)/\af_N\rarrow H(\sigma),\] require that $\sigma$ be
  finite free as an $\Oc$-module (to apply Propositions 5.9 and 5.12 in that paper).  However, we
  can remove this constraint by writing $\sigma$ as a quotient of a $U_p$-representation that
  \emph{is} finite free as an $\Oc$-module, as in the proof of Proposition~\ref{prop:galois-reps}.
\end{remark}
Write $H^{\square_T}_N(\sigma) = H_N(\sigma) \otimes_{R^{\univ}_N} R^{\square_T}_N$.  We pick isomorphisms \[R^{\square_T}_N \isomto R^{\univ}_N \hat{\otimes} \Oc[[y_1,\ldots,y_{n^2\#
  T}]]\] and \[R^{\square_T}_\Sc \isomto R^{\univ}_{\Sc} \hat{\otimes} \Oc[[y_1,\ldots,y_{n^2\# T}]]\]
compatible with reduction modulo $\af_N$.  Let 
\[R_\infty = R^{\loc}[[z_1,\ldots,z_g]]\]
and 
\begin{align*}S_\infty &= (\varprojlim \Oc[\Delta_N]) \hat{\otimes} \Oc[[y_1,\ldots,y_{n^2\# T}]] \\
&\cong \Oc[[x_1,\ldots,x_r,y_1,\ldots, y_{n^2 \# T}]]
\end{align*}
and note that (by Theorem~\ref{thm:complete-intersection} and Proposition~\ref{prop:ord-rings}) we have 
\begin{align*}\dim R_\infty &= 1 + n^2 \# T + [L^+:\QQ]\frac{n(n-1)}{2} + r - [L^+:\QQ]\frac{n(n-1)}{2} \\
 &= \dim S_\infty.\end{align*}
Write $\af$ for the kernel of the map $S_\infty \rarrow \Oc$ taking $x_i$ and $y_i$ to zero.  Thus
$R^{\square_T}_N/\af \isomto R^{\univ}_\Sc$ and $H^{\square_T}_N(\sigma)/\af \isomto H(\sigma)$.

\subsubsection{} 

We patch the modules $H_N^{\square_T}(\sigma)$ following the proof of the sublemma in \cite{MR2748398}, Theorem
3.6.1. Pick representations $\sigma_1,\sigma_2,\ldots$ such that each of the countably many
isomorphism classes in $\Rc^f$ is represented by exactly one $\sigma_i$.  For $h \in \NN$, let
$\Rc^f_{\leq h}$ be the full subcategory of $\Rc^f$ whose objects are $\sigma_1,\ldots, \sigma_h$.

Choose a strictly increasing sequence $(h(N))_N$ of positive integers.  Let $\cf_N = \ker ( S_\infty \rarrow
\Oc[\Delta_N] \hat{\otimes} \Oc[[y_1,\ldots,y_{n^2\#T}]])$ and choose a sequence $\bfrak_1 \supset
\bfrak_2 \supset \ldots$ of open ideals of $S_\infty$ such that $\bfrak_N \supset \cf_N$ for all $N$
and $\bigcap_N \bfrak_N = (0)$.  Choose also open ideals $\df_1 \supset \df_2 \supset \ldots $ of
$R^{\univ}_\Sc$ with $\bfrak_N R^{\univ}_{\Sc} + \ker(R^{\univ}_\Sc \rarrow \TT(\sigma)) \supset
\df_N \supset \bfrak_N R^{\univ}_\Sc$ for all $\sigma \in \Rc^f_{\leq h(N)}$ and $\bigcap_N \df_N =
(0)$.  

Define a \textbf{patching datum of level $N$} to be:

\begin{itemize}
\item a surjective $\Oc$-algebra homomorphism \[\phi : R_\infty \onto R^{\univ}_\Sc/\df_N;\]
\item a covariant, exact functor $\Mc_N$ from $\Rc^f_{\leq h(N)}$ to the category of $R_\infty \hat{\otimes}
  S_\infty$-modules that are finite free over $S_\infty/\bfrak_N$;
\item for $\sigma \in \Rc^f_{\leq h(N)}$, functorial isomorphisms of $R_\infty$-modules
  \[\Mc_N(\sigma) / \af \isomto H(\sigma)/\bfrak_N\] (the right hand side being an $R_\infty$-module via $\phi$).
\end{itemize}

Since $S_\infty/\bfrak_N$, $R^{\univ}_\Sc/\df_N$, $H(\sigma)/\df_N$ are finite sets and the sets of
objects and morphisms in $\Rc^f_{\leq h(N)}$ are finite, there are only finitely many patching data
of level $N$.  Note that if $N' \geq N$ then from any patching datum of level $N'$ we can get one of
level $N$ by reducing modulo $\bfrak_N$ and $\df_N$ and restricting $\Mc_{N'}$ to $\Rc^f_{h(N)}$.

For each pair of integers $M \geq N \geq 1$ define a patching datum $D(M,N)$ of
level $N$ by taking: 
\begin{itemize}
\item $\phi: R_\infty \onto R^{\square_T}_N \onto R/\df_N$ where the first map is our chosen
  presentation of $R^{\square_T}_N$ over $R^{\loc}$ and the second is induced by $R^{\square_T}_N/\af
  \isomto R^{\univ}_\Sc$;
\item $\Mc_N(\sigma) = H_M^{\square_T}(\sigma)/\bfrak_N$, which is finite free over
  $S_\infty/\bfrak_N$ and is an $R_\infty$-module via $R_\infty \onto R^{\square_T}_M \onto
  \TT^{\square_T}_M$ (clearly $\Mc_N$ is a functor);
\item the isomorphism $\psi : \Mc_N/\af \isomto H(\sigma)/\bfrak_N$ coming from the natural isomorphism
  $H_M^{\square_T}(\sigma) / \af \isomto H(\sigma)/\bfrak_N$.
\end{itemize}
 
Since there are only finitely many isomorphism classes of patching datum of each level $N$, we may
choose an infinite sequence of pairs $(M_j,N_j)_{j \geq 1}$ with $M_j \geq N_j$, $M_{j+1} > M_j$ and
$N_{j+1} > N_j$ such that $D(M_{j+1},N_{j+1})$ reduces to $D(M_j,N_j)$ for each $j$.  We may
therefore define a functor $H_\infty$ from $\Rc^f$ to the category of $R_\infty \hat{\otimes}
S_\infty$-modules by the formula: \[ H_\infty(\sigma) = \varprojlim_{j}
H_{M_j}^{\square_T}(\sigma)/\bfrak_{N_j}\] (and extending to the whole of $\Rc^f$ by picking an
isomorphism from each object to one of the $\sigma_i$).  Note that the terms in the limit are
defined for $j$ sufficiently large.  Extend $H_\infty$ to $\Rc$ by setting $H_\infty(\varprojlim
\sigma_i) = \varprojlim H_\infty(\sigma_i)$.

\subsubsection{} We need to verify that $H_\infty$ has the properties needed for the proof of
Theorem~\ref{thm:BM-Gln}.  The functor $H_\infty$ is exact and covariant, and for all $\sigma$ we have
\[H_\infty(\sigma \otimes_\Oc \FF) = H_\infty(\sigma) \otimes \FF\] (these statements all follow from
the corresponding statements at finite level).   

\begin{lemma}\label{lem:support} 
  For each $\sigma$, the support $\supp_{R_\infty}(H_\infty(\sigma))$ is a union of irreducible
  components of $\Spec R_\infty$.
\end{lemma}
\begin{proof}
  We may factor the map $S_\infty \rarrow \End_{R_\infty}(H_\infty(\sigma))$ through a map $S_\infty
  \rarrow R_\infty$ (since we may do this at finite level by definition of the action of
  $S_\infty$).  So we have a map $S_\infty \rarrow R_\infty$ and a finitely generated
  $R_\infty$-module $H_\infty(\sigma)$ that is finite free over the regular local ring $S_\infty$.  Thus we have:
  \begin{align*}
    \depth_{R_\infty}(H_\infty(\sigma)) & \geq \depth_{S_\infty}(H_\infty(\sigma)) \\
    & = \dim S_\infty\\
    & = \dim R_\infty\\
    & \geq \depth_{R_\infty}(H_\infty(\sigma)).
  \end{align*}
  Therefore by \cite{Taylor2008-AutomorphyII}, Lemma~2.3, $\supp_{R_\infty}(H_\infty(\sigma))$ is a union of
  irreducible components of $\Spec R_\infty$.
\end{proof}

The argument of the next lemma goes back to \cite{MR1440309}:

\begin{lemma} \label{lem:diamond}
  Let $\qf$ be a prime ideal of $R_\infty$ such that $(R_\infty)_\qf$ is regular.  Then
  $H_\infty(\sigma)_\qf$ is finite free over $(R_\infty)_\qf$.
\end{lemma}

\begin{proof}
  We may suppose that $\qf \in \supp_{R_\infty} H_\infty$.  Since $(R_\infty)_\qf$ is regular, it is
  a domain.  By the previous lemma, $(R_\infty)_\qf$ acts faithfully on $(H_\infty(\sigma))_\qf$.  Thus
  $(R_\infty)_\qf$ is finite over $(S_\infty)_{S_\infty \cap \qf}$.  The argument of the previous
  lemma now shows that 
  \[\depth_{(R_{\infty})_\qf} (H_\infty(\sigma)_\qf) =\depth (R_\infty)_\qf.\]

  The module $H_\infty(\sigma)_\qf$ has finite projective dimension over $(R_\infty)_\qf$ and the
  Auslander--Buchsbaum formula holds:
  \[ \depth_{(R_\infty)_\qf} (H_\infty(\sigma)_\qf) +
  \rom{pd}_{(R_\infty)_\qf}(H_\infty(\sigma)_\qf) = \depth(R_\infty)_\qf.\]
  Therefore $H_\infty(\sigma)_\qf$ is a finitely generated projective $(R_\infty)_\qf$-module as required.
\end{proof}

\subsubsection{}

Assume now that $\Oc$ is large enough that every irreducible component of $R^\square(\rhobar)$ is
geometrically integral.  Note that, by Lemma~\ref{lem:ordinary-rep} and the fact that
$R_{\tilde{v}_1}$ is formally smooth, $R_\infty$ is a completed tensor product of the ring
\[\widehat{\bigotimes_{v\mid p}}R_{\tilde{v}}\] with a geometrically
integral, $\Oc$-flat ring
\[A = \widehat{\bigotimes}_{v \in T,v\nmid p}R_{\tilde{v}}[[z_1,\ldots,z_g]]\] in
$\Cc_{\Oc}^\wedge$.  Then, by \cite{MR2827723} Lemma~3.3, giving a minimal prime of $R_\infty$ is
the same as giving a minimal prime of each
\[R_{\tilde{v}}=R^\square(\breve{\bar{r}}|_{G_{L_{v}}}) \equiv R^\square(\rhobar).\]
\begin{proposition}\label{prop:multiplicities}
  Let $\sigma \in \Rc$ be finite free as an $\Oc$-module and of the form
  $\otimes_{v \mid p} \sigma_v$ for representations $\sigma_v$ of $U_v\cong GL_n(\Oc_F)$.  For each
  place $v$ of $L^+_v$ above $p$ let $\pf$ be a minimal prime of $R_\infty$ and let $\pf_v$ be its
  pre-image in the copy of $R^\square(\rhobar)$ corresponding to $v$.  Each $\pf_v$ is a
  minimal prime of $R^\square(\rhobar,\tau_v)$ for a unique inertial type $\tau_v$.

  Then $H_\infty(\sigma)/\pf$ is generically free of rank
  \[ n!\prod_{v\mid p} m((\sigma_v\otimes\bar{E})^\vee, \tau_v)\] over $R_\infty/\pf$.
\end{proposition}

\begin{proof}
  Let $\Sc'$ be the deformation problem
  \[(L/L^+,T,\tilde{T},\Oc,\bar{r},\epsilon^{1-n}\delta_{L/L^+}^n,\{R'_{\tilde{v}}\}_{v\in
    T})\] where $R_{\tilde{v}}' = R_{\tilde{v}}$ unless $v | p$, in which case $R_{\tilde{v}}' =
  R_{\tilde{v}}/\pf_v$.  Then $R^{\univ}_{\Sc'}$ is a quotient of $R^{\univ}_{\Sc}$.
  
  By Proposition~\ref{prop:global-dimension}, \[\dim R^{\univ}_{\Sc'} \geq 1.\]
  
  By \cite{Thorne2012-AutomorphySmallResidualImage}, Theorem~10.2, $R^{\univ}_{\Sc'}$ is a finite
  $\Oc$-module; it therefore admits an $\Oc$-algebra homomorphism
  \[x:R^{\univ}_{\Sc'} \rarrow \Oc'\] for a finite extension $\Oc'/ \Oc$; enlarging $E$, we may
  assume that $\Oc' = \Oc$.  There is a corresponding representation
  $r' : G_{L^+} \rarrow \Gc_n(\Oc)$.  By \cite{Thorne2012-AutomorphySmallResidualImage},
  Theorem~9.1, $r'$ is the representation attached to some regular algebraic polarized cuspidal
  automorphic representation $(\pi',\delta_{L/L^+}^n)$ of $GL_n(\AA_{L})$ with
  $\pi' = \otimes_v \pi'_v$, such that $\pi'_v$ is unramified and $\iota$-ordinary (see
  definition~5.12 of \cite{MR2941425}) for $v \mid l$.  By Lemmas \ref{lemsec:base-change} and
  \ref{lem:descent} we see that the fibre of $H(\sigma)$ at $x$ (for any
  $\sigma = \bigotimes_{v \mid p} \sigma_v$) has dimension:
  \begin{align*}&\dim\Hom_{U_S}\left((\sigma\otimes \bar{E})^\vee,\bigotimes_{v \in
      S}\pi'_v\otimes_{\CC,\iota^{-1}}\bar{E}\right)  \\
    =\quad &\dim
    \pi_{\tilde{v}_1}^{U_{v_1}}\prod_{v \mid p}\dim
    \Hom_{U_v}\left((\sigma_v \otimes\bar{E})^\vee, \pi'_{\tilde{v}}\otimes_{\CC,\iota^{-1}}\bar{E}\right)\\
    =\quad &n!\prod_{v \mid p}m((\sigma_v \otimes \bar{E})^\vee, \tau_v).
    \end{align*}
    To see the last equality note that, for each $v \mid p$, $\pi'_{\tilde{v}}\otimes_{\CC,\iota^{-1}}\bar{E}$ is a
    local component of a unitary cuspidal automorphic representation of $GL_n(\AA_L)$, and so generic by
    \cite{MR0348047}, and by local--global compatibility (see, for instance, \cite{BLGGT2014-PotentialAutomorphy}
    Theorem~2.1.1) it has type $\tau_v$.  The factor of $n!$ is the contribution from the Iwahori invariants in the
    unramified principal series representation $\pi'_{\tilde{v}_1}$.
  
    Now choose an $\Oc$-point $\tilde{x}$ of $\Spec R^{\infty}$ above $x$.  As $\tilde{x}$ is (in
    the terminology of Proposition~\ref{prop:fixed-type}) a non-degenerate point of each factor
    $R^\square_v$ of $R_\infty$, we see that $\Spec R_\infty$ is formally smooth at $\tilde{x}$.  By
    Lemma~\ref{lem:diamond}, we see that $H_\infty(\sigma)_{\tilde{x}}$ is free over
    $(R_\infty)_{\tilde{x}}$.  To determine the rank, note that $H_\infty(\sigma)_{\tilde{x}}/\af =
    H(\sigma)_x$, and applying the above calculation we get the proposition.
\end{proof}

\begin{corollary} \label{cor:patch=cycle} Identify $\Zc(R_\infty)$ with $\bigotimes_{v\mid p} \Zc(R^\square(\rhobar))$ using
  Lemma~\ref{lem:product}.  If $\sigma \in \Rc$ is finite free as an $\Oc$-module, then
  \[Z(H_\infty(\sigma)) = n!\cyc^{\otimes d}(\sigma).\]
\end{corollary}

\begin{proof}
  It suffices to prove this for $\sigma$ of the form $\bigotimes_{v \mid p} \sigma_v$.  If $\pf$ is
  a minimal prime of $R_\infty$, corresponding to minimal primes $\pf_v$ of $R^\square(\rhobar)$ of
  inertial type $\tau_v$, then by Proposition~\ref{prop:multiplicities} the coefficient of $[\pf]$
  in $Z(H_\infty(\sigma))$ is
  \[n!\prod_{v \mid p}m((\sigma_v \otimes \bar{E})^\vee, \tau_v).\]
   As $m((\sigma_v \otimes \bar{E})^\vee, \tau_v)$ is the multiplicity of $\pf_v$ in
   $\cyc(\sigma_v)$, we obtain the required formula.
\end{proof}

We have therefore shown that $H_\infty$ has all the properties needed for the proof of Theorem
\ref{thm:BM-Gln}.  To be specific, in the notation of that proof we take $d$ equal to the number of
places $v$ of $L^+$ dividing $p$, $c = n!$, and $A = \widehat{\bigotimes}_{v \in T,v\nmid
  p}R_{\tilde{v}}[[z_1,\ldots,z_g]]$.

\begin{remark}
  The reasons we work with ordinary automorphic forms are the following:
  \begin{enumerate}
  \item We can ensure that the local deformation rings $R_{\tilde{v}}$ for $v \mid l$ are
    (geometrically) irreducible, which is necessary for the argument.  This could be difficult to
    arrange with low weight crystalline deformation rings if $l \leq n$, as then Fontaine--Lafaille
    theory would break down.
  \item In globalising $\rhobar$ and in arguing that every component of $R_\infty$ is automorphic we
    can appeal to the ordinary automorphy lifting theorem Theorem~2.4.1 of
    \cite{BLGGT2014-PotentialAutomorphy} (which is Theorem~9.1 of
    \cite{Thorne2012-AutomorphySmallResidualImage}), which only requires $l > 2$ (once the new
    definition of adequacy is used) rather than Theorem~4.2.1 of
    \cite{BLGGT2014-PotentialAutomorphy} which requires $l \geq 2n + 1$, as well as a potential
    diagonalizability assumption.
  \end{enumerate}
\end{remark}
 
\section{$K$-types.}
\label{sec:k-types}
We construct representations $\sigma(\tau)$ satisfying the conclusion of Theorem~\ref{thm:K-types},
and compute their reduction modulo $l$ and their multiplicities in generic representations of
$GL_n(F)$.  Such representations were already constructed in
\cite{SchneiderZink1999-KTypesTemperedComponentsGeneralLinear}, and our construction follows theirs
closely with minor modifications to make things work modulo $l$.  It seems likely that the two
constructions yield the same representations $\sigma(\tau)$ but we have not tried to prove it.  The
multiplicity formula, Corollary~\ref{cor:multiplicity-calculation}, could be shown by our methods to
hold with either construction.

We outline the contents of this section as an aid to the reader.  Bushnell--Kutzko theory (recalled
in sections~\ref{sec:simple-chars}--\ref{sec:covers}) provides various compact open subgroups
$J^1 \subset J$ inside $K = GL_n(\Oc_F)$ such that $J^1$ is pro-$p$ and $J/J^1$ is a finite general
linear group, and representations $\kappa$ of $J$ such that $\kappa |_{J^1}$ is irreducible.  Then
the $K$-types are constructed as $\Ind_J^K (\kappa \otimes \nu)$ for irreducible representations
$\eta$ of $J / J^1$, at least for those Bernstein components with only one representation in their
supercuspidal support (the general case requires `$G$-covers').  The key constructions are
Definition~\ref{defn:sigma-SZ} which constructs the representations and
Definition~\ref{def:SZ-datum} which relates them to inertial types. In section~\ref{sec:parabolic}
we apply the functor $\Hom_J(\kappa, \cdot)$ to reduce the calculation of multiplicities in
parabolic inductions for $GL_n(F)$ --- Theorem~\ref{thm:types-multiplicities} --- to a calculation
with finite general linear groups, which is worked out in
sections~\ref{sec:symmetric-groups}--\ref{sec:finite-gener-line}.  In
section~\ref{thm:type-reduction-gln} we show that the reduction mod $l$ of types is controlled by
the reduction mod $l$ of representations of finite general linear groups, which is essentially
tautologous given our construction and Theorem~\ref{thm:irreducibility}.

\subsection{Symmetric groups}
\label{sec:symmetric-groups}

If $P \in \Part$ with $\deg P = n$, for each $i \in \NN$ let $T_{P,i}$ be the subset 
\[\left\lbrace 1+\sum_{j=1}^{i-1}P(j), 2 + \sum_{j=1}^{i-1}P(j), \ldots, \sum_{j=1}^iP(j)\right\rbrace\]  of
$\{1,\ldots,n\}$.  Let $S_P$ be the subgroup of $S_n$ stabilising each $T_{P,i}$, so that
$S_P = \prod_{i}S_{P(i)}$.  Let \[\pi_P^\circ = \Ind_{S_P}^{S_n}(\sgn)\] where $\sgn$ is the sign
representation.
\begin{definition}
  Let $\sigma_P^\circ$ be the unique irreducible representation of $S_n$ that appears in
  $\pi_P^\circ$ and that appears in no $\pi_{P'}^\circ$ for $P'\succ P$ (see the proposition
  \cite{SchneiderZink1999-KTypesTemperedComponentsGeneralLinear}~\S3) .
\end{definition}
Every irreducible representation is of the form
$\sigma_P^\circ$ for a unique $P$.  Note that this is not the standard association of
representations of $S_n$ to partitions, but rather its twist by the sign representation.

\begin{definition}\label{def:kostka}
  The \textbf{Kostka number} $m(P,P')$ is the multiplicity with which $\sigma_P^\circ$ appears
  in $\pi_{P'}^\circ$.  
\end{definition}
We adopt the conventions that if $\deg P \neq \deg P'$ then $m(P,P') = 0$, while if $\deg P = \deg
P' = 0$ then $m(P,P')=1$.  Thus $m(P,P') > 0 $ if and only if $P \succeq P'$, and if $P = P'$ then
$m(P,P') = 1$.  This \emph{does} coincide with the standard definition of Kostka numbers.

\subsection{PSH-algebras}
\label{sec:psh-algebras}

For our calculations of multiplicities we will require the notion of a PSH-algebra, due to
Zelevinsky \cite{MR643482}; see also chapter 3 of \cite{1409.8356} and the proof of the Proposition in
\cite{SchneiderZink1999-KTypesTemperedComponentsGeneralLinear} section 4.  

\begin{definition}
  A \textbf{positive self-adjoint Hopf} (or \textbf{PSH-}) \textbf{algebra} is a graded connected Hopf
  algebra \[R = \bigoplus_{n \geq 0} R_n\] over $\ZZ$, with multiplication $m: R\otimes R \rarrow R$
  and comultiplication $\mu: R \rarrow R\otimes R$, together with a $\ZZ$-basis $\Sigma$ of
  homogeneous elements with the following property: let $\langle \cdot,\cdot\rangle$ be the
  $\ZZ$-bilinear form on $R$ making $\Sigma$ an orthonormal basis. Then for all
  $\sigma_1,\sigma_2,\sigma_3 \in \Sigma$ we have
  \[\langle m(\sigma_1\otimes\sigma_2),\sigma_3\rangle =
  \langle\sigma_1\otimes\sigma_2,\mu(\sigma_3)\rangle > 0.\]
\end{definition}

Suppose that $R$ is a PSH-algebra, with notation as in the definition.  An element $\sigma \in R$ is
\textbf{primitive} if $\mu(\sigma) = \sigma \otimes 1 + 1 \otimes \sigma$.  Say that $R$ is
\textbf{indecomposable} if there is a unique primitive element in $\Sigma$.  The basic structure theorem is then:

\begin{theorem} (\cite{MR643482} Theorem~2.2 and Theorem~3.1g) Let $R$ be a PSH-algebra and $\Sigma$ its distinguished basis.  For
  each primitive $\sigma \in \Sigma$ there is an indecomposable sub--PSH-algebra $R(\sigma)$ of $R$
  such that \[\bigotimes_\sigma R(\sigma) \isomto R\] is an isomorphism of PSH-algebras.\footnote{If there
    are infinitely many primitive elements of $\Sigma$, this
    should be interpreted as the direct limit of the tensor products over finite subsets of the
    primitive elements in $\Sigma$.}
  
  If $R$ and $R'$ are indecomposable PSH-algebras then, after rescaling the gradings so that each
  has a primitive element of degree one, there are precisely two isomorphisms of PSH-algebras
  between $R$ and $R'$.
\end{theorem}

We can obtain an indecomposable PSH-algebra $R^S$ from the representation theory of the symmetric
group as follows: let $R^S_n$ be the Grothendieck group of representations of $S_n$, and take
$\Sigma$ to be the subset of isomorphism classes of irreducible representations.  The multiplication
is given by induction: if $\sigma_1$ and $\sigma_2$ are irreducible representations of degrees
$S_{n_1}$ and $S_{n_2}$ then \[m(\sigma_1\otimes \sigma_2) = \Ind_{S_{n_1} \times S_{n_2}}^{S_{n_1 +
    n_2}} (\sigma_1 \otimes \sigma_2)\] regarded as an element of the Grothendieck group.  Similarly
the comultiplication is given by restriction: if $\sigma$ is a representation of $S_n$ then
\[\mu(\sigma) = \sum_{a + b = n}\Res^{S_n}_{S_a \times S_b}\sigma\]
where we have identified the Grothendieck group of representations of $S_a \times S_b$ with the
tensor product of those of $S_a$ and $S_b$.  That this (with the obvious unit and counit) is a Hopf
algebra is an exercise using Mackey's theorem (see \cite{1409.8356} Corollary~4.26), and the
self-adjointness property is a consequence of Frobenius reciprocity.  The unique primitive element
is the trivial representation of the trivial group. The non-identity isomorphism $R^S \rarrow R^S$
takes the trivial representation of any $S_n$ to the sign representation.

\subsection{Finite general linear groups}
\label{sec:finite-gener-line}

Let $k$ be a finite field, $n \geq 1$ be an integer, and $\bar{G} = GL_n(k)$.  For all unsupported
assertions in this subsection see
\cite{SchneiderZink1999-KTypesTemperedComponentsGeneralLinear}~\S4.
\begin{definition}
  Let $\bar{\Ic}_0$ be the union over all $d$ of the set of isomorphism classes of cuspidal
  representations of $GL_d(k)$.  
  Let $\bar{\Ic}$ be the set of functions $\bar{\Pc}:\bar{\Ic}_0 \rarrow \Part$ with finite support.
\end{definition}
The degree $\deg \bar{\Pc}$ of an element of $\bar{\Ic}$ is defined to be the sum \[\sum_{\sigma
  \in\bar{\Ic}_0}\deg(\bar{\Pc}(\sigma)) \dim \sigma.\] Every irreducible representation of
$\bar{G}$ has a cuspidal support, a function $\bar{\Sc} :\bar{\Ic_0} \rarrow \NN_{\geq 0}$ with
$\sum_{\sigma \in\bar{\Ic}_0}\bar{\Sc}(\sigma) \dim \sigma = n$.  For each such $\bar{\Sc}$, let
$\Omega_{\bar{\Sc}}$ be the full subcategory of $\Rep_{\bar{E}}(\bar{G})$ whose objects are
representations all of whose irreducible constituents have cuspidal support $\bar{\Sc}$.

If $\sigma$ is a cuspidal representation of $GL_d(k)$ and $t$ is a positive integer, then define the
generalised Steinberg representation $\St(\sigma, t)$ to be the unique non-degenerate irreducible
representation of $GL_{dt}(k)$ whose cuspidal support is $t$ copies of $\sigma$.  If $\bar{\Pc}\in \bar{\Ic}$
with $\deg \bar{\Pc} = n$, define a Levi subgroup $\bar{M}_{\bar{\Pc}}$ of $\bar{G}$ by
\[ \bar{M}_{\bar{\Pc}} = \prod_{\sigma \in \bar{\Ic}_0,i\in \NN}\bar{G}_{\bar{\Pc}(i)\dim \sigma}.\]
\begin{definition}
  Let $\St(\bar{\Pc})$ be the irreducible representation of $\bar{M}_{\bar{\Pc}}$ whose tensor
  factors are the $\St(\sigma,\bar{\Pc}(\sigma)(i))$ for each $(\sigma,i)$.
\end{definition}

Choose a parabolic subgroup $\bar{Q}$ with Levi factor $\bar{M}_{\bar{\Pc}}$ and let 
\[\pi_{\bar{\Pc}} = \Ind_{\bar{Q}}^{\bar{G}} \St(\bar{\Pc}).\]
\begin{definition}
  Let $\sigma_{\bar{\Pc}}$ be the unique irreducible representation contained in $\pi_{\bar{\Pc}}$
  that is not contained in $\pi_{\bar{\Pc}\,'}$ for any $\bar{\Pc}\,' \succ \bar{\Pc}$.
\end{definition}
\begin{proposition}
  Every irreducible representation of $\bar{G}$ is of the form $\sigma_{\bar{\Pc}}$ for a unique
  $\bar{\Pc}$. \qed
\end{proposition}

Let $R^{GL}$ be the PSH-algebra defined by taking the $d$th graded piece $R^{GL}_d$ to be the
Grothendieck group of representations of $GL_d(k)$, defining multiplication via parabolic induction,
comultiplication via Jacquet restriction, and taking $\Sigma$ to be the set of isomorphism classes
of irreducible representation (see \cite{1409.8356} \S4 for details).  The primitive elements in
$\Sigma$ are the cuspidal representations; for each cuspidal representation $\sigma$ of some
$GL_d(k)$ let $R(\sigma)$ be the PSH-subalgebra of $R^{GL}$ spanned by those elements of $\Sigma$
having cuspidal support some number of copies of $\sigma$.  Then we have (see the proof of the
Proposition in \cite{SchneiderZink1999-KTypesTemperedComponentsGeneralLinear} section 4):

\begin{proposition}\label{prop:psh-Rgl}
  The PSH-algebras $R(\sigma)$ are indecomposable and there is an isomorphism of PSH-algebras
  \[R^{GL} = \bigotimes_{\sigma \in \bar{\Ic}_0}R(\sigma).\] For each cuspidal representation
  $\sigma$ there is (after rescaling the gradings) a unique isomorphism of PSH-algebras $R(\sigma)
  \isomto R^S$ that takes $\St(\sigma,t)$ to the sign representation of $S_t$ for all $t$. \qed
\end{proposition}

\begin{corollary}\label{cor:finite-gl-mults}
  If $\bar{\Pc},\bar{\Pc}\,' \in \bar{\Ic}$ both have degree $n$, then the multiplicity
  \[m(\bar{\Pc},\bar{\Pc}\,')\defeq \dim \Hom_{\bar{G}}(\sigma_{\bar{\Pc}},\pi_{\bar{\Pc}\,'})\]
is equal to the product of Kostka numbers
\[\prod_{\sigma \in \bar{\Ic}_0}m(\bar{\Pc}(\sigma),\bar{\Pc}\,'(\sigma)).\]\end{corollary}
\begin{proof} First, observe that the bilinear form on $R^{GL}$ is given (on homogeneous elements of
  the same degree $n$ in the $\NN$-span of $\Sigma$) by $\dim \Hom_{\bar{G}}(-,-)$. Thus we can
  read off $m(\bar{\Pc},\bar{\Pc}\,')$ from the PSH-algebra structure on $R^{GL}$.  By
  Proposition~\ref{prop:psh-Rgl}, we can reduce to the case where $\bar{\Pc}$ and $\bar{\Pc}\,'$ are
  both supported on the same cuspidal representation $\sigma$; let $P$ and $P'$ be
  $\bar{\Pc}(\sigma)$ and $\bar{\Pc}\,'(\sigma)$ respectively.  Then it is easy to see that, under
  the isomorphism $R(\sigma) \isomto R^S$ of Proposition~\ref{prop:psh-Rgl}, $\sigma_{\bar{\Pc}}$ is
  taken to $\sigma^\circ_P$ and $\pi_{\bar{\Pc}\,'}$ is taken to $\pi_{P'}^\circ$.  The formula
  follows.
\end{proof}

\subsection{Simple characters}
\label{sec:simple-chars}
We recall a little of the theory of Bushnell and Kutzko (for which see
\cite{BushnellKutzko1993-AdmissibleDualGLN}, \cite{MR1643417},
\cite{BushnellKutzko1999-SemisimpleTypesGLN}). Let $C$ be an algebraically closed field of
characteristic distinct from $p$ (for the case when $C$ has positive characteristic we refer to the
works of Vign\'{e}ras \cite{Vigneras1996-RepresentationsLModulaires},
\cite{Vigneras1998-InducedRRepresentationsPAdicReductiveGroups} and M\'{i}nguez, S\'{e}cherre, and Stevens
\cite{MR3273486}, \cite{secherrestevens15}; we will not require much from the positive
characteristic theory).

Let $V$ be a vector space over $F$, let $G =
\Aut_F(V)$, and let $A = \End_F(V)$.  An \textbf{$\Oc_F$-lattice chain} in $V$ is a sequence $\Lc =
(\Lambda_i)_{i \in \ZZ}$ of $\Oc_F$-lattices in $V$ such that $\Lambda_i \supset \Lambda_{i+1}$ for
all $i\in \ZZ$, and such that there exists an integer $e \geq 1$ (the \textbf{period} of $\Lambda$)
with $\Lambda_{i+e} = \pf_F\Lambda_i$ for all $i \in\ZZ$.  The hereditary $\Oc_F$-orders in $A$ are
those orders $\Af$ that arise as the stabiliser of some $\Oc_F$-lattice chain (which is uniquely
determined up to shift by the order).  The order $\Af$ is maximal if and only if it stabilises a
lattice chain of period $e = 1$.  A hereditary order $\Af \subset A$ has a unique two-sided maximal
ideal $\Pf$; if $\Af$ stabilises $\Lambda$ then $\Pf$ is the set $\{x \in \Af :
x\Lambda_i \subset\Lambda_{i+1} \text{ for all } i \in \ZZ\}$.  We write $U(\Af)$ for the group of
units in $\Af$ and $U^1(\Af) = 1 + \Pf$.  

In \cite{BushnellKutzko1993-AdmissibleDualGLN} \S1.5, the notions of stratum, pure stratum, and
simple stratum in $A$ are defined.  We will only require simple strata in $A$ of the form $[\Af,m,0,\beta]$; this
means that 
\begin{itemize}
\item $\Af$ is a hereditary $\Oc_F$-order in $A$;
\item $m > 0$ is an integer;
\item $\beta \in \Pf^{-m} \setminus \Pf^{1-m}$ is such that $E = F[\beta]$ is a
  field\footnote{The coefficient field $E$ used in the rest of this paper does not appear in this section.} and $E^\times$ is contained in the
  normaliser of $U(\Af)$;
\item $k_0(\beta, \Af(E)) < 0$ where $k_0(\beta, \Af(E))$ is the integer defined in
  \cite{BushnellKutzko1993-AdmissibleDualGLN}~\S1.4.
\end{itemize}
If $[\Af,m,0,\beta]$ is a simple stratum then we may regard $V$ as an $E$-vector space and write $B
= \End_E(V)$.  Any lattice chain defining $\Af$ is then an $\Oc_E$-lattice chain and $\Bf \defeq \Af
\cap B \subset B$ is its stabiliser; we define the groups $U(\Bf)$ and $U^1(\Bf)$ as for $\Af$.  To
a simple stratum $[\Af,m,0,\beta]$ we may associate, as in
\cite{BushnellKutzko1993-AdmissibleDualGLN}~\S3.1, compact open subgroups $J = J(\beta,\Af)$, $J^1 =
J^1(\beta,\Af)$ and $H = H^1(\beta,\Af)$ of $U(\Af)$ such that
\begin{itemize}
\item $J^1$ is a normal pro-$p$ subgroup of $J$;
\item $H^1$ is a normal subgroup of $J^1$;
\item $U(\Bf) \subset J$ and $U^1(\Bf) \subset J^1$, and the induced map 
\[U(\Bf)/U^1(\Bf) \rarrow J/J^1\]
is an isomorphism.
\end{itemize} There is a set $\Cc(\Af,0,\beta)$ of \textbf{simple characters} of $H^1(\beta, \Af)$
(see \cite{BushnellKutzko1993-AdmissibleDualGLN} \S3.2).  If $\theta \in \Cc(\Af,0,\beta)$ is a
simple character, then there is a unique irreducible representation $\eta$ of $J^1(\beta, \Af)$
whose restriction to $H^1(\beta, \Af)$ contains $\theta$, and in fact this restriction is a multiple
of $\theta$.  There is then a distinguished class (the ``$\beta$-extensions'') of extensions
$\kappa$ of $\eta$ to $J(\beta,\Af)$ (see \cite{BushnellKutzko1993-AdmissibleDualGLN}~\S5.2 for $\ch C
= 0$; \cite{Vigneras1996-RepresentationsLModulaires}~\S4.18 for the general case).

\subsection{Types}\label{sec:types}
Suppose that $\ch C = 0$ and $\Omega$ is a Bernstein component of $\Rep_{C}(G)$.
A \textbf{type} for $\Omega$ is a pair $(J,\lambda)$ where $J \subset G$ is a compact open subgroup
and $\lambda$ is an irreducible representation of $J$ with the property that $\Omega$ is equivalent
to the category of smooth $C$-representations of $G$ generated by their $\lambda$-isotypic vectors.

Recall that, for $H$ a unimodular locally profinite group, $K \subset H$ a compact open subgroup,
and $\rho$ a smooth $C$-representation of $K$, then the Hecke algebra $\Hc(H,K,\rho)$ is defined to
be the $C$-algebra \[\End_{C[H]}(\cInd_K^H(\rho)).\]  If $(J,\lambda)$ is a type for $\Omega$,
then $\Hom_J(\lambda,-)$ is an equivalence of categories between $\Omega$ and the category
$\Hc(G,J,\lambda)\operatorname{-Mod}$ of left $\Hc(G,J,\lambda)$-modules (see \cite{MR1643417}).

It is the main result of \cite{BushnellKutzko1999-SemisimpleTypesGLN} that every Bernstein component of
$\Rep_C(G)$ has a type, and there is an explicit construction of these types.  

Suppose that $\Omega$ is a supercuspidal Bernstein component of $\Rep_C(G)$ (that is, every
irreducible object of $\Omega$ is supercuspidal).  Then, by
\cite{BushnellKutzko1993-AdmissibleDualGLN}~\S6 and Theorem~8.4.1, we may construct a type
$(J,\lambda)$ for $\Omega$ such that: $J = J(\beta,\Af)$ for a simple stratum $[\Af,m,0,\beta]$ in
$A$ in which $\Bf$ is a maximal $\Oc_E$-order, and $\lambda$ is of the form $\kappa \otimes \nu$
where $\kappa$ is a $\beta$-extension of an irreducible representation $\eta$ containing a simple
character $\theta\in\Cc(\Af,0,\beta)$ and $\nu$ is a \emph{cuspidal} representation of $J/J^1 \cong
GL_{n/[E:F]}(k_E)$.  The integer $m$ is unique and the pair $(J,\lambda)$ and order $\Af$ are unique
up to conjugation in $G$.  A type $(J,\lambda)$ arising in this way is called a \emph{maximal} type.

\subsection{}\label{sec:kappas} Recall the notions of ps-character and endo-equivalence from
\cite{BushnellKutzko1999-SemisimpleTypesGLN} \S4.  In the situation of the previous paragraph, the
character $\theta$ determines a \textbf{ps-character} $(\Theta,0,\beta)$ attached to the simple pair
$(0,\beta)$ --- this is a function $\Theta$ on the set of simple strata $[\Af,m,0,\beta]$ taking
such a stratum to an element $\Theta(\Af) \in \Cc(\Af,0,\beta)$.  By
\cite{BushnellKutzko1999-SemisimpleTypesGLN}~\S4.5, the endo-class of this ps-character is
determined by $\Omega$.  For each endo-class of ps-character we fix a representative
$(\Theta,0,\beta)$.  We may and do assume that $\theta$ and $\beta$ in the previous paragraph come
from this chosen representative of the endo-class associated to $\Omega$.

We will need to impose a certain compatibility on our choices 
of $\beta$-extensions.  Suppose that $(\Theta,0,\beta)$ is a ps-character attached to the simple
pair $(0,\beta)$ and write $E = F[\beta]$.  Suppose that $E$ is embedded in $A = \End_F(V)$ so that
$V$ is an $E$-vector space, and let $V_1,\ldots,V_t$ be finite-dimensional $E$-vector spaces such
that $V = \bigoplus_{i=1}^t V_i$.  Let $M$ be the corresponding Levi subgroup of $G$, let $Q$ be the
parabolic subgroup with Levi $M$ that stabilises the flag (of $F$-vector spaces) $0 \subset V_1
\subset V_1 \oplus V_2 \subset \ldots \subset V$, and let $U$ be the unipotent radical of $Q$.  For
each $i$ let $A_i = \End_F(V_i)$ and $B_i = \End_E(V_i)$ and let $B = \End_E(V)$.  Suppose that, for
each $i$, there is an $\Oc_E$-lattice $\Lambda_i \subset V_i$ whose stabiliser is $\Bf_i$, a maximal
hereditary $\Oc_E$-order in $B_i$.  Let $\Af_i$ be the corresponding hereditary $\Oc_F$-order in
$A_i$, with associated groups $J_i \supset J^1_i \supset H^1_i$.  Let $\theta_i = \Theta(\Af_i)$ and
let $\eta_i$ be the unique irreducible representation of $J^1_i$ containing $\theta_i$.  Let $\Lf$
be the $\Oc_E$-lattice chain in $V$ whose elements are the lattices
\[\pf_E^{a}\Lambda_1 \oplus \ldots \oplus \pf_E^{a}\Lambda_b \oplus \pf_E^{a+1}\Lambda_{b+1} \oplus
\ldots \oplus \pf_E^{a+1}\Lambda_{t}\] for $a \in\ZZ$ and $1 \leq b \leq t$ (cf
\cite{BushnellKutzko1999-SemisimpleTypesGLN} \S7).  Let $\tilde{\Bf}$ (resp. $\tilde{\Af}$) be the
$\Oc_E$-order (resp. $\Oc_F$-order) associated to $\Lf$ and let $\Bf$ (resp. $\Af$) be the
stabiliser in $B$ (resp. $A$) of a single lattice in $\Lf$.  Let $\tilde{J}\supset \tilde{J}^1
\supset \tilde{H}^1$ be the groups associated to $\tilde{\Af}$, let $\tilde{\theta} =
\Theta(\tilde{\Af})$, and let $\tilde{\eta}$ be the irreducible representation of $\tilde{J}^1$ containing
$\tilde{\theta}$.  Similarly define $J \supset J^1 \supset H^1$, $\theta$ and
$\eta$ to be the objects associated to $\Af$ (and $\Theta$).  By
\cite{BushnellKutzko1993-AdmissibleDualGLN} Theorem~5.2.3, the choice of a $\beta$-extension
$\kappa$ of $\eta$ determines a $\beta$-extension $\tilde{\kappa}$ of $\tilde{\eta}$ such
that \[\Ind_{\tilde{J}}^{U(\tilde{\Af})}(\tilde{\kappa}) \cong
\Ind_{U(\tilde{\Bf})J^1}^{U(\tilde{\Af})}(\kappa|_{U(\tilde{\Bf})J^1}).\] If $\tilde{\kappa}_U$ is the representation
of $\tilde{J} \cap M = \prod_{i=1}^t J_i$ on the $\tilde{J} \cap U$-invariants of $\tilde{\kappa}$, then $\tilde{\kappa}_U = \kappa_1
\otimes \ldots \otimes \kappa_t$ for $\beta$-extensions $\kappa_i$ of each $\eta_i$.  When
$\kappa_1, \ldots, \kappa_t$ arise from a single $\kappa$ in this way, we say that they are
\textbf{compatible}.

\subsection{Covers}\label{sec:covers}
Types for a general Bernstein component of $G$ are constructed using the formalism of covers.
Suppose that $M \subset G$ is a Levi subgroup, that $J\subset G$
is a compact open subgroup, and that $\rho$ is an irreducible smooth representation of $J$.  Write
$J_M = J \cap M$ and suppose that $\rho_M = \rho|_{J_M}$ is irreducible.  The notion of $(J,\rho)$
being a \textbf{$G$-cover} of $(J_M,\rho_M)$ is defined in \cite{MR1643417}~Definition 8.1.

By \cite{MR1643417} Theorem~7.2, if $(J,\rho)$
is a $G$-cover of $(J_M,\rho_M)$, then for each parabolic subgroup $Q$ of $G$ with Levi factor $M$,
there is an injective Hecke algebra homomorphism
\[j_Q : \Hc(M,J_M,\rho_M) \rarrow \Hc(G,J,\rho).\]
Moreover, if every element of $G$ intertwining $\rho$ lies in $M$, then $j_Q$ is an isomorphism, by
Theorem~7.2 and the remark following Corollary~7.7 of \cite{MR1643417}.

If $[M,\pi]$ is an inertial equivalence class of supercuspidal pair corresponding to a Bernstein
component $\Omega$ of $\Rep_C(G)$, then let $(J_M,\lambda_M)$ be a maximal type for the
supercuspidal Bernstein component $\Omega_M$ of $\Rep_C(M)$ containing $\pi$.  By the results of
\cite{BushnellKutzko1999-SemisimpleTypesGLN}, there is a $G$-cover $(J,\lambda)$ of
$(J_M,\lambda_M)$.  The pair $(J,\lambda)$ is then a type for $\Omega$.  For every parabolic
subgroup $Q \subset G$ with Levi subgroup $M$, the diagram
\begin{equation}\begin{CD}
  \label{eq:lambda-covers}
  \Omega_M @>{\;\Ind_{Q}^G(-)\;}>> \Omega \\
  @V{\Hom_{J_M}(\lambda_M,-)}VV  @VV{\Hom_{J}(\lambda,-)}V \\
  \Hc(M,J_M,\lambda_M)\operatorname{-Mod} @>{j_Q}>> \Hc(G,J,\lambda)\operatorname{-Mod}
\end{CD}\end{equation}
commutes (by \cite{MR1643417} Corollary~8.4).

\subsection{SZ-data}
\label{sec:sz-data}
We return to the case of arbitrary $C$ with characteristic distinct from $p$.
\begin{definition}
  An SZ-datum over $C$ is a set \[ \{(E_i,\beta_i,V_i, \Bf_i, \lambda_i)\}_{i=1}^{r}\] where $r$ is a
  positive integer and, for each $i = 1, \ldots, r$, we have:
  \begin{itemize}
  \item $E_i/F$ is a finite extension generated by an element $\beta_i \in E_i$;
  \item $V_i$ is an $E_i$-vector space of finite dimension $N_i$;
  \item $\Bf_i \subset \End_{E_i}(V_i)$ is a maximal hereditary $\Oc_{E_i}$-order and $\Af_i$ is the
    associated $\Oc_F$-order in $A_i \defeq \End_F(V_i)$;
  \item if $m_i = - v_{E_i}(\beta_i)$, then $[\Af_i,m_i,0,\beta_i]$ is a simple stratum and
    $\lambda_i$ is a $C$-representation of $J_i = J(\beta_i,\Af_i)$ of the form $\kappa_i \otimes
    \nu_i$.  Here $\kappa_i$ is a $\beta_i$-extension of the representation $\eta_i$ of
    $J^1_i = J^1(\beta_i,\Af_i)$ containing some simple character $\theta_i \in
    \Cc(\Af_i,0,\beta_i)$ of $H^1_i = H^1(\beta_i,\Af_i)$, and
    $\nu_i$ is an irreducible representation of $U(\Bf_i)/U^1(\Bf_i)\cong GL_{N_i}(k_{E_i})$
    over $C$;
  \item no two of the $\theta_i$ are endo-equivalent.
  \end{itemize}
\end{definition}
Suppose that $\Sf = \{(E_i,\beta_i,V_i, \Bf_i, \lambda_i)\}_{i=1}^{r}$ is an SZ-datum, and adopt all
of the above notation (including the implied choices of $\beta_i$-extensions $\kappa_i$).
\begin{proposition}
  The representations $\lambda_i$ are irreducible.
\end{proposition}
\begin{proof}
  Suppose that some $\lambda_i$ is reducible.  Since $(\kappa_i \tensor \nu_i)|_{H^1_i}$ is a
  multiple of $\theta_i$, we must have that any irreducible subrepresentation $\rho$ of
  $\lambda_i=\kappa_i\tensor \nu_i$ contains $\theta_i$ when restricted to $H^1_i$.  Therefore, by
  \cite{Vigneras1996-RepresentationsLModulaires}, 4.22 Lemme, $\rho$ must also be of the form
  $\kappa_i \tensor \nu_i'$ for an irreducible representation $\nu_i'$ of $J_i/J^1_i$.  But now
  by \cite{Vigneras1996-RepresentationsLModulaires}, 4.22 ``Entrelacement'', we have
   \[ \Hom_{J_i}(\kappa_i\otimes\nu_i', \kappa_i\otimes\nu_i) = \Hom_{J_i/J^1_i}(\nu_i',\nu_i)\]
   and so we must have $\nu_i' = \nu_i$, as $\nu_i$ is irreducible.  Therefore $\rho =
   \kappa_i\tensor \nu_i = \lambda_i$ as required.
\end{proof}

Let $V = \bigoplus_{i=1}^r V_i$ (an $F$-vector space), $A = \End_F(V)$ and $G = \Aut_F(V)$.  The
Levi subgroup $M = \prod_{i=1}^r \Aut_F(A_i) \subset G$ has compact open subgroups $J^1_M\normal
J_M$, where $J^1_M = \prod_{i=1}^rJ^1_i$ and similarly for $J$.  Let $\eta_M =
\bigotimes_{i=1}^r\eta_i$ (a representation of $J^1_M$) and similarly define the representations
$\kappa_M$ and $\lambda_M$ of $J_M$.  Then $\eta_M$ and $\kappa_M$ are clearly irreducible, and
$\lambda_M$ is irreducible by the above proposition.

Since no two of the $\theta_i$ are endo-equivalent, the constructions of
\cite{BushnellKutzko1999-SemisimpleTypesGLN} \S8 (see also \cite{MR3273486} \S\S2.9-10) yield
compact open subgroups $J$ and $J^1$ of $G$ and representations $\eta$ of $J^1$, $\kappa$ of $J$ and
$\lambda$ of $J$ such that $(J^1,\eta)$ (resp. $(J,\kappa)$, resp. $(J,\lambda)$) is a $G$-cover of
$(J^1_M,\eta_M)$ (resp. $(J_M,\kappa)$, resp. $(J_M,\lambda_M)$), and $J/J^1 = J_M/J^1_M$ with
$\lambda = \kappa \otimes (\bigotimes_{i=1}^r \nu_i)$ under this identification.  
\begin{remark}
  If $\ch C \neq 0$, then to see that these are $G$-covers we must modify the proof of \cite{BushnellKutzko1999-SemisimpleTypesGLN}
  Corollary 6.6 as explained in the proof of \cite{MR3273486} Proposition~2.28.  However, here we only
  need the case $\ch C = 0$.
\end{remark}
Let $K$ be a maximal compact subgroup of $G$ such that $J_M \subset K\cap M$ (and so $K \cap
\prod_{i=1}^rB_i = \prod_{i=1}^rU(\Bf_i)$).
\begin{proposition}\label{prop:eta-intertwining}
  Every element of $K$ that intertwines $\eta$ lies in $J$. 
\end{proposition}
\begin{proof}
  By \cite{MR3273486} Proposition~2.31, the $G$-intertwining of $\eta$ is $J(\prod_{i=1}^r
  B_i^\times)J$, and so the $K$-intertwining of $\eta$ is \[J\left(\prod_{i=1}^rB_i^\times\right)J \cap K = J.\qedhere\]
\end{proof}

\begin{definition} \label{defn:sigma-SZ}
  Let $\Sf$ be an SZ-datum over $C$ and let $J,K$ and $\lambda$ be as above.  Then:
\[\sigma(\Sf) = \Ind_J^K(\lambda).\]
\end{definition}
\begin{theorem}\label{thm:irreducibility} The representation $\sigma(\Sf)$ is irreducible.
\end{theorem}
\begin{proof}
  We first show that \[\dim\Hom_{J^1}(\eta, \Ind_{J}^K\lambda) = \dim(\nu).\] By Proposition
  \ref{prop:eta-intertwining}, for $g \not \in J$ we have $\Hom_{J^1 \cap J^g}(\eta, \lambda^g) =
  0$.  Therefore by Mackey's formula,
  \begin{align*}
     \dim\Hom_{J^1}(\eta, \Ind_{J}^K\lambda) & = \dim\Hom_{J^1}(\eta, \lambda)\\
     &= \dim(\nu)
   \end{align*}
   Now suppose that $\Ind_{J}^K\lambda$ is reducible, with
   \[0 \varsubsetneq W \varsubsetneq \Ind_{J}^K\lambda\] a $K$-submodule and $W'$ the quotient.  We may
   write $\Res^{K}_{J^1}\Ind_{J}^K\lambda = W \oplus W'$, since
   $J^1$ is pro-$p$.  Now, by Frobenius reciprocity we have that \[\dim \Hom_{J}(W,\lambda) \geq 1.\] Since $\lambda$ is irreducible and $\lambda|_{J^1} =
   \dim(\nu)\cdot \eta$ this shows that \[\dim\Hom_{J^1}(\eta, W) = \dim\Hom_{J^1}(W, \eta) \geq
   \dim(\nu).\] But the same argument applies to $W'$, so that
   \[\dim\Hom_{J^1}(\eta, \Ind_{J}^K(\lambda))\geq 2 \dim \nu > \dim \nu,\]
   a contradiction!
\end{proof}

\subsection{$K$-types}\label{sec:k-types-2} Now take $C=\bar{E}$.  Let $\Pc \in \Ic$, let $\scs(\Pc) = \Sc :\Ic_0
\rarrow \ZZ_{\geq 0}$ be as in section~\ref{sec:inertial-types}, and let $\Omega = \Omega_\Sc$ be the
associated Bernstein component of $\Rep_C(G)$.  Let $n = \sum_{\tau_0 \in \bar{\Ic}_0} \dim \tau_0
\Sc(\tau_0)$ and let $G = \Aut_F(V)$ for an $n$-dimensional $F$-vector space $V$. Let $(M^0,\pi)$ be
a supercuspidal pair in the inertial equivalence class associated to $\Omega$.  Write $M^0 =
\prod_{i=1}^t M^0_i$ with each $M^0_i$ the stabiliser of some $n_i$-dimensional subspace $V^0_i$ of
$V$, write $\pi = \bigotimes_{i=1}^t \pi_i$, and let $\Omega_i$ be the supercuspidal Bernstein
component of $\Rep_C(M^0_i)$ containing $\pi_i$.  For each $\Omega_i$ there is an associated
endo-class of ps-character, for which we have chosen a representative $\Theta^0_i =
(\Theta^0_i,0,\beta^0_i)$.  Construct a Levi subgroup $M = \prod_{i=1}^rM_i$ with $M^0 \subset M
\subset G$ by requiring that $M_j^0$ and $M_k^0$ are both contained in some $M_i$ if and only if
$\Theta_j^0 = \Theta_k^0$; in this case we write \[(\Theta_i,0,\beta_i) = (\Theta^0_j,0,\beta^0_j) =
(\Theta^0_k,0,\beta^0_k)\] for the common value.  Let $V = \bigoplus_{i=1}^rV_i =
\bigoplus_{i=1}^tV_i^0$ be the decompositions of $V$ corresponding to $M$ and $M^0$ respectively, so
that the second is strictly finer than the first.

Suppose first that $r = 1$ (the \textbf{homogeneous} case).  Then write $(\Theta,0,\beta)$ for the
common value of $(\Theta_i^0,0,\beta_i^0)$, and $E = F[\beta]$.  For $0 \leq i \leq t$, there is a
maximal simple type $(J_i^0,\lambda_i^0)$ for $\Omega_i$ such that $J_i^0 = J(\beta,\Af_i^0)$ for a
simple stratum $[\Af_i^0,m,0,\beta]$, and $\lambda_i^0$ contains $\theta_i^0 \defeq
\Theta(\Af_i^0,0,\beta)$.  We are in the situation of \cref{sec:kappas}, and adopt the notation
there (adorning it with a superscript `0' where appropriate).  In particular we have
compact open subgroups $J^1 \subset J$ of $G$ and a representation $\eta$ of $J^1$ containing the
simple character $\Theta(\Af,0,\beta)$, where $\Af$ is a hereditary $\Oc_F$-order in $A$ and $\Af
\cap B = \Bf$ is a maximal hereditary $\Oc_E$-order.  We choose \emph{compatible} $\beta$-extensions
$\kappa_i^0$ of $\eta_i^0$ coming from a $\beta$-extension $\kappa$ of $\eta$, and decompose each
$\lambda_i^0$ as $\kappa_i^0 \otimes \nu_i^0$, where $\nu_i^0$ is a cuspidal representation of
$J_i^0/J_i^{1,0} = U(\Bf_i^0)/U^1(\Bf_i^0)$.  Choosing an $\Oc_E$-basis of each $\Bf^0_i$, we
identify $U(\Bf_i^0)/U^1(\Bf_i^0)$ with $GL_{n^0_i/[E:F]}(k_E)$ for an integer $n_i^0$ and
$J/J^1=U(\Bf)/U^1(\Bf)$ with $GL_{n/[E:F]}(k_{E})$.  So we may view each $\nu^0_i$ as an element of
$\bar{\Ic}_0$, and define an element $\bar{\Pc} \in\bar{\Ic}$ by $\bar{\Pc}(\nu^0_i) =
\Pc(\tau_i)$, where $\tau_i \in \Ic_0$ corresponds to $\Omega_i$.  Then write $\nu =
\pi_{\bar{\Pc}}$, a representation of $GL_{n/[E:F]}(k_E)$, and regard it as a representation of $J/J^1$.

\begin{definition} \label{def:SZ-datum}
  In this homogeneous case, we define an SZ-datum $\Sf_\Pc$ by
\[\Sf_{\Pc} = \{(E,\beta,V,\Bf,\kappa\otimes \nu)\}.\]
In the general case, for $1 \leq i \leq r$ let 
\[\{(E_i,\beta_i,V_i,\Bf_i,\kappa_i\otimes \nu_i)\}\]
be the SZ-datum for $M_i$ given by the construction in the homogeneous case, and set 
\[\Sf_{\Pc} = \{(E_i,\beta_i,V_i,\Bf_i,\kappa_i \otimes \nu_i)\}_{i=1}^r.\] 
If $\tau = \tau_\Pc$, we write $\sigma(\tau) =  \sigma(\Sf_\Pc)$.
\end{definition}

\subsection{}\label{sec:parabolic} We show that the representations $\sigma(\tau)$ satisfy the conclusion of
Theorem~\ref{thm:K-types}.  Continue with the notation of section~\ref{sec:k-types-2}, and suppose
that $r = 1$.  Let $M^2 \supset M^0$ be a Levi subgroup of a parabolic subgroup $Q^2 \subset
G$, let $\Bf^2= \Bf \cap M^2$, and let $J^{1,2} \subset J^2$, $\eta^2$ and $\kappa^2$ be the
subgroups of $M^2$ and their representations obtained from $\Theta$.  We require that $\kappa^2$ is
compatible with $\kappa$.

Write $\bar{G} = J/J^1 = U(\Bf)/U^1(\Bf)$, $\bar{M}\,^2 = (M^2 \cap U(\Bf))J^1/J^1$, and $\bar{Q}\,^2
= (Q^2\cap U(\Bf))J^1/J^1$.  Then $\bar{Q}\,^2$ is a parabolic subgroup of $\bar{G}$ with Levi $\bar{M}\,^2$.
\begin{proposition} \label{prop:homog-ind}
  The following diagram commutes: 
  \begin{equation}\begin{CD}
  \label{eq:kappa-max}
\Rep_{C}(M^2) @>{\;\Ind_{Q^2}^G(-)\;}>> \Rep_C(G) \\
@V{\Hom_{J^{1,2}}(\kappa^2,-)}VV  @VV{\Hom_{J^1}(\kappa,-)}V \\
\Rep_{C}(\bar{M}\,^2) @>>{\;\Ind_{\bar{Q}\,^2}^{\bar{G}}(-)\;}> \Rep_C(\bar{G}).
\end{CD}\end{equation}
\end{proposition}

\begin{proof}
  This may be proved in the same way as
  \cite{SchneiderZink1999-KTypesTemperedComponentsGeneralLinear} Proposition~7; we omit the details.
  See also \cite{secherrestevens15} Proposition~5.6.
\end{proof}

\begin{corollary}\label{cor:ds-ind}
  Suppose that $M^2$ is as above and further suppose that $(M^2,\pi^2)$ is a discrete pair in the
  inertial equivalence class associated to some $\Pc'\in \Ic$; let $\bar{\Pc}\,' : \bar{\Ic}_0
  \rarrow \Part$ correspond to $\Pc'$.  Then
\[\Hom_{J^1}(\kappa,\Ind_{Q^2}^G(\pi^2)) = \pi_{\bar{\Pc}\,'}\]
as representations of $\bar{G}$.
\end{corollary}
\begin{proof}
  By Proposition~\ref{prop:homog-ind}, 
\[\Hom_{J^1}(\kappa,\Ind_{Q^2}^G(\pi^2)) =
\Ind_{\bar{Q}\,^2}^{\bar{G}}(\Hom_{J^{1,2}}(\kappa^2,\pi^2)).\]
By \cite{SchneiderZink1999-KTypesTemperedComponentsGeneralLinear} Proposition~5.6,
\[\Hom_{J^{1,2}}(\kappa^2,\pi^2) = \St(\bar{\Pc}\,').\]
Therefore: 
\begin{align*}\Ind_{\bar{Q}\,^2}^{\bar{G}}(\Hom_{J^{1,2}}(\kappa^2,\pi^2))
  &=\Ind_{\bar{Q}\,^2}^{\bar{G}}(\St(\bar{\Pc}\,')) \\
  &=\pi_{\bar{\Pc}\,'}.\qedhere
\end{align*}
\end{proof}

Now suppose that $r > 1$, so that $M \subset G$ is a proper Levi subgroup.  Let $J_M = \prod_{i=1}^r
J_i$, $J_M^1 = \prod_{i=1}^r J_i^1$, and $\eta_M = \bigotimes_{i=1}^r \eta$.  Then as in
section~\ref{sec:sz-data} there is a $G$-cover $(J^1,\eta)$ of $(J^1_M,\eta_M)$.  We have a
canonical isomorphism $J_M/J^1_M = J/J^1$ induced by the inclusion $J_M \into J$.  For each
parabolic subgroup $Q$ of $G$ with Levi $M$, there is an isomorphism
\[j_Q : \Hc(M,J^1_M,\eta_M) \isomto \Hc(G,J^1,\eta)\]  
such that the diagram
\begin{equation}\begin{CD}
  \label{eq:eta-covers}
  \Rep_{C}(M) @>{\;\Ind_{Q}^G(-)\;}>> \Rep_C(G) \\
  @V{\Hom_{J^{1}_M}(\eta_M,-)}VV  @VV{\Hom_{J^1}(\eta,-)}V \\
  j_Q:\Hc(M,J_M,\eta_M)\operatorname{-Mod} @>{\sim}>> \Hc(G,J,\eta)\operatorname{-Mod}
\end{CD}\end{equation}
commutes, by the discussion of section~\ref{sec:covers} and the intertwining bound of
\cite{MR3273486} Proposition~2.31.  Then, writing $K_M = K \cap M$, $j_Q$ induces an isomorphism
\[\Hc(K_M,J^1_M,\eta_M) \isomto \Hc(K,J^1,\eta).\]
But, by Proposition~\ref{prop:eta-intertwining}, we have 
\[\Hc(K_M,J^1_M,\eta_M) = \Hc(J_M,J^1_M,\eta_M)\]
and choosing $\kappa_M$ identifies this with $C[J_M/J^1_M]$.  Similarly, choosing $\kappa$
identifies $\Hc(K,J^1,\eta)$ with $C[J/J^1]$.  As $j_Q$ is support-preserving, if we choose $\kappa$
such that $\kappa|_{J_M} = \kappa_M$ then the isomorphism $j_Q$ agrees with the identification
$C[J_M/J^1_M] = C[J/J^1]$. Therefore, when $\nu$ is a
representation of $J/J^1 = J_M/J^1_M$, the isomorphism $j_Q$ takes $\Ind_{J_M}^{K_M}(\kappa_M \otimes
\nu)$ to $\Ind_J^K(\kappa\otimes \nu)$.  So we have shown that, for every smooth representation
$\pi_M$ of $M$, we have:
\begin{equation}\label{eq:inhom-ind}
  \Hom_K(\Ind_J^G(\kappa \otimes \nu),\Ind_Q^G(\pi_M)) = \Hom_{K_M}(\Ind_{J_M}^{K_M}(\kappa_M \otimes \nu_M),\pi_M).
\end{equation}

\begin{theorem} \label{thm:types-multiplicities} Let $\Pc'\in \Ic$ with $\deg \Pc' = n$, let
  $(M',\pi')$ be any discrete pair in the inertial equivalence class associated to $\Pc'$, and let
  $Q' \subset G$ be any parabolic subgroup with Levi subgroup $M'$.  Then
  \[\dim \Hom_K(\sigma(\Sf_{\Pc}),\Ind_{Q'}^G(\pi')) = \prod_{\tau_0 \in
    \Ic_0}m(\Pc(\tau_0),\Pc'(\tau_0)).\]
\end{theorem}
\begin{proof}
  We can assume that $\scs(\Pc') = \Sf = \scs(\Pc)$; otherwise both sides are zero --- the left hand
  side because $\sigma(\Sf_\Pc)$ contains a type for the Bernstein component $\Omega$ corresponding
  to $\Sc$. Therefore we can assume that $M_0 \subset M' \subset M$. Using the commutative diagram
  \eqref{eq:eta-covers}, we may reduce to the case in which $\Pc$ and $\Pc'$ are homogeneous.  But
  now the result follows from Corollary~\ref{cor:ds-ind} and Corollary~\ref{cor:finite-gl-mults}.
\end{proof}
\begin{corollary}\label{cor:construction-proof} Let $\Pc' \in \Ic$ and let $(M',\pi')$ be a discrete
  pair in the inertial equivalence class associated to $\Pc'$.  Let $\pi =L(M',\pi')$ be the irreducible
  admissible representation defined in section~\ref{sec:bernst-zelev-class}, so that $r_l(\pi)|_{I_F}
  \cong \tau_{\Pc'}$. Then:
  \begin{enumerate}
  \item if $\pi|_K$ contains $\sigma(\Sf_\Pc)$, then $\Pc' \preceq \Pc$;
  \item if $\Pc' = \Pc$, then $\pi|_K$ contains $\sigma(\Sf_\Pc)$ with multiplicity one;
  \item if $\Pc' \preceq \Pc$ and $\pi$ is \emph{generic}, then $\pi|_K$ contains
    $\sigma(\Sf_{\Pc})$ with multiplicity one.
  \end{enumerate}
\end{corollary}
\begin{proof} This is proved in the same way as
  \cite{SchneiderZink1999-KTypesTemperedComponentsGeneralLinear} Proposition~5.10. By
  Theorem~\ref{thm:types-multiplicities}, if $Q' \subset G$ is a parabolic subgroup with Levi $M'$,
  then $\sigma(\Sf_\Pc)$ is contained in $\Ind_{Q'}^G(\pi')$ if and only if $\Pc' \preceq \Pc$.
  Therefore if $L(M',\pi')$ contains $\sigma(\Sf_\Pc)$, then $\Pc' \preceq \Pc$, proving part 1.  If
  $L(M',\pi')$ is generic, then it is equal to $\Ind_{Q'}^G(\pi')$ for any $Q'$, proving part 3.
  Finally, suppose $\Pc' = \Pc$.  By Theorem~\ref{thm:types-multiplicities}, $\sigma(\Sf_\Pc)$ occurs in $\Ind_{Q'}^G(\pi')$
  with multiplicity one; in other words, exactly one constituent of $\Ind_{Q'}^G(\pi')$ contains
  $\sigma(\Sf_\Pc)$, and it does so with multiplicity one.  But every constituent of
  $\Ind_{Q'}^G(\pi')$ other than $L(M',\pi')$ is equal to $L(M'',\pi'')$ for some discrete pair
  $(M'',\pi'')$ in the inertial equivalence class associated to $\Pc''$ for some $\Pc'' \succ \Pc$
  (see \cite{SchneiderZink1999-KTypesTemperedComponentsGeneralLinear}~\S2 Lemma), and so by part 1
  does not contain $\sigma(\Pc)$.  Hence $\sigma(\Pc)$ is contained in $L(M',\pi')$ with
  multiplicity one, as required.
\end{proof}
\begin{corollary}\label{cor:multiplicity-calculation}Let $\Pc$ and $\Pc'$ be elements of $\Ic$.  Then
 \[m(\sigma(\tau_{\Pc}),\tau_{\Pc'}) =  \prod_{\tau_0 \in \Ic_0}m(\Pc(\tau_0),\Pc'(\tau_0)).\]
\end{corollary}
\begin{proof}
  This follows from Theorem~\ref{thm:types-multiplicities} together with that fact that any \emph{generic} irreducible admissible
  representation $\pi$ of $GL_n(F)$ is the \emph{irreducible} induction of a discrete series representation
  of a Levi subgroup.
\end{proof}

\subsection{Reduction modulo $l$.}

Let $\Sf = \{(E_i,\beta_i, V_i,\Bf_i,\lambda_i)\}_{i=1}^r$ be an SZ-datum over $\bar{E}$.  Decompose
each $\lambda_i$ as $\kappa_i \otimes \nu_i$ for irreducible representations $\nu_i$ of $J_i/J^1_i$.
Suppose that 
\[\bar{\nu_i}^{ss} = \bigoplus_{j \in S_i} \mu_{ij}\nu_{ij}\]
where $S_i$ is some finite indexing set, $\nu_{ij}$ are distinct irreducible representations of
$J_i/J^1_i$ over $\bar{\FF}$ and $\mu_{ij} \in \NN$.  Note that each $\bar{\eta_i}$, and hence
$\bar{\kappa_i}$, is irreducible.  For $\bm{j} = (j_1,\ldots,j_r) \in S_1 \times \ldots \times S_r$,
define an SZ-datum $\Sf_{\bm{j}}$ over $\bar{\FF}$ by
\[\Sf_{\bm{j}} = \{(E_i,\beta_i,V_i,\Bf_i,\bar{\kappa_i} \otimes \nu_{ij_i}\}_{i=1}^r\]
and an integer $\mu_{\bm{j}} = \prod_{i=1}^r \mu_{ij_i}$.  Then we have:

\begin{theorem}\label{thm:type-reduction-gln} 
  The semisimplified mod $l$ reduction of $\sigma(\Sf)$ is 
\[\bigoplus_{\bm{j}\in S_1\times \ldots\times S_r} \mu_{\bm{j}}\sigma(\Sf_{\bm{j}}).\]
\end{theorem} 

\begin{proof}
  This follows immediately from Theorem \ref{thm:irreducibility}. 
\end{proof}

\section{Towards a local proof}
\label{sec:towards-local-proof}

In this section, we give a proof of Conjecture~\ref{conj:BM} in the case that $q \equiv 1 \mod l$
and $l > n$ (we say that $l$ is \textbf{quasi-banal}), $\rhobar|_{\tilde{P}_F}$ is trivial, and
$R_E(GL_n(\Oc_F))$ and $R_{\FF}(GL_n(\Oc_F))$ are replaced by subgroups generated by certain
representations inflated from $GL_n(k_F)$.  The strategy of proof is to first show that it suffices
to prove Conjecture~\ref{conj:BM} for a single $\rhobar$ on each irreducible component of
$\Mc(n,q)_\FF$ such that $\rhobar$ is on no other irreducible components.  But for good choices of
$\rhobar$, we may explicitly determine the rings $R^\square(\rhobar,\tau)$ for all inertial types
$\tau$.  As we also have a very good understanding of the mod $l$ representation theory of
$GL_n(k_F)$ under our assumptions on $l$, Conjecture~\ref{conj:BM} reduces to a combinatorial
identity, which we verify.

\subsection{Reduction to finite type}
\label{sec:reduct-finite-type}

Let $\Xf$ be the affine scheme $\Mc(n,q)_\Oc$ from section~\ref{sec:deformation-rings}.  We suppose
that $\Oc$ contains all of the $(q^{n!}-1)$th roots of unity, so that every irreducible component of
$\Xf_\FF$ or $\Xf_E$ is geometrically irreducible.  Once we have fixed generators $\sigma$ and
$\phi$ for $G_F/P_F$ as usual, then there is a natural bijection between $\Xf(\bar{\FF})$ and the
set of continuous homomorphisms $\bar{\rho}:G_F \rarrow GL_n(\bar{\FF})$ with kernel containing
$P_F$.  If $x$ is a closed point of $\Xf$ corresponding to such a homomorphism $\bar{\rho}_x$, and
we suppose that the residue field of $\Xf$ at $x$ is $\FF$, then there is a natural isomorphism
\[ \Oc_{\Xf, x}^\wedge = R^\square(\rhobar_x).\]
From the map \[i:\Spec R^\square(\rhobar_x) \rarrow \Xf\] we get a pullback 
\[i^* : \Zc(\Xf) \rarrow \Zc(R^\square(\rhobar_x))\]
as in section~\ref{sec:cycles}.  Similarly, writing $\bar{\Xf} = \Xf
\times _{\Spec \Oc} \Spec \FF$, we have a map 
\[i^* : \Zc(\bar{\Xf}) \rarrow \Zc(R^\square(\rhobar_x)\otimes_{\Oc}\FF),\]
and the diagram 
\begin{equation}
  \label{eq:ft-local-compat}
\begin{CD}
\Zc(\Xf) @>{i^*}>> \Zc(R^\square(\bar{\rho}_x)) \\
@V{\red}VV  @V{\red}VV \\
\Zc(\bar{\Xf}) @>{i^*}>> \Zc(\Rbarbox(\rhobar_x))
\end{CD}
\end{equation}
commutes, by Lemma~\ref{lem:flat-pullback}.

There is a unique map
\[\cyc_{\ft}:R_E(GL_n(\Oc_F)) \rarrow \Zc(\Xf)\]
such that for each $x \in \bar{\Xf}(\FF)$ the map $\cyc : R_E(GL_n(\Oc_F)) \rarrow \Zc(R^\square(\bar{\rho}_x))$, which to avoid ambiguity we
will call $\cyc_x$, is equal to the composition $i^* \circ \cyc_{\ft}$.

Let $\rom{BM}_\ft$ be the statement that there exists a map $\bar{\cyc}_\ft$ (necessarily unique)
making the diagram
\begin{equation}
  \label{eq:BM-ft}
\begin{CD}
R_E(GL_n(\Oc_F)) @>{\cyc_\ft}>> \Zc(\Xf) \\
@V{\red}VV  @V{\red}VV \\
R_{\FF}(GL_n(\Oc_F)) @>{\bar{\cyc}_\ft}>> \Zc(\bar{\Xf}).
\end{CD}
\end{equation}
commute, and for $x \in \Xf(\FF)$ let $\rom{BM}_x$ be the statement that Conjecture~\ref{conj:BM} holds for $\rhobar =
\rhobar_x$.  Then we have:
\begin{proposition}
  \label{prop:ft-reduction}
  \begin{enumerate}
  \item If $\rom{BM}_\ft$ is true, so is $\rom{BM}_x$ for all $x \in \Xf(\FF)$.
  \item Suppose that $S\subset \Xf(\FF)$ has the property that, for every irreducible component $\Zf$ of
    $\bar{\Xf}$, there is an $x \in S$ such that $x$ lies on $\Zf$ and on no other irreducible
    component of $\bar{\Xf}$.  If $\rom{BM}_x$ is true for all $x \in S$, then $\rom{BM}_{\ft}$ is true.
  \end{enumerate}
    \end{proposition}
\begin{proof}
  For the first part, given the existence of a map $\bar{\cyc}_\ft$ we define
  $\bar{\cyc}_x$ to be the composition of $\bar{\cyc}_{\ft}$ with $i^*$.  Then $\rom{BM}_x$ follows
  from the commutativity of diagrams~\eqref{eq:ft-local-compat} and~\eqref{eq:BM-ft}.  
  
  For the second part we simply need to observe that, under the given assumptions on $S$, the map 
  \[i^* : \Zc(\bar{\Xf}) \rarrow \prod_{x \in S} \Zc(\Rbarbox(\rhobar_x))\] is injective.  
\end{proof}

Let $(\bar{M},\rhobar,
(\bar{e}_i)_i)$ be a representation $(\bar{M},\rhobar)$ of $G_F$ with a basis $(\bar{e}_i)_i$. Let
$\bar{M} = \bar{M}_1 \oplus\ldots\oplus \bar{M}_r$ for the decomposition of $\bar{M}$ into generalised
eigenspaces for $\rhobar(\phi)$, with $\bar{M}_i$ having generalised eigenvalue $\alpha_i \in \FF$
and dimension $n_i$.  
\begin{definition}\label{def:standard}  Say that $(\bar{M}, \rhobar, (\bar{e}_i)_i)$ is
  \textbf{standard} if each $\bar{e}_i$ lies in some $\bar{M}_j$.
  
  Let $A$ be an object of $\Cc_\Oc$ and let $(M,\rho, (e_i)_i)$ be a lift of $(\bar{M},\rhobar,
  (\bar{e}_i)_i)$. Say that $(M,\rho,(e_i)_i)$ is \textbf{standard} if we may write $M = M_1 \oplus
  \ldots \oplus M_r$ with each $M_i$ being a $\rho(\phi)$-stable lift of $\bar{M}_i$ and, whenever
  $\bar{e}_i \in \bar{M}_j$ for some $i,j$, we have $e_i \in M_j$.  

  The property of being standard only depends on the equivalence class of $(M,\rho,(e_i)_i)$, and so
  we can talk of homomorphisms $\rho : G_F \rarrow GL_n(A)$ being standard.
  
  Let $R^{\std}(\rhobar)$ be the maximal quotient of $R^\square(\rhobar)$ on which $\rho^\square$ is standard.
\end{definition}
Thus we are requiring that $\rhobar(\phi)$ is block diagonal with each block having a single generalised
eigenvalue and different blocks having different eigenvalues, and that $\rho(\phi)$ is block
diagonal with blocks lifting those of $\rhobar(\phi)$.  It is clear that, given $(\bar{M},\bar{\rho})$, we
may choose a basis $\bar{e}_i$ such that $(\bar{M},\bar{\rho}, (\bar{e}_i))$ is standard. 
\begin{lemma}\label{lem:diagonalization2}
  Let $(\bar{M},\rhobar, (\bar{e}_i)_i)$ be standard.  Then there is an injective
  morphism \[R^{\std}(\rhobar) \rarrow R^\square(\rhobar)\] in $\Cc^\wedge_\Oc$ making
  $R^\square(\rhobar)$ formally smooth over $R^{\std}(\rhobar)$.
  \end{lemma}
\begin{proof}
  Adopt the notation of Definition~\ref{def:standard}.  Let $\bar{P}_i(X) = (X - \alpha_i)^{n_i}$ for $i=
  1,\ldots, r$, so that the characteristic polynomial of $\rhobar(\phi)$ is \[\bar{P}(X) =
  \prod_{i=1}^r \bar{P}_i(X).\]
 
  If $(M, \rho, (e_{i})_i)$ is a lift of $(\bar{M},\rhobar,(\bar{e}_{i})_i)$ to $A \in \Cc_\Oc$, we
  will functorially produce a new basis $(f_i)_i$ such that $(M,\rho,(f_i)_i)$ is standard.
  Let $P(X)$ be the characteristic polynomial of $\rho(\phi)$.  By Hensel's lemma, there is a
  factorisation
  \[P(X) = \prod_{i=1}^r P_i(X)\] with $P_i(X) \in A[X]$ such that the image of $P_i(X)$ in $\FF[X]$
  is $\bar{P}_i(X)$ for each $i$.  Let \[Q_i(X) = \frac{P(X)}{P_i(X)} = \prod_{j \neq i} P_j(X)\]for
  $1 \leq i \leq r$.  Writing $M_i = Q_i(\rho(\phi))M$, we have \[M = \bigoplus_{i=1}^r M_i.\] 
  Then the isomorphism $M \otimes \FF \isomto \bar{M}$ takes $M_i$ to $\ker(\bar{P}_i)= \bar{M}_i$
  and each $M_i$ is a $\rho(\phi)$-stable submodule of $M$.

  Now, each $e_i$ may be written uniquely as $e_i^{(1)} + \ldots + e_i^{(r)}$ with $e_i^{(j)}
  \in M_j$ for each $j$; we take $f_i = e_i^{(i)}$.  Then $(M,\rho, (f_i)_i)$ is a standard lift of
  $(\bar{M},\rhobar,(\bar{e}_{i})_i)$.
  
  We have therefore defined a map $R^{\std}(\rhobar) \rarrow R^\square(\rhobar)$ which is easily seen to
  be injective and formally smooth.
\end{proof}

\subsection{Representation theory}
\label{sec:rep-theory}

From now until the end of section~\ref{sec:towards-local-proof}, we suppose that $l$ is quasi-banal
--- that is, that $l > n$ and $q \equiv 1$ mod $l$.  Let $a = v_l(q-1)$ and let $\mu_{l^a}$ be the
group of $l^a$th roots of unity in $\Oc$.  Let $T\subset B \subset GL_n(k_F)$ be the standard
maximal torus and Borel subgroup, let $U$ be the unipotent radical of $B$, and let $B_1$ be the
maximal subgroup of $B$ of order coprime to $l$, so that $B/B_1 \cong (\ZZ/l^a\ZZ)^n$.  Let
$R^1_E(GL_n(k_F))\subset R_E(GL_n(k_F))$ and $R^1_{\FF}(GL_n(k_F)) \subset R_{\FF}(GL_n(k_F))$ be the
subgroups generated by those irreducible representations having a $B_1$-fixed vector.

Recall the notation $\bar{\Ic}$, $\bar{\Ic}_0$, $\pi_{\bar{\Pc}}$, $\sigma_{\bar{\Pc}}$ from
section~\ref{sec:finite-gener-line}.  If $\chi$ is a character of $k_F^\times$ with values in
$\mu_{l^a}$, then $\chi$ is an element of $\bar{\Ic}_0$ of degree one.  Let $\bar{\Ic}_1$ be the set
of functions $\bar{\Ic}_0 \rarrow \Part$ supported on the set of $\chi$ of this form.  If $P$ is a
partition of $n$ then define $\sigma^1_P$ to be the representation $\sigma_{\bar{\Pc}}$ of
$GL_n(k_F)$ where $\bar{\Pc} : \bar{\Ic}_0 \rarrow \Part$ takes the trivial representation to $P$ and
everything else to zero.  In other words, $\sigma^1_P$ is the unipotent representation
associated to the partition $P$.  If $\mathbbm{1}$ is the trivial representation of $GL_1(k_F)$ then, under
the isomorphism of PSH-algebras $R(\mathbbm{1}) \isomto R^S$ of Proposition~\ref{prop:psh-Rgl},
$\sigma^1_P$ corresponds to $\sigma^{\circ}_P$.

\begin{lemma}\label{lem:quasiunipotent-reduction}
  \begin{enumerate}
  \item Every irreducible representation of $GL_n(k_F)$ having a $B_1$-fixed vector is of the form
    $\sigma_{\bar{\Pc}}$ for some $\bar{\Pc} \in \bar{\Ic}_1$.
  \item If $P$ is a partition of $n$, then $\red(\sigma^1_{P})$ is irreducible.
  \item If $\bar{\Pc} \in \bar{\Ic}_1$ sends each $\chi$ to a partition $P_\chi$ of degree $n_\chi$, then
    let $P$ be the partition of $n$ whose parts are the $n_\chi$, let $\bar{M}$ be the corresponding
    standard Levi subgroup of $\bar{G} =GL_n(k_F)$, and let $\bar{Q}$ be a parabolic subgroup with Levi
    $\bar{M}$.  Then
\[\red(\sigma_{\bar{\Pc}}) = \red(\Ind_{\bar{Q}}^{\bar{G}}(\bigotimes_\chi \sigma^1_{P_\chi})).\]
  \end{enumerate}
\end{lemma}
\begin{proof}
  If $\sigma$ is an irreducible representation of $GL_n(k_F)$ having a $B_1$-fixed vector, then it
  has non-trivial $U$-invariants, on some subrepresentation of which $T = B/U$ acts as $\chi =
  \bigotimes_{i=1}^n \chi_i$ with $\chi_i$ having values in $\mu_{l^a}$.  So $\sigma$ is a
  subquotient of $\Ind_B^G \chi$ and is therefore of the required form, proving part 1.
  
  Part 2 follows from the discussion in section 3 of \cite{MR1031453}.

  Part 3 is immediate from the definition of $\sigma_{\bar{\Pc}}$ and the observation that if $\chi$
  takes values in $\mu_{l^a}$ then its mod $\lambda$ reduction is trivial.
\end{proof}

The representation $\Ind_{\bar{Q}}^{\bar{G}}(\bigotimes \sigma^1_{P_\chi})$ appearing in part~3 of
the lemma decomposes as a direct sum of irreducible representations of the form $\sigma_{P'}$ for
partitions $P'$ of $n$.  More specifically, from the isomorphism of PSH-algebras $R(\mathbbm{1})
\isomto R^S$ of Proposition~\ref{prop:psh-Rgl} we obtain:
\begin{equation}\label{eq:linear-to-symmetric}\dim \Hom_{\bar{G}}(\sigma_{P'},\Ind_{\bar{Q}}^{\bar{G}}(\bigotimes_\chi
    \sigma^1_{P_\chi})) = \dim \Hom_{S_n}(\sigma^\circ_{P'}, \Ind_{S_P}^{S_n}(\bigotimes_\chi \sigma_{P_{\chi}}^\circ)).\end{equation}
Here $S_P$, $\sigma_{P'}^\circ$ and $\sigma_{P_\chi}^\circ$ are as in
section~\ref{sec:symmetric-groups}.

We will need to compute the Mackey decomposition \[\Res^{S_n}_{S_P} \Ind_{S_Q}^{S_n} \sgn\] for
pairs of partitions $P$ and $Q$ of degree $n$, and for this we introduce some notation:

\begin{definition}
  Let $P,Q \in \Part$ of degree $n$.  A \textbf{$(P,Q)$-bipartition} is a matrix $A = (a(i,j))_{i,j}$ of
  non-negative integers (with $i, j \in \NN$) such that:
  \begin{itemize}
    \item all but finitely many $a(i,j)$ are zero;
  \item for each $i$, the sum $\sum_{j} a(i,j)$ of the entries of the $i$th row is $P(i)$;
  \item for each $j$, the sum $\sum_{i} a(i,j)$ of the entries of the $j$th column is $Q(j)$.
  \end{itemize}
  The $i$th row of $A$ determines a partition $P_i$ of $P(i)$.  We define the \textbf{weight} of $A$
  to be the sequence of partitions $(P_1,P_2,\ldots)$.
  
  If $(P_i)_i$ is a finite sequence of partitions and $P$ is the partition formed by their
  degrees, then define $\Bip\left((P_i)_i,Q\right)$ to be the number of $(P,Q)$-bipartitions of weight~$(P_i)_i$.
\end{definition}

If $P$ is a partition of $n$, then let $T_{P,i}$ be the set \[\left\lbrace1 + \sum_{j=1}^{i-1} P(j), 2 +
\sum_{j=1}^{i-1} P(j), \ldots, \sum_{j=1}^i P(j)\right\rbrace\] (with the convention that this is empty if the
first term is greater than the last), so that $\{1,\ldots,n\}$ is the disjoint union of the
$T_{P,i}$; write $T_P$ for the sequence $(T_{P,i})_i$.  In the left action of $S_n$ on the set of
partitions of $\{1,\ldots,n\}$ into disjoint subsets, $S_P$ is the stabiliser of $T_P$.

\begin{lemma} Let $P$ and $Q$ be partitions of $n$.  There is a bijection between the double coset
  set $S_P\backslash S_n/S_Q$ and the set of all $(P,Q)$-bipartitions.
\end{lemma}
\begin{proof}
  This is standard; let us just recall the construction.  If $g \in S_n$, define a matrix $A_g = (A_g(i,j))$ by
  \[A_g(i,j) = \#(T_{P,i} \cap gT_{Q,j}).\] Then $A_g$ is a $(P,Q)$-bipartition that only depends on
  the double coset $S_PgS_Q$, and the map $S_PgS_Q \mapsto A_g$ gives the required bijection.
\end{proof}

\begin{proposition} \label{prop:mackey} Let $P$ and $Q$ be partitions of $n$.  Then we have:
\[\Res_{S_P}^{S_n}(\pi_Q^\circ) \cong \bigoplus_{(P_i)_i} \left(\Bip\left((P_i)_i,Q\right)\cdot \bigotimes_{i} \pi_{P_i}^\circ\right)\] where the sum runs over all
sequences of partitions $(P_1,P_2,\ldots)$ with $\deg P_i = P(i)$ and, for an integer $a$ and
representation $\rho$, $a \cdot \rho$ denotes the direct sum of $a$ copies of $\rho$.
\end{proposition}
\begin{proof}
  By definition, $\pi_Q^\circ = \Ind_{S_Q}^{S_n}(\sgn)$ and $\pi_{P_i}^\circ =
  \Ind_{S_{P_i}}^{S_{P(i)}} (\sgn)$ for each $i$.  The formula follows from Mackey's theorem upon
  observing that, if $S_P gS_Q$ is the double coset corresponding to a $(P,Q)$-bipartition of weight
  $(P_i)_i$, then $S_P \cap S_Q^g$ is conjugate (in $S_P$) to the subgroup $\prod_{i} S_{P_i}\subset
  \prod_{i}S_{P(i)} = S_P$.
\end{proof}

\subsection{Deformation rings at distinguished points}
\label{sec:geom-at-dist}

Let $(\bar{M},\bar{\rho})$ be a representation of $G_F$ over $\FF$ such that $\tilde{P}_F$ acts
trivially, and that $(\bar{e}_i)_i$ is a basis for $\bar{M}$.
\begin{definition}
  Say that $(\bar{M},\bar{\rho},(\bar{e}_i)_i)$ is \textbf{distinguished} if:
  \begin{itemize}
  \item it is standard, with generalized eigenspace decomposition $\bar{M} =
    \bar{M}_1 \oplus \ldots \oplus \bar{M}_r$ for $\rhobar(\phi)$ (we thus adopt the notation of
    Definition~\ref{def:standard}); 
    \item for each $i$, $\bar{M}_i$ is stable under $\bar{\rho}(\sigma)$;
    \item for each $i$, the minimal polynomial of $\rhobar(\sigma)$ acting on $\bar{M}_i$ is $(X-1)^{n_i}$.
  \end{itemize}
\end{definition}

\begin{lemma}\label{lem:distinguished-diagonal} Suppose that $(\bar{M},\bar{\rho},(\bar{e}_i)_i)$ is
  distinguished and that $(M,\rho,(e_i)_i)$ is a standard lift to some $A \in \Cc_\Oc$.  Let $M = M_1
  \oplus \ldots \oplus M_r$ be the decomposition of Definition~\ref{def:standard}.  Then
  $\rho(\sigma)$ preserves each $M_i$.
\end{lemma}
\begin{proof} Let $\Sigma = \rho(\sigma) \in \End(M)$ and let $\Phi = \rho(\phi) \in \End(M)$.  Let
  $\Phi_i$ be the image of $\Phi$ in $\End(M_i)$; then by assumption $\Phi = \bigoplus_{i=1}^r
  \Phi_i$.  Let $\Sigma_{ij}$ be the image of $\Sigma$ in $\Hom(M_i,M_j)$; we must show that
  $\Sigma_{ij} = 0$ for $i \neq j$.

  Let $I$ be the ideal of $A$ generated by the matrix entries of $\Sigma_{ij}$ (with respect to the
  basis $(e_i)$) for $i \neq j$.  We will
  show that $I = \mf_A I$ and hence, by Nakayama, that $I = 0$, as required.  Write 
  \begin{align*}
    \Sigma^q &= (1 + (\Sigma - 1))^q \\
    &= 1 + q(\Sigma - 1) + \sum_{s \geq 2}\binom{q}{s} (\Sigma - 1)^s.
  \end{align*}
  As $q - 1 \in \mf_A$, $\binom{q}{s} \in \mf_A$ for $2 \leq s \leq n$ (using that $l$ is
  quasi-banal), and $(\Sigma - 1)^n \equiv 0 \bmod \mf_A$, we see that 
\[ (\Sigma^q)_{ij} \equiv \Sigma_{ij} \mod \mf_A I\]
for $i \neq j$.  From the equation $\Phi \Sigma = \Sigma^q \Phi$ we deduce:
\begin{align*} \Phi_i \Sigma_{ij} &= (\Sigma^q)_{ij}\Phi_j\\
&\equiv \Sigma_{ij} \Phi_j \mod \mf_A I.
\end{align*}
If $P_i$ is the characteristic polynomial of $\Phi_i$, then $P_i(\Phi_j)$ is invertible for $i \neq
j$ (as the reductions mod $\mf_A$ of $P_i$ and $P_j$ are coprime).  But we have
\[0  = P_i(\Phi_i)\Sigma_{ij} \equiv \Sigma_{ij} P_i(\Phi_j) \mod \mf_A I\]
and so $\Sigma_{ij} \equiv 0 \bmod \mf_A I$ as claimed.
\end{proof}

If $\chi$ is a representation of $k_F^\times$ with image in $\mu_{l^a}$, then we regard $\chi$ as an
element of $\Ic_0$ via the canonical surjection $I_F \onto k_F^\times$.  Let $\Ic_1 \subset \Ic$ be
the set of $\Pc : \Ic \rarrow \Part$ supported on such $\chi$; note that $\Ic_1$ can be identified
with the set $\bar{\Ic}_1$ from the last section.  For convenience, we pick an enumeration $\chi_1,
\ldots, \chi_{l^a}$ of the characters $k_F^\times \rarrow \mu_{l^a}$; thus an element of $\Ic_1$ can
be regarded as a sequence $(P_1,\ldots, P_{l^a})$ of partitions.

To compute $R^\square(\rhobar,\tau_\Pc)$ for a distinguished $\rhobar$ and for $\Pc \in \Ic_1$, first note that Lemma
\ref{lem:diagonalization2} allows us to reduce to the case in which there is a single $\bar{M}_i$.
We then have:

\begin{proposition}\label{prop:maximal-unipotent}
  Suppose that $\rhobar(\sigma)$ has minimal polynomial $(X-1)^n$. Let $\Pc \in \Ic$.  If $\Pc
  \not \in \Ic_1$ then $R^\square(\rhobar,\tau_\Pc) = 0$.  If $\Pc \in \Ic_1$ corresponds to a sequence
  $(P_1,\ldots, P_{l^a})$ of partitions, then: 
  \begin{itemize}
  \item $R^\square(\rhobar,\tau_\Pc) = 0$ if any $P_i$ has more
  than one part (i.e. if $P_i(2)>0$ for some $i$);
\item $R^\square(\rhobar,\tau_{\Pc})$ is formally smooth of relative dimension $n^2$ over $\Oc$ if
  each $P_i$ has only one part.
\end{itemize}
  The special fibre $R^\square(\rhobar)\otimes \FF$ has a single minimal prime.
\end{proposition}
\begin{proof}
  If $\rho : G_F \rarrow GL_n(\Oc')$ has reduction isomorphic to $\rhobar$, with $\Oc'$ the ring of
  integers in a finite extension $E'/E$, then the minimal polynomial $f(X)$ of $\rho(\sigma)$ is
  congruent to $(X-1)^n$ modulo the maximal ideal of $\Oc'$.  Moreover, its roots are $(q^d-1)$th
  roots of unity for some $d \leq n$.  As $l$ is quasi-banal, it follows that the roots are $l^a$th
  roots of unity.  We deduce that $R^\square(\rhobar,\tau_{\Pc}) = 0$ if $\Pc \not \in \Ic_1$.

  Suppose that $\Pc \in \Ic_1$ corresponds to $(P_1, \ldots, \Pc_{l^a})$.  If some $P_i$ has more
  than one part, then the minimal polynomial of $\rho(\sigma)$ in a lift of $\rhobar$ of type
  $\tau_\Pc$ would have degree $< n$.  Therefore there are no such lifts and
  $R^\square(\rhobar,\tau_{\Pc}) = 0$ in this case.

  Suppose now that each $P_i$ has exactly one part, so that $P_i(1) = n_i$.  Let $R$ be the quotient of
  $R^\square(\rhobar)$ obtained by demanding that the characteristic polynomial of $\rho(\sigma)$
  is \[f_\Pc(X) = \prod_{i}(X - \chi_i(\sigma))^{n_i}.\] Then $R^\square(\rhobar,\tau_\Pc)$ is the
  maximal reduced, $l$-torsion free quotient of $R$ and in fact we will show that $R$ is formally
  smooth over $\Oc$, so that $R = R^\square(\rhobar,\tau_\Pc)$.
  
  Let $(\zeta_1,\ldots,\zeta_n)$ be a tuple of elements of $\Oc^\times$ in which each
  $\chi_i(\sigma)$ appears precisely $n_i$ times, so that $f_\Pc(X) = \prod_{i=1}^n(X-\zeta_i)$.
  
  Choose a basis $\bar{e}_1,\ldots,\bar{e}_n$ for $\bar{M}$ with respect to which the action of $G_F$ is
  given by $\rhobar$.  Conjugating $\rhobar$ if necessary, we may assume
  that \[\bar{e_i} = (\rhobar(\sigma) - 1)^{i-1}\bar{e}_1\] for $1 \leq i \leq n$.
  
  Let $(M,\rho,(e_i)_i)$ be a lift of $(\bar{M},\rhobar,(\bar{e}_i)_i)$ to some $A \in \Cc_\Oc$ and
  suppose that the characteristic polynomial of $\rho(\sigma)$ is $f_\Pc(X)$ (regarded as an element
  of $A[X]$ via the structure map $\Oc \rarrow A$).  Write $\Sigma = \rho(\sigma) \in \End(M)$.
  Define $f_1,\ldots,f_n \in M$ by
  \begin{align*}
    f_1 &= e_1 \\
    f_2 &= (\Sigma - \zeta_1)e_1\\
    f_3 &= (\Sigma - \zeta_1)(\Sigma - \zeta_2)e_1 \\
    \vdots \\
    f_n &= (\Sigma - \zeta_1)(\Sigma - \zeta_2)\ldots(\Sigma - \zeta_{n-1})e_1.
  \end{align*}
  Then $f_1, \ldots, f_n$ is a basis of $M$ in which the matrix of $\Sigma$ is:
  \[ \begin{pmatrix}
\zeta_1 &  0 & 0 & 0 & \hdots \\
1 & \zeta_2 & 0 & 0 & \hdots \\
0 & 1 & \zeta_3 & 0 & \hdots \\
\vdots && \ddots & \ddots & \vdots \\
0 & 0 & \hdots & 1 & \zeta_n
\end{pmatrix}.\]
Let $S$ be the maximal quotient of $R$ on which $\Sigma$ has this form.  Since the formation of the $f_i$
from the $e_i$ is functorial, we have a morphism $S \rarrow R$ in $\Cc^\wedge_\Oc$ that is easily
seen to be formally smooth.  To see that $S$ is formally smooth over $\Oc$, I claim that for every
$m \in M$ there is a unique $\Phi \in \End(M)$ such that $\Phi(f_1) = m$ and $\Phi \Sigma = \Sigma^q
\Phi$.  Indeed, for each $i$ we must have
\begin{align*} \Phi(f_i) &= \Phi(\Sigma - \zeta_1)\ldots(\Sigma - \zeta_{i-1})f_1 \\
&= (\Sigma^q - \zeta_1)\ldots(\Sigma^q - \zeta_{i-1})\Phi(f_1) \\
&= (\Sigma^q - \zeta_1)\ldots(\Sigma^q - \zeta_{i-1})m,
  \end{align*}
  and the endomorphism $\Phi$ defined by this formula works.  Therefore lifting $\rhobar(\phi)$ to
  an automorphism $\rho(\phi)$ of $M$ such that $\rho(\phi)\Sigma \rho(\phi)^{-1} = \Sigma^q$ is the
  same as giving a single element of $M$ lifting $\rhobar(f_1)$, and we see that $S$ is formally
  smooth of dimension $n$ over $\Oc$.  Thus $R$, and hence also $R^\square(\rhobar,\tau_\Pc)$, is
  formally smooth over $\Oc$ as required.
  
  For the statement about the special fibre, simply note that $R \otimes \FF$, as a quotient of
  $R^\square(\rhobar)\otimes \FF$, is independent of the choice of $\Pc$.
\end{proof}

\begin{corollary}\label{cor:distinguished-multiplicity} Suppose that $\rhobar$ is distinguished and
  that the generalised eigenspaces of
  $\rhobar(\phi)$ have dimensions $n_1 \geq n_2 \geq \ldots \geq n_r $.  Let $Q$ be the partition of $n$ with
  $Q(i) = n_i$ for $1 \leq i \leq r$.  Let $\Pc \in \Ic$.  If $\Pc \not \in \Ic_1$ then $R^\square(\rhobar,\tau_\Pc) = 0$.
  Otherwise, suppose that $\Pc$ corresponds to the sequence of partitions $(P_i)_i$, and suppose
  (without loss of generality) that $\deg P_1 \geq \deg P_2 \geq \ldots$.  Let $P$ be the partition
  of $n$ with $P(i) = \deg P_i$.  Then $R^\square(\rhobar)\otimes \FF$ has a unique minimal prime
  $\pf$ and
  \[Z(R^\square(\rhobar,\tau_\Pc)\otimes \FF) = \Bip\left((P_i)_i,Q\right)\cdot [\pf].\]
\end{corollary}
\begin{proof} We combine Lemmas \ref{lem:diagonalization2}, \ref{lem:distinguished-diagonal} and
  Proposition~\ref{prop:maximal-unipotent}.  Let $\bar{M}_j$ ($1 \leq j \leq r$) be the generalised
  eigenspaces of $\rhobar(\phi)$ on $\bar{M}$ and let $\bar{\rho}_j$ be the representation of $G_F$
  on $\bar{M}_j$ for each $j$.  Then we have that $R^\square(\rhobar)$ is formally smooth over
  $\widehat{\bigotimes}_j R^\square(\rhobar_j)$, by Lemma~\ref{lem:diagonalization2}.  That
  $R^\square(\tau_\Pc)$ is zero if $\Pc \not \in \Ic_1$ is now clear.

  If $\Pc \in \Ic_1$ corresponds to the sequence $(P_i)_i$, then the irreducible components of
  $\Spec R^\square(\rhobar,\tau_{\Pc})$ are all formally smooth with the same special fibre.  The
  number of such irreducible components is the number of sequences $(\Pc_1, \Pc_2,\ldots, \Pc_r)$
  where:
  \begin{itemize}
  \item for $j = 1, \ldots, r$, $\Pc_j \in \Ic_1$ has degree $n_j$;
  \item each $\Pc_j(\chi_i)$ consists of a single part (that is, $\Pc_j(\chi_i)(1) = d_{ij}$ for
    some non-negative integer $d_{ij}$, and $\Pc_j(\chi_i)(2) = 0$);
  \item for each $i$, the sequence $(d_{i1}, d_{i2}, \ldots, d_{ir})$ is a
    reordering of $P_i(1),\ldots, P_i(r)$.
  \end{itemize}
  Indeed, such a sequence gives rise to the irreducible component
  \[\Spec\widehat{\bigotimes_j} R^\square(\rhobar_j, \tau_{\Pc_j})\]
  of 
\[\Spec\widehat{\bigotimes_j} R^\square(\rhobar_j),\]
and hence of $\Spec R^\square(\rhobar)$, that has type $\tau_\Pc$, and all irreducible components
have this form.

But now $(d_{ij})_{i,j}$ is a $(P,Q)$-bipartition of type $(P_i)_i$ and we see that the number of
irreducible components of $R^\square(\rhobar,\tau_\Pc)$ is the number of $(P,Q)$-bipartitions of
type $(P_i)_i$.  Since all the irreducible components are formally smooth with the same special
fibre, we get the claimed formula.
\end{proof}

Let $\bar{\Xf}_1$ be the closed subscheme of $\bar{\Xf}$ on which $\Sigma$ is unipotent.  Let $\Xf_1$
be the connected component of $\Xf$ containing $\bar{\Xf}$.

\begin{lemma} \label{lem:enough-distinguished} Every irreducible component of $\bar{\Xf}_1$ contains
  a point $x\in \bar{\Xf}(\FF)$ such that $\rhobar_x$ is distinguished (possibly after enlarging
  $\FF$).  If $x \in \bar{\Xf}(\FF)$ is such that $\rhobar_x$ is distinguished, then $x$ lies on a
  unique irreducible component of $\bar{\Xf}$.
\end{lemma}
\begin{proof}
  The irreducible components of $\bar{\Xf}$ are precisely the closures of the
  preimages under $\pi_\Sigma$ of conjugacy classes of $\Sigma$ in $GL_n(\bar{\FF})$.\footnote{See
    the proof of Theorem~\ref{thm:complete-intersection}.}  If $\Sigma$ is
  unipotent, then (using that $l$ is quasi-banal and so $\binom{q}{i} = 0 \mod l$ for $2 \leq i \leq
  n$):
\begin{align*}\Sigma^q &= (1 + (\Sigma - 1))^q \\
  &= 1 + q (\Sigma - 1) + \sum_{i = 2}^n \binom{q}{i}(\Sigma - 1)^i \\
  &= 1 + (\Sigma - 1) = \Sigma.
\end{align*}
Thus (for unipotent $\Sigma$) the equation $\Phi \Sigma \Phi^{-1} = \Sigma^q$ is equivalent to
$\Phi$ commuting with $\Sigma$.  But then for each unipotent $\Sigma \in GL_n(\bar{\FF})$ it is
straightforward (using Jordan normal form) to choose a $\Phi\in GL_n(\bar{\FF})$ commuting with
$\Sigma$ such that the representation $\rhobar_x$ attached to the point $x = (\Phi, \Sigma)$ of
$\bar{\Xf}$ is distinguished; possibly enlarging $\FF$, we can assume that $\Phi \in GL_n(\FF)$.

The second assertion follows from the last part of Proposition~\ref{prop:maximal-unipotent}.
\end{proof}

\subsection{Comparison of multiplicities}
\label{sec:comp-mult}

Continue to assume that $l$ is quasi-banal.  Recall that (given a choice of generator $\sigma$ of
tame inertia) we have defined $\cyc_{\ft} : R^1_E(GL_n(k_F)) \rarrow \Zc(\Xf)$.
\begin{theorem}\label{thm:quasibanal-BM}
  There is a unique map $\bar{\cyc}_{\ft} : R^1_{\FF}(GL_n(k_F)) \rarrow \Zc(\bar{\Xf})$ such that the diagram 
\begin{equation}
\begin{CD}
R^1_E(GL_n(k_F)) @>{\cyc_\ft}>> \Zc(\Xf) \\
@V{\red}VV  @V{\red}VV \\
R^1_{\FF}(GL_n(k_F)) @>{\bar{\cyc}_\ft}>> \Zc(\bar{\Xf})
\end{CD}
\end{equation}
commutes.
\end{theorem}
\begin{proof}As explained in section~\ref{sec:reduct-finite-type}, this implies a similar statement
  with $\Xf$ replaced by $\Spec R^\square(\rhobar)$ for any continuous $\rhobar: G_F \rarrow
  GL_n(k_F)$ such that $\rhobar|_{\tilde{P}_F}$ is trivial.  Moreover, by Proposition~\ref{prop:ft-reduction} and Lemma~\ref{lem:enough-distinguished}, it suffices to prove that, for
  $\rhobar$ distinguished, there is a map $\bar{\cyc} : R^1_{\FF}(GL_n(k_F)) \rarrow
  \Zc(\Rbarbox(\rhobar))$ such that
\begin{equation}\label{eq:BM-diag}
\begin{CD}
R^1_E(GL_n(k_F)) @>{\cyc}>> \Zc(R^\square(\rhobar)) \\
@V{\red}VV  @V{\red}VV \\
R^1_{\FF}(GL_n(k_F)) @>{\bar{\cyc}}>> \Zc(\Rbarbox(\rhobar))
\end{CD}
\end{equation}
commutes.  (Although we work with the whole $R_E(GL_n(\Oc_F))$ in section~\ref{sec:reduct-finite-type}, the arguments apply just as well with $R^1_E(GL_n(k_F))$, using that
$\red:R^1_E(GL_n(k_F)) \rarrow R^1_{\FF}(GL_n(k_F))$ is surjective in the quasi-banal case.) 

So suppose that $\rhobar$ is distinguished, and that the generalized eigenspaces of $\rhobar(\phi)$
have dimensions $n_1 \geq n_2 \geq \ldots \geq n_r$, giving a partition $Q = (n_1, n_2,\ldots, n_r)$
of $n$.  First we make the definition of the cycle map explicit. If $\bar{\Pc} \in \bar{\Ic}_1$, then
we have (identifying an element $\Pc' \in \Ic_1$ with an element $\bar{\Pc}\,'$ of $\bar{\Ic}_1$):
\[\cyc : \sigma_{\bar{\Pc}} \mapsto \sum_{\Pc' \in \Ic_1} \dim
\Hom(\sigma_{\bar{\Pc}},\pi_{\bar{\Pc}\,'})Z(R^\square(\rhobar, \tau_{\Pc'}^\vee)).\] Note that
$R^\square(\rhobar,\tau_{\Pc'}) = 0$ for $\Pc' \not \in \Ic_1$, and that we have (for convenience)
rearranged the position of the dual occurring in Definition~\ref{def:cyc}.  If $\pf$ is the unique
minimal prime of $R^\square(\rhobar)\otimes \FF$, then we find (by
Corollary~\ref{cor:distinguished-multiplicity}) that
\[\red \circ \cyc : \sigma_{\bar{\Pc}} \mapsto [\pf]\cdot\sum_{\Pc'}\dim
\Hom(\sigma_{\bar{\Pc}},\pi_{\bar{\Pc}\,'})\Bip\left((\Pc'(\chi_i))_i,Q\right).\] Now, if $\bar{\Pc}$ takes $\chi_1$
(the trivial representation) to the partition $P$ of $n$, then we see that
\[\red \circ \cyc : \sigma_{\bar{\Pc}} \mapsto [\pf] \cdot \dim\Hom(\sigma_P^\circ, \pi_Q^\circ)\]
and so we must have 
\[\bar{\cyc}(\red(\sigma_{\bar{\Pc}})) = [\pf] \cdot \dim\Hom(\sigma_P^\circ, \pi_Q^\circ).\]
By Lemma~\ref{lem:quasiunipotent-reduction}, the $\red(\sigma_{\bar{\Pc}})$ for $\Pc$ supported on
the trivial representation are all irreducible, and are a basis for $R^1_{\FF}(GL_n(k_F))$; there is
therefore a unique map $\bar{\cyc}$ defined by the above equation for such $\Pc$, and we must show that it makes
diagram~\eqref{eq:BM-diag} commute.  Using Lemma~\ref{lem:quasiunipotent-reduction} and
the subsequent equation~\eqref{eq:linear-to-symmetric}, we see that it suffices to show the
following.  If $\Pc \in \bar{\Ic}_1$ has $\Pc(\chi_i) = P_i$, and $P$ is the partition corresponding
to $(\deg P_1, \deg P_2, \ldots)$, then
\[ \sum_{P'}\dim \Hom_{S_n}(\sigma_{P'}^\circ, \pi_Q^\circ)\dim \Hom_{S_n}\left(\sigma_{P'}^\circ,\Ind_{S_P}^{S_n}
\left(\bigotimes\sigma_{P_i}^\circ\right)\right)\] is equal to \[\sum_{\Pc'}\dim
\Hom(\sigma_{\bar{\Pc}},\pi_{\bar{\Pc}\,'})\Bip\left((\Pc'(\chi_i))_i,Q\right),\] where the first sum is over
partitions $P'$ of $n$ and the second is over $\Pc' \in \Ic_1$.  Indeed, the first displayed
equation is the value of $\bar{\cyc}(\red(\sigma_{\bar{\Pc}}))$, and the second is the value of
$\red(\cyc(\sigma_{\bar{\Pc}}))$.

But
\begin{align*}
  &\sum_{P'}\dim \Hom(\sigma_{P'}^\circ, \pi_Q^\circ)\dim \Hom\left(\sigma_{P'}^\circ,\Ind_{S_P}^{S_n}
\left(\bigotimes\sigma_{P_i}^\circ\right)\right)\\
  &= \dim \Hom\left(\pi_Q^\circ, \Ind_{S_P}^{S_n}\left(\bigotimes\sigma_{P_i}^\circ\right)\right) \\
  &=  \sum_{(P'_i)_i}\Bip\left((P'_i)_i,Q\right)\dim \Hom_{S_P}\left(\bigotimes \pi_{P'_i}^\circ, \bigotimes \sigma_{P_i}^\circ\right)\\
  &= \sum_{\Pc'}\Bip\left((\Pc'(\chi_i))_i,Q\right) \dim \Hom(\sigma_{\bar{\Pc}}, \pi_{\bar{\Pc}\,'})
\end{align*}
as required.  The sum on the third line is over sequences of partitions $(P'_i)_i$ with $\deg P'_i = \deg P_i$, and to
go from the second to the third line we have used Proposition~\ref{prop:mackey}.  The sum on the
fourth line is over $\Pc' \in \Ic_1$ and to go from the third to the fourth line we have used
Corollary~\ref{cor:finite-gl-mults}.
\end{proof}

\bibliography{../references.bib}{}
\bibliographystyle{amsalpha}
\end{document}

%% file: header.tex
\usepackage{amsmath}
\usepackage{amssymb}
\usepackage{amsthm}
\usepackage{amscd}
\usepackage{mathtools}
\usepackage{marginnote}
\usepackage{bbm}
\usepackage{bm}
\usepackage[notcite,notref]{showkeys}
\usepackage{cite}
\usepackage[colorlinks=true,linkcolor=black,citecolor=black,urlcolor=blue]{hyperref} 

\newtheorem{theorem}{Theorem}[section]
\newtheorem*{theorem*}{Theorem}
\newtheorem{conjecture}[theorem]{Conjecture}
\newtheorem*{conjecture*}{Conjecture}

\newtheorem*{question*}{Question}

\newtheorem{lemma}[theorem]{Lemma}
\newtheorem*{lemma*}{Lemma}

\newtheorem{proposition}[theorem]{Proposition}
\newtheorem*{proposition*}{Proposition}

\newtheorem{corollary}[theorem]{Corollary}
\newtheorem*{corollary*}{Corollary}

\theoremstyle{definition}

\newtheorem{definition}[theorem]{Definition}
\newtheorem*{definition*}{Definition}
\newtheorem{remark}[theorem]{Remark}

\newtheorem{example}[theorem]{Example}

\newtheorem*{example*}{Example}
\newtheorem*{examples*}{Examples}

\newcommand{\rom}[1]{\mathrm{#1}}
\newcommand{\twomat}[4]{\begin{pmatrix} #1 & #2 \\ #3 & #4 \end{pmatrix}}

\renewcommand{\bar}{\overline}
\usepackage{mathtools}

\renewcommand{\AA}{\mathbb{A}}

\newcommand{\CC}{\mathbb{C}}

\newcommand{\FF}{\mathbb{F}}

\newcommand{\NN}{\mathbb{N}}

\newcommand{\QQ}{\mathbb{Q}}
\newcommand{\RR}{\mathbb{R}}

\newcommand{\TT}{\mathbb{T}}

\newcommand{\ZZ}{\mathbb{Z}}

\newcommand{\Ac}{\mathcal{A}}

\newcommand{\Cc}{\mathcal{C}}

\newcommand{\Fc}{\mathcal{F}}
\newcommand{\Gc}{\mathcal{G}}
\newcommand{\Hc}{\mathcal{H}}
\newcommand{\Ic}{\mathcal{I}}

\newcommand{\Lc}{\mathcal{L}}
\renewcommand{\Mc}{\mathcal{M}}

\newcommand{\Oc}{\mathcal{O}}
\newcommand{\Pc}{\mathcal{P}}

\newcommand{\Rc}{\mathcal{R}}
\newcommand{\Sc}{\mathcal{S}}
\newcommand{\Tc}{\mathcal{T}}

\newcommand{\Zc}{\mathcal{Z}}

\newcommand{\Af}{\mathfrak{A}}
\newcommand{\af}{\mathfrak{a}}
\newcommand{\Bf}{\mathfrak{B}}
\newcommand{\bfrak}{\mathfrak{b}}

\newcommand{\cf}{\mathfrak{c}}

\newcommand{\df}{\mathfrak{d}}

\newcommand{\Lf}{\mathfrak{L}}

\newcommand{\mf}{\mathfrak{m}}

\newcommand{\Pf}{\mathfrak{P}}
\newcommand{\pf}{\mathfrak{p}}

\newcommand{\qf}{\mathfrak{q}}

\newcommand{\Sf}{\mathfrak{S}}

\newcommand{\Xf}{\mathfrak{X}}

\newcommand{\Zf}{\mathfrak{Z}}

\newcommand{\rarrow}{\rightarrow}

\newcommand{\onto}{\twoheadrightarrow}
\newcommand{\into}{\hookrightarrow}
\newcommand{\isomto}{\xrightarrow{\sim}}

\newcommand{\tensor}{\otimes}
\newcommand{\normal}{\lhd}
\newcommand{\defeq}{\coloneqq}
\newcommand{\rhobar}{\bar{\rho}}

\newcommand{\supp}{\operatorname{supp}}
\newcommand{\depth}{\operatorname{depth}}
\newcommand{\Aut}{\operatorname{Aut}}
\newcommand{\End}{\operatorname{End}}
\newcommand{\Hom}{\operatorname{Hom}}

\newcommand{\Frob}{\operatorname{Frob}}
\newcommand{\rec}{\operatorname{rec}}
\newcommand{\Gal}{\operatorname{Gal}}
\newcommand{\Res}{\operatorname{Res}}
\newcommand{\Ind}{\operatorname{Ind}}
\newcommand{\cInd}{\operatorname{c-Ind}}
\newcommand{\Spec}{\operatorname{Spec}}
\newcommand{\Spf}{\operatorname{Spf}}
\newcommand{\ord}{\operatorname{ord}}

\newcommand{\tr}{\operatorname{tr}}
\newcommand{\WD}{\operatorname{WD}}
\newcommand{\Nm}{\operatorname{Nm}}

\newcommand{\Sp}{\operatorname{Sp}}

\newcommand{\ch}{\operatorname{char}}

\newcommand{\St}{\operatorname{St}}
\newcommand{\Rep}{\operatorname{Rep}}

\newcommand{\crord}{\operatorname{cr-ord}}
\newcommand{\univ}{\operatorname{univ}}

\newcommand{\ad}{\operatorname{ad}}